\documentclass[11pt,reqno]{amsart}

\usepackage{amsmath,amsthm,amssymb}
\usepackage{epsf}
\usepackage{amssymb}
\usepackage{graphics}
\usepackage{hyperref}

\newtheorem{thm}{Theorem}[section]

\newtheorem{cor}[thm]{Corollary}
\newtheorem{lem}[thm]{Lemma}
\newtheorem{prop}[thm]{Proposition}

\newtheorem{rmk}[thm]{Remark}

\numberwithin{equation}{section}

\newcommand\be{\begin{equation}}
\newcommand\ee{\end{equation}}

\newcommand{\bsl}{\backslash}

\newcommand{\bH}{\mathbb{H}}
\newcommand{\R}{\mathbb{R}}
\newcommand{\C}{\mathbb{C}}
\newcommand{\Z}{\mathbb{Z}}
\newcommand{\Q}{\mathbb{Q}}
\newcommand{\N}{\mathbb{N}}

\newcommand{\cO}{\mathcal{O}}
\newcommand{\cM}{\mathcal{M}}
\newcommand{\cL}{\mathcal{L}}
\newcommand{\cI}{\mathcal{I}}

\newcommand{\cS}{\mathcal{S}}
\newcommand{\cR}{\mathcal{R}}

\newcommand{\tE}{\tilde{E}}
\newcommand{\tM}{\tilde{M}}

\newcommand{\cuspa}{\mathfrak{a}} 

\newcommand{\m}{\mathfrak{m}}

\newcommand{\sgn}{{\rm sign}} 

\newcommand{\sm}{\left(\begin{smallmatrix}}
\newcommand{\esm}{\end{smallmatrix}\right)}
\newcommand{\bpm}{\begin{pmatrix}}
\newcommand{\ebpm}{\end{pmatrix}}

\newcommand{\cusp}{{\rm cusp}}
\newcommand{\cont}{{\rm cont}}

\newcommand{\Res}{{\rm Res}}

\newcommand{\Int}{{\rm Int}}

\newcommand{\fin}{{\rm finite}}

\newcommand{\fund}{\mathfrak{F}}

\newcommand{\bnu}{\bar{\nu}}

\newcommand{\tth}{\tilde{h}}

\begin{document}
\title{Shifted multiple Dirichlet series}

\author{Jeff Hoffstein} \email{jhoff@math.brown.edu}
\address{Department of Mathematics,
Brown University, Providence, RI, $02912$}

\author{Min Lee}
 \email{minlee@math.brown.edu}
 \address{Department of Mathematics,
Brown University, Providence, RI, $02912$}

\date{\today}

\maketitle
\begin{abstract}
We develop certain aspects of the theory of shifted multiple Dirichlet series and study their meromorphic continuations. 
These continuations are used to obtain explicit spectral first and second moments of Rankin-Selberg convolutions.  
One consequence is a Weyl type estimate for the Rankin-Selberg convolution of a holomorphic cusp form and a Maass form with spectral parameter $|t_j|\le T$, namely:
$$
	\left| L\left(\frac{1}{2}+ir, f\times u_j\right) \right|
	\ll_N T^{2/3+\epsilon},
$$
uniformly, for $|r| \le T^{2/3}$, with the implied constant depending only on $f$ and the level $N$.
\end{abstract}

\tableofcontents

\section{Introduction}
The object of this paper is to present a new technique for approaching moment formulas, and to give several examples of applications.   The foundation of the results of this paper is  a spectral identity which may be interesting in its own right.   This identity is related to some results of  \cite{Mot92} and \cite{Ivi02}, and also to \cite{JM05}, but comes from an entirely different approach.    They begin with the Kuznetsov trace formula, and follow this with a use of a Voronoi formula.  Interestingly, Motohashi-Jutila speculate at the end of their paper on the possibility of extending their results to Rankin-Selberg convolutions without recourse to Kuznetsov and Voronoi.   

Our method uses neither Kuznetsov nor Voronoi, but depends instead on analytic properties of shifted multiple Dirichlet series developed in \cite{HHR} and \cite{Hul13}.    In particular,  an application of an integral transform to a spectral decomposition of a single shifted multiple Dirichlet series leads to a certain first moment identity, presented below in Theorem~\ref{thm:first_Eisenstein}.    Summing over a parameter leads to a shifted Dirichlet series in two variables, and the analytic properties of this series lead to our second moment results.  

One of the usual characteristics of this sort of work is the careful analysis of integral transforms of Dirichlet series that are turned into essentially finite sums by means of the approximate functional equation.   These sums are generally broken into a number of different pieces and different techniques are used to estimate the contributions of the integral transforms of these pieces.   We would argue that although it certainly has its share of complicated formulas, our approach is conceptually much simpler.   Dirichlet series are dealt with as complete objects, and are used to build shifted multiple Dirichlet series in several variables.   The meromorphic properties of these series are then used to directly extract information, using standard techniques from analytic number theory.

We will begin our discussion of the results in this paper with a second moment formula for Rankin-Selberg convolutions of cusp forms and Eisenstein series, and some historical background.

Let $f$ and $g$ be holomorphic cusp forms of even weight $k$ and level $N$.   
Assume that $f$ and $g$ are newforms. 
We will in general take $N$ to be square free, in order to simplify the Fourier expansions of Eisenstein series of level $N$.  
We will write the Fourier expansions of $f$ and $g$ as
$$
	f(z) = \sum_{n=1}^\infty a(n) e^{2\pi inz} \quad \text{and} \quad g(z) = \sum_{n=1}^\infty b(n) e^{2\pi inz}
$$
and also use the notation
$$
a(n) = A(n)n^{\frac{k-1}{2}}    \quad \text{and} \quad  b(n) = B(n)n^{\frac{k-1}{2}}.
$$

Let $\Gamma:= \Gamma_0(N) < SL_2(\Z)$ and $\bH=\{x+iy\;|\; x\in \R, y>0\}$ be the Poincar\'e upper half plane. 
For any $\varphi_1, \varphi_2\in L^2(\Gamma\bsl \bH)$, we define the Petersson inner product 
$$
	\left< \varphi_1, \varphi_2\right>
	:=
	\iint_{\Gamma\bsl \bH} \varphi_1(z) \overline{\varphi_2(z)} 
	\; \frac{dx\; dy}{y^2}. 
$$

Let $\{u_j\}_{j\geq 0}$
be an orthonormal basis (with respect to the Petersson inner product) 
for the discrete part of the spectrum of the Laplace operator on $L^2\left(\Gamma\bsl \bH\right)$.   We will assume the basis is simultaneously diagonalized with respect to Hecke operators corresponding to primes not dividing $N$.
Let $s_j(1-s_j)$ for $s_j\in \C$ be an eigenvalue of the Laplace operator for $u_j$ for each $j\geq 0$.   Then we have the Fourier expansion
\begin{equation}\label{e:normalized_maassform}
	u_j(z) 
	= 
	\sum_{n\neq 0} \rho_j(n)\sqrt{y}K_{it_j}(2\pi|n|y)e^{2\pi inx}\,.
\end{equation}
Here $s_j=\frac{1}{2}+it_j$ for $t_j\in \R$ or $\frac{1}{2}< s_j< 1$. 
Also 
$$
	K_{\nu}(y)
	=
	\frac{1}{2} \int_0^\infty e^{-\frac{1}{2} y \left(t+t^{-1}\right)} t^{\nu} \; \frac{dt}{t},  
$$
for $\nu\in \C$ and $y>0$, is the $K$-Bessel function. 

For each $a\mid N$, let $E_{1/a}(z, 1/2+it)$ be the Eisenstein series at the cusp $1/a$ as defined in \eqref{e:Ea}. 
	
The Rankin-Selberg convolutions of the cusp form $f$ with $u_j$ is written as
$$
	\cL\left(s, f\times u_j\right)
	=
	\zeta(2s)
	\sum_{n=1}^\infty \frac{A(n) \rho_j(n)}{n^{s}}
$$
and
$$
	L\left(s, f\times E_{1/a}(*, 1/2+it)\right)
	=
	\zeta(2s) \sum_{n=1}^\infty \frac{A(n)\sigma_{-2it}^a(n) n^{it}}{n^s}.
$$
Here $\sigma_{-2it}^a(n)$ is given in \eqref{e:sigmaa}. 

If $u_j$ is a new form then $\rho_j(n)= \rho_j(1)\lambda_j(n)$, where $\lambda_j(n)$ is the $n^{th}$ Hecke eigenvalue.  In this case 
$$
	\cL\left(s, f\times u_j\right)
	=
	 \rho_j(1)\zeta(2s)
	\sum_{n=1}^\infty \frac{A(n) \lambda_j(n)}{n^{s}} = \rho_j(1) L\left(s, f\times u_j\right),
$$
where $L\left(s, f\times u_j\right)$ is the usual normalized Rankin-Selberg convolution.

The first sub convexity result for Rankin-Selberg $L$-functions was  
obtained by Sarnak \cite{Sar01} in the weight aspect.   He showed that
 for $f$ a holomorphic cusp form of weight $k$ for $\Gamma_0(N)$, and $g$ 
 a fixed either  holomorphic cusp form or Maass form for $\Gamma_0(N)$, 
$$
L\left(1/2 + it, f \times g \right) \ll_{t,g,\epsilon,N} k^{576/601 + \epsilon}.
$$
Here 1 represents the convexity bound and the exponent $576/601$ was achieved with the input of the $\theta=7/64$ bound for the Ramanujan conjecture.
There were several improvements in the next few years, culminating in  2006 in a result of  Lau, Liu and Ye in \cite{LLY06}.  They showed that for any small $\epsilon >0$
$$
\sum_{K-L\le k \le K+L}|L\left(1/2 + it, f \times g \right)|^2\ll_{t,g,\epsilon}\left(N^{A(\epsilon)}KL\right)^{1+\epsilon}.
$$
They also obtained the consequent sub convexity bound
$$
L\left(1/2 + it, f \times g \right) \ll_{t,g,\epsilon} k^{2/3 + \epsilon}.
$$
In both results the constant $A(\epsilon)>0$ depends only on $\epsilon$, and the implied constant depends polynomially on $t$.
Also $f$ was a holomorphic Hecke eigenform for $\Gamma_0(N)$ of weight $k$, or a Maass Hecke eigenform for $\Gamma_0(N)$ with Laplace eigenvalue $1/4 + k^2$, and $g$ was a fixed holomorphic or Maass cusp form for $\Gamma_0(N)$, or for $\Gamma_0(N')$ with $(N,N')= 1$. 

Our results are the first we are aware of giving asymptotic estimates for the second moment of Rankin-Selberg $L$-functions with sharp error terms, though in \cite{Ivi02}, Ivic gives a fourth moment estimate for the value of Hecke $L$-functions at $1/2$, with an error term of $K^{3/2}(\log K)^{25/2}$.  In our work, the parameter $K$ is replaced by $T$ and refers only to the special parameter, not the weight.   Our second moment result has the form
\begin{thm}\label{thm:second_intro}
Fix $T \gg 1$,  $1 \ge \alpha \ge 1/3$,  $0 \le r \le T^{2/3}$ and $\epsilon >0$.  Then 
\begin{multline}\label{e:IfEg}
	\sum_{j} 
	\frac{e^{-\frac{(t_j-T)^2}{T^{2\alpha}}}}
	{\cosh(\pi t_j)}
	\cL\left(1/2-ir, \overline{u_j}\times \overline{g}\right)
	\cL(1/2+ir, f\times u_j)
	\\
	+	
	\sum_{\cuspa} 
	\frac{1}{4\pi} \int_{-\infty}^\infty 
	\frac{e^{-\frac{(t-T)^2}{T^{2\alpha}}}}
	{\cosh(\pi t)}
	\frac{4\left(\frac{N}{a}\right)^{-1}
	L\left(1/2-ir, E_{1/a}(*, 1/2-it) \times \overline{g}\right)
	L\left(1/2+ir, f\times E_{1/a}(*, 1/2+it) \right)}
	{\zeta^*(1-2it)\zeta^*(1+2it)\prod_{p\mid N} (1-p^{-1+2it})(1-p^{-1-2it})}
	\; dt
	\\
	=
	M_{f, g}(T, \alpha; r) 
	+
	 \cO_N\left(T^{3/2-\alpha/2 + \epsilon}+ T^{\alpha+\epsilon}|r|^{3/2 +\epsilon}\right).
\end{multline}
The implied constant is polynomial in $N$.
Exact formulas for $M_{f, g}(T, \alpha; r)$ are given explicitly in \eqref{e:Nrne}, \eqref{e:1re} and \eqref{e:1re0}, and asymptotically

$$
	M_{f, g}(T, \alpha; 0)
	=
	T^{1+\alpha} P_d(\log T) + \cO(T^{1+\alpha-\epsilon}), 
$$
where $P_d(\log T)$ is a degree $d$ polynomial of $\log T$. 
For $f=g$, $d=3$ and for $f\neq g$, $d=2$. 
The leading coefficient of $P_d(\log T)$ is 
$$ 
	c
	\Res_{s=1} L(s, f\times\bar{f})
$$
for $f=g$, 
and 
$$
	c'
	L(1, f\times g) 
$$
for $f\neq g$. 
Here $c,c'$ are explicit positive constants which are independent of $f$ and $g$. 

\end{thm}
Taking $\alpha = 1/3$ and bounding $M_{f, g}(T, \alpha; r)$ from above by $T^{4/3+ \epsilon}$, an immediate consequence of the theorem is
\begin{cor}\label{cor:subconvexity}
$$
	\left| L\left(\frac{1}{2}+ir, f\times u_j\right) \right|
	\ll_N T^{2/3+\epsilon}
$$
for $|r|\le T^{2/3}$, with the implied constant polynomially dependent on $N$ but independent of $r$ and $T$.
\end{cor}
Thus, in addition to obtaining an estimate with a sharp error term for the fourth order Euler product which is the product of two Rankin-Selberg convolutions, we also achieve the same Weyl type estimate as in \cite{LLY06}, with the advantage that the result is uniform in $r$ for $r \le T^{2/3}$.  We are, however restricted to the spectral case and do not consider the weight aspect.

\begin{rmk}
The non-vanishing of $L(1, f \times g)$ implies that the product  $L\left(1/2, f\times u_j\right)L\left(1/2, g\times u_j\right)$ is non-zero for infinitely many $j$.
A more general theorem for the mean values of 
$$
	L\left(\frac{1}{2}+ir, f\times u_j\right) L\left(\frac{1}{2}+ir_1, \phi \times u_j\right)
$$
for $\phi=$ holomorphic or non-holomorphic cusp form is given in Section \ref{s:second}.
\end{rmk}

The second moment estimate is built on our first moment estimate, which
 is given in full generality in Theorem~\ref{thm:first}.
Here is a special case of it:
\begin{thm}\label{thm:first_Eisenstein}
Fix $r\in \R$, and $m \ge 1$.  Let $h(t)$ be an even function on $\C$, which is holomorphic on $|\Im(t)| < 1/2+\epsilon$, for some $\epsilon>0$. 
As $t\to \infty$, we assume that 
$$
	\left|h(t) \right| \ll (1+|t|)^a, 
$$
for $a<0$. 
We also assume that $h(\pm i/2)=0$.

We have 
\begin{multline*}
	\sum_{j} 
	\frac{h(t_j) }{\cosh(\pi t_j)}
	\overline{\rho_j(m)} 
	\cL(1/2+ir, f\times u_j)
	\\
	+
	\sum_{a\mid N} 
	\frac{1}{4\pi} \int_{-\infty}^\infty 
	\frac{h(t)}{\cosh(\pi t)}
	\tau_{1/a} (1/2-it, m) 
	\frac{2 \left(\frac{N}{a}\right)^{-\frac{1}{2}-it}
	L\left(1/2+ir, f\times E_{1/a}(*, 1/2+it) \right)}
	{\zeta^*(1+2it)\prod_{p\mid N} (1-p^{-1-2it})}
	\; dt
	\\
	=
	M_f(m; r) + E^{(1)}_f(m; r) + E^{(2)}_f (m; r)
	, 
\end{multline*}
Here
\begin{multline*}
	M_f(m; r)
	:=
	\zeta(1+2ir)
	\frac{a(m)}{m^{\frac{k}{2}+ir}}
	\frac{1}{\pi^2} \int_{-\infty}^\infty h(t) t \tanh(\pi t) \; dt	
	\\
	+
	(2\pi)^{4ir}
	\frac{\varphi(N)\zeta(1-2ir)}
	{N^{1+2ir} \prod_{p\mid N}(1-p^{-1-2ir})} 
	\frac{a(m)}{m^{\frac{k}{2}-ir}}
	\\
	\times
	\frac{1}{\pi^2} \int_{-\infty}^\infty
	h(t) t \tanh (\pi t)
	\frac{\Gamma\left(\frac{k}{2}-ir+it\right)
	\Gamma\left(\frac{k}{2}-ir-it\right)}
	{\Gamma\left(\frac{k}{2}+ir +it\right)
	\Gamma\left(\frac{k}{2}+ir -it\right) }
	\; dt
	, 
\end{multline*}
\begin{multline*}
	E^{(1)}_f(m; r)
	:=
	-
	4  
	\frac{(2\pi)^{2ir}\cos\left(\pi ir\right)}{\pi \prod_{p\mid N}(1-p^{-1-2ir})}
	\frac{1}{2\pi i} \int_{(\sigma_u = 1/2+2\epsilon)} 
	\frac{h(u/i) u \tan(\pi u)}
	{\Gamma\left(\frac{k}{2}+ir +u\right) \Gamma\left(\frac{k}{2}+ir-u\right)}
	\\
	\times
	\frac{1}{2\pi i } \int_{(\sigma_w = 1/2-k/2-\epsilon)}
	\Gamma\left(-w+u\right) \Gamma\left(-w-u\right)
	\Gamma\left(w+\frac{k}{2}-ir\right) \Gamma\left(w+\frac{k}{2}+ir\right)
	\\
	\times
	m^{w}
	\sum_{n=1}^{m-1} 
	\frac{a(m-n) \sigma^N_{-2ir}(n) n^{ir} }{n^{w+\frac{k}{2}}} 
	\; dw\; du
\end{multline*}
and 
\begin{multline*}
	E^{(2)}_f(m; r)
	:=
	\frac{4 i^k  (2\pi)^{2ir} }{\pi\prod_{p\mid N}(1-p^{-1-2ir})} 
	\frac{1}{2\pi i} \int_{(\sigma_u = 1/2+2\epsilon)} 
	\frac{h(u/i) u 
	\Gamma\left(\frac{1}{2}+u\right) \Gamma\left(\frac{1}{2}-u\right)}
	{\Gamma\left(\frac{k}{2}+ir +u\right) \Gamma\left(\frac{k}{2}+ir-u\right)}
	\\
	\times
	\frac{1}{2\pi i } \int_{(\sigma_w = 1/2+\epsilon)}
	\frac{\Gamma\left(-w+u\right)
	\Gamma\left(w+\frac{k}{2}+ir\right) \Gamma\left(w+\frac{k}{2}-ir\right)} 
	{\Gamma\left(1+w+u\right)}
	m^{w} \sum_{n=1}^\infty \frac{a(m+n) \sigma^N_{-2ir}(n) n^{ir_1}}
	{n^{w+\frac{k}{2}}}
	\; dw\; du
	.
\end{multline*}
The function $ \sigma^N_{-2ir}(n)$ is a modified divisor function defined in \eqref{e:sigmaa}.
\end{thm}

The case $r=0$ in Theorem \ref{thm:first_Eisenstein}, leads to an interesting first moment statement and an application.    For $\alpha \ge 0$ and $T \gg 1$, set
$$
	h_{T, \alpha} (t) = \left(e^{-\left(\frac{t-T}{T^\alpha} \right)^2}+ e^{-\left(\frac{t+T}{T^\alpha} \right)^2}\right)
	\frac{t^2+\frac{1}{4}}{t^2+R}, 
$$
for $1\ll R < T^2$. 
Then, for $m \ge 1$
\begin{thm}\label{thm:upperbound_first}
\begin{multline*}
	\sum_{j} 
	\frac{h_{T, \alpha} (t_j) }{\cosh(\pi t_j)}
	\overline{\rho_j(m)} 
	\cL(1/2, f\times u_j)
	\\
	+
	\sum_{a\mid N} 
	\frac{1}{4\pi} \int_{-\infty}^\infty 
	\frac{h_{T, \alpha} (t)}{\cosh(\pi t)}
	\tau_{1/a} (1/2-it, m) 
	\frac{2 \left(\frac{N}{a}\right)^{-\frac{1}{2}-it}
	L\left(1/2, f\times E_{1/a}(*, 1/2+it) \right)}
	{\zeta^*(1+2it)\prod_{p\mid N} (1-p^{-1-2it})}
	\; dt
	\\
	=
	\frac{a(m)}{m^{\frac{k}{2}}} 
	\left\{
	\frac{1}{\pi^2} \int_{-\infty}^\infty h_{T, \alpha} (t) t\tanh(\pi t)\left(\psi\left(\frac{k}{2}+it\right) + \psi\left(\frac{k}{2}-it\right)\right) \; dt
	\right.
	\\
	+
	\left.
	\left(-2\log(2\pi) - \sum_{p\mid N} \frac{\log p}{p-1}
	-\log m +2\gamma_0\right)
	\frac{1}{\pi^2} \int_{-\infty}^\infty h_{T, \alpha} (t) t \tanh(\pi t) \; dt \right\}
	\\
	+
	\cO\left(N^\epsilon m^{\frac{1}{2}+\epsilon} T^{1+\epsilon}\right)
	.
\end{multline*}
The implied constant in the estimate is independent of $N$, $m$ and $T$. 
\end{thm}
The main term on the right hand side is actually asymptotic to  
$$
	\frac{a(m)}{m^{\frac{k}{2}}}
	T^{1+\alpha} c_\alpha\log T,
$$
for $c_\alpha >0$,
and the left hand side is dominated by the discrete spectral contribution.  This means that if 
 $g$ is a holomorphic cusp form and $\cL(1/2, f\times u_j)=\cL(1/2, g\times u_j)$ for sufficiently many $j$, then $a(m)$ must be very close to $b(m)$.  If this is the case for sufficiently many $m$, then by the argument in the proof of Theorem 1 in \cite{Sen04}, it follows that $f=g$. 
More precisely, we have the following
\begin{cor}\label{cor:determine}
Let $f$ and $g$ be holomorphic newforms of even weight $k_f$ and $k_g$ respectively, and level $N$. 
Let $k=\max\{k_f, k_g\}$, and let $Q\gg_\epsilon (N(k+1))^{2+\epsilon}$. 
If 
\be\label{above}
	\cL(1/2, f\times u_j) = \cL(1/2, g\times u_j), 
\ee
for $|t_j| \ll Q^{1+\epsilon'} N^{1/2+\epsilon''}$, 
then $f=g$. 
\end{cor}
The proof is given in Section \ref{ss:proof_cor_detrmine}.
This is the latest incarnation of a theorem that was proved first by Luo and Ramakrishnan \cite{LR97}.  They showed that if
$$
		L(1/2, f \times \chi_d) = L(1/2, g\times \chi_d) 
$$
	for all quadratic characters $\chi_d$, then $f=g$.  Generalizations of this appeared in 
 \cite{Luo99},  \cite{CD05},  \cite{GHS09}  and \cite{Zha11}.
The most recent result that we are aware of is \cite{MS12}, where they prove a result similar to our corollary, generalized to the case where $f,g$ are Hecke-Maass new forms of full level, with eigenvalues $1/4 +\nu^2,1/4+\mu^2$.   They show that if \eqref{above} holds for 
$|t_j| \ll \max(|\mu|,|\nu|)^{4\theta +3 + \epsilon}$ then $f=g$, where $\theta$ refers to the best progress toward the Ramanujan Conjecture for Mass forms of full level.
We consider any square-free level, and improve the exponent for the analytic conductor.

We will mention one final application to non-vanishing.
A combination of the first and second moment asymptotics, along with Cauchy-Schwartz, taking $\alpha = 1/3$,  leads to
\begin{cor}
Let $f$ and $u_j$ be as above.  Given $T \gg 1$, let $\mathcal{N}(T)$ denote the number of $t_j$ with $|T-t_j| \ll T^{1/3}$ such that $L(1/2,f \times u_j) \ne 0$.  Then
$$
\mathcal{N}(T) \gg T^{4/3}/\log T.
$$

\end{cor}
The proof is elementary and is omitted. 

The paper is organized as follows.   
In Section \ref{s:singleSDS}, we define our single variable shifted Dirichlet series and obtain it's analytic properties.   
In Section \ref{s:first}, we apply an integral transformation and obtain the first moment identity given in Theorem~\ref{thm:first_Eisenstein}, and more generally in Theorem~\ref{thm:first}.  
In Section~\ref{s:SDDS}, we define the two variable shifted Dirichlet series and obtain its analytic properties.  In Section~\ref{s:second} we obtain our second moment identity and prove Theorem~\ref{thm:second}, a generalized version of Theorem~\ref{thm:second_intro}.  Finally, in the last section, we estimate the error terms of Theorem~\ref{thm:second} and prove Theorem~\ref{thm:second_intro}.
\section{Single shifted Dirichlet series}\label{s:singleSDS}

\subsection{Fourier expansions of Eisenstein series}
The Eisenstein series for $\Gamma$ are indexed by the cusps $\cuspa\in \Q$.
For each cusp $\cuspa\in \Q$, let $\sigma_\cuspa\in SL_2(\R)$, 
with $\sigma_\cuspa \infty = \cuspa$, 
be a scaling matrix for the cusp $\cuspa$, i.e., 
$\sigma_\cuspa$ is the unique matrix (up to right translations)
such that 
$\sigma_\cuspa \infty=\cuspa$
and 
$$
	\sigma_\cuspa^{-1}\Gamma_\cuspa \sigma_\cuspa 
	=
	\Gamma_\infty = \left\{\left.
	\pm \bpm 1 & b\\ 0 & 1\ebpm \;\right|\;
	b\in \Z\right\},
$$
where
$$
	\Gamma_\cuspa 
	= 
	\left\{\gamma\in \Gamma\;|\; \gamma \cuspa = \cuspa\right\}
	.
$$

For a cusp $\cuspa$, define the Eisenstein series at the cusp $\cuspa$ to be
\be\label{e:Ea}
	E_\cuspa(z, s)
	:=
	\sum_{\gamma\in \Gamma_\cuspa \bsl \Gamma}
	\Im(\sigma_\cuspa^{-1}\gamma z)^s, 
\ee
with the following Fourier expansion:
\be\label{m:EisensteinSeries_fex}
	E_\cuspa \left(z, s\right) 
	=
	\delta_{\cuspa, \infty} y^{s}
	+
	\tau_\cuspa \left(s, 0\right) y^{1-s}
	+
	\sum_{n\neq 0}
	\tau_\cuspa \left(s, n\right)
	\sqrt y K_{s-\frac{1}{2}}(2\pi |n|y)e^{2\pi inx}.
\ee

When $N$ is square-free, 
the cusps of $\Gamma_0(N)$ are uniquely represented by 
the rationals $\cuspa=1/a$ for $a \mid N$, 
and we have the following explicit description of the Fourier coefficients which we quote from \cite{DI83}:
$$
	\tau_\cuspa \left(s, 0\right)
	=
	\frac{\sqrt\pi \Gamma\left(s-\frac{1}{2}\right)}
	{\Gamma\left(s\right)}
	\rho_\cuspa \left(s, 0\right)
$$
and
$$
	\tau_\cuspa \left(s, n\right)
	=
	\frac{2\pi^{s}} {\Gamma\left(s\right)} 
	\rho_\cuspa \left(s, n\right) |n|^{s-\frac{1}{2}},
$$
with 
\be\label{e:rhocusp}
	\rho_\cuspa \left(s, n\right)
	=
	\left(\frac{1}{aN}\right)^{s}
	\sum_{\gamma\geq 1, \atop \left(\gamma, \frac{N}{a}\right)=1}
	\gamma^{-2s}
	\sum_{\delta\pmod{\gamma a}, (\delta, \gamma a)=1}
	e^{-2\pi i n\frac{\delta}{\gamma a}}
	\,.
\ee
Here we fix the scaling matrix for the cusp $\cuspa=1/a$ as 
\be\label{e:sigma1/a}
	\sigma_{1/a} 
	= \bpm  1 & 0 \\a & 1\ebpm\bpm \sqrt{\m_\cuspa} & 0 \\ 0 & \sqrt{\m_\cuspa}^{-1}\ebpm
	,
\ee
where $\m_\cuspa = N/a$. 

For each $a\mid N$, for $n\neq 0$, define
\be\label{e:sigmaa}
	\sigma^a_{1-2s}(n)
	:=
	\sum_{d\mid n, \atop (d, N)=1} d^{1-2s} 
	\prod_{p\mid a, p^\alpha\|n, \atop \alpha\geq 0} 
	\frac{p^{-2s}}{1-p^{1-2s}}
	\left(p-p^{\alpha(1-2s)+1} -1 +p^{(\alpha+1)(1-2s)}\right)
	.
\ee
Then, for each cusp $\cuspa=1/a$ with $a\mid N$, following \cite{HM06}, 
for $n\neq 0$, we have
\be\label{e:rho_an}
	\rho_{1/a}(s, n)
	=
	\left(\frac{N}{a}\right)^{-s}
	\frac{\sigma^a_{1-2s}(|n|)}
	{\zeta(2s)\prod_{p\mid N} (1-p^{-2s})}
	, 
\ee
so, 
\be\label{e:tau1/a_neq0}
	\tau_{1/a}(s, n)
	=
	\left(\frac{N}{a}\right)^{-s}
	\frac{2\sigma^a_{1-2s}(|n|) |n|^{s-\frac{1}{2}}}
	{\zeta^*(2s) \prod_{p\mid N}(1-p^{-2s})}
	.
\ee
By \cite{DI83}, for $n=0$, we have
\be\label{e:const_eis}
	\rho_{1/a}\left(s, 0\right)
	=
	\frac{\zeta(2s-1)\prod_{p\mid\frac{N}{a}}\left(1-p^{1-2s}\right)}
	{\zeta(2s)\prod_{p\mid N} \left(1-p^{-2s}\right)}
	\varphi(a)
	\left(\frac{1}{aN}\right)^{s}
	.
\ee
We thus get
\begin{multline}\label{e:EisensteinSeries_1/a_Fourier}
	E_{1/a} \left(z, s \right) 
	=
	\delta_{\cuspa, \infty} y^{s}
	+
	\frac{\zeta^*(2s-1)\prod_{p\mid\frac{N}{a}}\left(1-p^{1-2s}\right)}
	{\zeta^*(2s)\prod_{p\mid N} \left(1-p^{-2s}\right)}
	\varphi(a)
	\left(\frac{1}{aN}\right)^{-\frac{1}{2}+2s}
	y^{1-s}
	\\
	+
	\frac{2\left(\frac{N}{a}\right)^{-s}}
	{\zeta^*(2s)\prod_{p\mid N} (1-p^{-2s})}
	\sum_{n\neq 0}
	\sigma^a_{1-2s}(|n|) |n|^{s-\frac{1}{2}}
	\sqrt{y} K_{s-\frac{1}{2}}(2\pi |n|y)e^{2\pi inx}.
\end{multline}

In Lemma 3.4 in \cite{B}, it is shown that for $t\in \R$ the $\rho_\cuspa \left(\frac{1}{2}+it, n\right)$ satisfy
$$
	\left|\zeta(1+it)\right|^2
	\sum_\cuspa 
	\left|\rho_\cuspa \left(\frac{1}{2}+it, n\right)\right|^2
	\ll_\epsilon ((1+|t| )N)^{\epsilon}
$$
for any $\epsilon>0$.

\subsection{Definition and analytic properties of a single shifted Dirichlet series}

Let $f$ be a holomorphic cusp form of even weight $k$ for $\Gamma$.  Recall that we have notated the Fourier expansion of $f$ as: 
\be\label{e:f_Fourier}
	f(z) = \sum_{n=1}^\infty a(n) e^{2\pi in z}
	= \sum_{n=1}^\infty A(n) n^{\frac{k-1}{2}} e^{2\pi in z}
	.
\ee
Let $\phi$ be an automorphic form of weight $k$ for $\Gamma$, of type $\nu$, with the following Fourier expansion:
\be\label{e:Fourier_phi}
	\phi(z) = c_{\phi}^+ y^{\frac{1}{2}+\nu} + c_{\phi}^- y^{\frac{1}{2}-\nu} 
	+ \sum_{n\neq 0} c_{\phi}(n) W_{\sgn(n)\frac{k}{2}, \nu}(4\pi |n|y) e^{2\pi inx}. 
\ee
Assume that 
\be\label{e:cphipm}
	c_\phi(\pm n) = c_\phi(\pm1)C_\phi(n), 
\ee
for any $n\geq 1$.                    

Here 
$$
	W_{\alpha, \nu}(y)
	=
	\frac{y^{\nu+\frac{1}{2}} e^{-\frac{y}{2}} }{\Gamma\left(\nu-\alpha+\frac{1}{2}\right)} 
	\int_0^\infty e^{-yt} t^{\nu-\alpha-\frac{1}{2}} (1+t)^{\nu+\alpha-\frac{1}{2}}\; dt
$$
is the Whittaker function. 
When $\alpha=0$, the Whittaker function is the $K$-Bessel function:
$$
	W_{0, \nu}(y) = \left(\frac{y}{\pi}\right)^{\frac{1}{2}} K_\nu \left(\frac{y}{2}\right)
	.
$$
When $\nu = \pm (\alpha-1/2)$ and $1/2+\alpha \pm \nu\in \N$, we have
$$
	W_{\alpha, \pm(\alpha-1/2)} (y) = y^\alpha e^{-\frac{y}{2}}
$$
and
$$
	W_{1-\alpha, \pm (\alpha-1/2)}(y) = y^{1-\alpha} e^{-\frac{y}{2}}. 
$$

For each integer $m\geq 1$ and $\Re(w) > 1$, define a shifted Dirichlet series: 
\be\label{e:Dfphi}
	D_{f, \phi} (w; m)
	:=
	\sum_{n\geq 1} \frac{a(n+m) \overline{c_{\phi}(n)}}{n^{w+\frac{k}{2}-1}}
	.
\ee

and define
\be\label{e:Ufphi}
	U_{f, \phi}(z)
	:=
	y^{\frac{k}{2}} \overline{f(z)} \phi(z). 
\ee
The following theorem gives a complete meromorphic continuation of $D_{f, \phi} (w; m)$ to $\C$:
\begin{thm}\label{thm:Dfphi}
For $\Re(w) > 1$, we have
\be\label{e:Dfphi_Dirichlet}
	D_{f, \phi}(w; m)
	=
	\sum_{n=1}^\infty \frac{a(m+n) \overline{c_\phi(n)}}{n^{w+\frac{k}{2}-1}}
	.
\ee

For $\Re(w) < 1/2-k/2$, we have
\begin{multline}\label{e:Dfphi_spec}
	G_{f, \phi}(w)
	D_{f, \phi}(m; w)
	=
	\left\{
	\cI_{0, \cusp}(w; m) + \cI_{0, \cont, \Int}(w; m) + \Omega_{f, \phi}(w; m)
	\right\}
	\\
	-G_{f, \phi}(w) D_{f, \phi}^{\fin}(w; m)
	, 
\end{multline}
where
\be\label{e:Gfphi}
	G_{f, \phi}(w) 
	:=
	\frac{\Gamma\left(w+\frac{k}{2}+\bnu -\frac{1}{2}\right)
	\Gamma\left(w+\frac{k}{2}-\bnu-\frac{1}{2}\right) 
	(4\pi)^{1-w-\frac{k}{2}} }
	{\Gamma(w)}, 
\ee
\begin{multline}\label{e:cI0cusp}
	\cI_{0, \cusp}(w; m)
	\\
	:=
	(4\pi m)^{\frac{1}{2}-w} \sqrt{\pi} \Gamma(1-w)
	\sum_j \overline{\rho_j(-m)} 
	\frac{ 
	\Gamma\left(w-\frac{1}{2}-it_j \right) \Gamma\left(w-\frac{1}{2}+it_j \right)}
	{\Gamma\left(\frac{1}{2}-it_j \right) \Gamma\left(\frac{1}{2}+it_j \right)}
	\left<u_j, U_{f, \phi}\right>
	, 
\end{multline}
\begin{multline}\label{e:cI0contInt}
	\cI_{0, \cont, \Int}(w; m)
	:=
	(4\pi m)^{\frac{1}{2}-w} \sqrt{\pi} \Gamma(1-w) 
	\\
	\times
	\sum_{\cuspa} \frac{1}{4\pi i }\int_{C_w} 
	\tau_\cuspa\left(1/2-z, m\right) 
	\frac{\Gamma\left(w-\frac{1}{2}-z \right) \Gamma\left(w-\frac{1}{2}+z\right)}
	{\Gamma\left(\frac{1}{2}-z \right) \Gamma\left(\frac{1}{2}+z \right)}
	\left<E_\cuspa(*, 1/2+z), U_{f, \phi}\right>\; dz 
	, 
\end{multline}
\begin{multline}\label{e:Omegafphi}
	\Omega_{f, \phi} (w; m) 	
	:=
	-
	(4\pi m)^{\frac{1}{2}-w} \sqrt{\pi} \Gamma(1-w) 
	\sum_{\ell=0}^{\lfloor-\Re(w)+\frac{1}{2}\rfloor}
	\frac{(-1)^\ell}{\ell!} 
	\frac{\Gamma\left(\ell +2w-1\right)}
	{\Gamma\left(1-\ell-w\right) \Gamma\left(\ell+w\right)}
	\\
	\times
	\frac{1}{2}
	\sum_{\cuspa}
	\left\{
	\tau_\cuspa \left(1-\ell -w, m\right)
	\left<E_\cuspa (*, \ell+w), U_{f, \phi}\right>
	+
	\tau_\cuspa \left(\ell+w, m\right)
	\left<E_\cuspa (*, 1-\ell-w ), U_{f, \phi}\right>
	\right\}
\end{multline}
and
\be\label{e:Dfphi_fin}
	D_{f,\phi}^{\fin}(w; m)
	:=
	\frac{\Gamma(w)\Gamma(1-w)}
	{\Gamma\left(\frac{1}{2}+\frac{k}{2}+\bnu\right) 
	\Gamma\left(\frac{1}{2}+\frac{k}{2}-\bnu\right)}
	\sum_{n=1}^{m-1} \frac{a(m-n) \overline{c_\phi(-n)}}{n^{w+\frac{k}{2}-1}} 
	.
\ee
For $\Re(w) < 1/2-k/2$, the functions $\cI_{0, \cusp}(w; m)$ and $\cI_{0, \cont, \Int}(w; m)$ satisfy the following upper bound:
$$
	\cI_{0, \cusp}(w; m) + \cI_{0, \cont, \Int} 
	\ll 
	m^{1/2-w+\epsilon}
	.
$$

The function $D_{f,\phi}(w; m)$ has a meromorphic continuation to all $w\in \C$, 
and $G_{f, \phi}(w) D_{f, \phi}(w; m) - \Omega_{f, \phi}(w; m)$ has simple poles at the points 
$w=1/2\pm it_j -r$ for $r\geq 0$ with the residues
\begin{multline}\label{e:Dfphi-Omegafphi_Res}
	d_{j, r}(m)
	:=
	\Res_{w = 1/2\pm it_j -r } \left(G_{f, \phi}(w) D_{f, \phi}(w; m) - \Omega_{f, \phi}(w; m)\right)
	\\
	=
	\frac{(-1)^r (4\pi m)^{\mp it_j +r} \sqrt{\pi}
	\Gamma\left(\frac{1}{2} \mp it_j +r\right)}
	{r!} 
	\frac{\Gamma\left(\pm 2it_j-r\right) }{\Gamma\left(\frac{1}{2}+it_j\right) \Gamma\left(\frac{1}{2}-it_j\right)}
	\overline{\rho_j(-m)} \left<u_j, U_{f, \phi}\right>
	.
\end{multline}
In particular, at $w=-k/2+1/2\pm \bnu$, we have
\be\label{e:Dfphi_special}
	D_{f, \phi}( -k/2+1/2\pm \bnu ;m)
	=
 	a(m)\overline{c_\phi^\mp }
	\frac{
	(4\pi)^{-\frac{1}{2}\pm \bnu} 
	\Gamma\left(-\frac{k}{2}+\frac{1}{2}\pm \bnu\right) }
	{\Gamma\left(\pm 2\bnu \right)}
	-
	D_{f, \phi}^{\fin}(-k/2+1/2\pm \bnu; m)
\ee

Let 
$$
	D^*_{f, \phi}(w; m)
	:=
	G_{f, \phi}(w) D_{f, \phi}(w; m) - \Omega_{f, \phi}(w; m)
	.
$$
Then for $1/2-k/2 < \Re(w_0) < 1$, 
\begin{multline}\label{e:Dfphi_middle}
	D^*_{f, \phi}(w_0; m)
	=
	\int_{C} \left[ \frac{D^*_{f, \phi} (w; m) e^{(w-w_0)^2}}{w-w_0} 
	- \sum_{0\leq r\leq k/2} \sum_j \sum_\pm
	\frac{d_{r, j} e^{\left(\frac{1}{2}-r\pm it_j -w_0\right)^2}}
	{w-\left(\frac{1}{2}-r\pm it_j\right)} \right]
	\; dw, 
\end{multline}
where $C$ is the infinite rectangle with right hand side the vertical line $\Re(w) = 2$ and left hand side the vertical line $\Re(w) = 1/2-k/2-\epsilon$. 
In this region, away from the poles, 
\be\label{e:D*fphi_upper}
	D^*_{f, \phi}(w_0; m) 
	\ll m^{k/2+\epsilon}
	.
\ee

\end{thm}

The proof of this theorem is similar in many ways, but a bit more complicated than, 
the corresponding theorem in \cite{HHR}, which applies to a less general object; 
in \cite{HHR}, $\phi$ is restricted to be a holomorphic cusp form.   A completely general version of 
Theorem \ref{thm:Dfphi} is given in \cite{Hul13}.   

\section{A first moment identity}\label{s:first}

Let $h(t)$ be an even function on $\C$, which is holomorphic on $|\Im(t)| < 1/2+\epsilon$, for some $\epsilon>0$. 
As $t\to \infty$, we assume that 
$$
	\left|h(t) \right| \ll (1+|t|)^a, 
$$
for $a<0$. 
In this section we will prove the following theorem.  
\begin{thm}\label{thm:first}
For $m\geq 1$, we have
\begin{multline}\label{e:first}
	(4\pi)^{k-\frac{1}{2}} i^k 
	\left\{
	\sum_{j} 
	\frac{h(t_j) }{\cosh(\pi t_j)}
	\overline{\rho_j(-m)} 
	\frac{\left<u_j, U_{f, \phi}\right>}{\Gamma\left(\frac{k}{2}+\bnu+it_j\right)
	\Gamma\left(\frac{k}{2}+\bnu-it_j\right)}
	\right.
	\\
	\left.
	+
	\sum_{\cuspa} 
	\frac{1}{4\pi} \int_{-\infty}^\infty 
	\frac{h(t)}{\cosh(\pi t)}
	\tau_\cuspa (1/2-it, m) 
	\frac{\left<E_{\cuspa}(*, 1/2+it), U_{f, \phi}\right>} 
	{\Gamma\left(\frac{k}{2}+\bnu+it\right) \Gamma\left(\frac{k}{2}+\bnu-it\right)}
	\; dt
	\right\}
	\\
	=
	M_{f, \phi}(m) + E^{(1)}_{f, \phi}(m) + E^{(2)}_{f, \phi}(m)
	, 
\end{multline}
where
\begin{multline}\label{e:Mfphi}
	M_{f, \phi}(m)
	=
	(4\pi)^{\frac{k}{2}-\bnu} 
	\frac{\Gamma\left(-\frac{k}{2}+\frac{1}{2}-\bnu\right) }
	{\Gamma\left(\frac{1}{2}+\bnu\right) \Gamma\left(\frac{1}{2}-\bnu\right)}
	\frac{a(m)\overline{c^+_\phi} }{m^{\frac{k}{2}+\bnu}}
	\frac{1}{2\pi^2} \int_{-\infty}^\infty h(t) t \tanh(\pi t) \; dt	
	\\
	+
	(4\pi)^{\frac{k}{2}+\bnu}
	\frac{\Gamma\left(-\frac{k}{2}+\frac{1}{2}+\bnu\right)}
	{\Gamma\left(\frac{1}{2}+\bnu\right) \Gamma\left(\frac{1}{2}-\bnu\right)}
	\frac{a(m) \overline{c^-_\phi} }{m^{\frac{k}{2}-\bnu}}
	\frac{1}{2\pi^2} \int_{-\infty}^\infty
	h(t) t \tanh (\pi t)
	\frac{\Gamma\left(\frac{k}{2}-\bnu+it\right)
	\Gamma\left(\frac{k}{2}-\bnu-it\right)}
	{\Gamma\left(\frac{k}{2}+\bnu +it\right)
	\Gamma\left(\frac{k}{2}+\bnu -it\right) }
	\; dt
	, 
\end{multline}
\begin{multline}\label{e:E1fphi}
	E^{(1)}_{f, \phi}(m)
	=
	-
	4
	(4\pi)^{\frac{k-1}{2}} i^k  
	\frac{1}{2\pi i} \int_{(\sigma_u = 1/2+2\epsilon)} 
	\frac{h(u/i) u \tan(\pi u)}{\Gamma\left(\frac{k}{2}+\bnu +u\right) \Gamma\left(\frac{k}{2}+\bnu-u\right)}
	\\
	\times
	\frac{1}{2\pi i } \int_{(\sigma_w = 1/2-k/2-\epsilon)}
	\Gamma\left(-w+u\right) \Gamma\left(-w-u\right)
	\Gamma\left(w+\frac{k}{2}-\bnu\right) \Gamma\left(w+\frac{k}{2}+\bnu\right)
	\\
	\times
	\frac{1}{\Gamma\left(\frac{1}{2}+\frac{k}{2}+\bnu\right)
	\Gamma\left(\frac{1}{2}+\frac{k}{2}-\bnu\right)}
	m^{w}
	\sum_{n=1}^{m-1} 
	\frac{a(m-n) \overline{c_\phi(-n)}}{n^{w+\frac{k}{2}-\frac{1}{2}}} 
	\; dw\; du
\end{multline}
and 
\begin{multline}\label{e:E2fphi}
	E^{(2)}_{f, \phi}(m)
	=
	(4\pi)^{\frac{k-1}{2}} i^k \frac{4}{\pi} 
	\frac{1}{2\pi i} \int_{(\sigma_u = 1/2+2\epsilon)} 
	\frac{h(u/i) u 
	\Gamma\left(\frac{1}{2}+u\right) \Gamma\left(\frac{1}{2}-u\right)}
	{\Gamma\left(\frac{k}{2}+\bnu +u\right) \Gamma\left(\frac{k}{2}+\bnu-u\right)}
	\frac{1}{2\pi i } \int_{(\sigma_w = 1/2+\epsilon)}
	\frac{\Gamma\left(-w+u\right)}
	{\Gamma\left(1+w+u\right)}
	\\
	\times
	\Gamma\left(w+\frac{k}{2}+\bar{\nu}\right) 
	\Gamma\left(w+\frac{k}{2}-\bar{\nu}\right)
	m^{w} \sum_{n=1}^\infty \frac{a(m+n) \overline{c_\phi(n)}}{n^{w+\frac{k}{2}-\frac{1}{2}}}
	\; dw\; du
	.
\end{multline}

If $\phi$ is holomorphic, then $E^{(1)}_{f, \phi}(m)$ is not present, 
and $M_{f, \phi}(m)$ is not present if $\phi$ is cuspidal. 
\end{thm}


\subsection{A proof of Theorem \ref{thm:first}}

The left hand side of \eqref{e:first} is an integral transform of the spectral expansion of the shifted Dirichlet series $D_{f, \phi}(w+1/2; m)$, 
in its region of absolute convergence. 
To prove the theorem, we will move the line of integration from the region where the spectral side is convergent, to the region where the Dirichlet series is convergent. 

Consider
\begin{multline*}
	(4\pi)^{\frac{k-1}{2}} i^k \frac{4}{\pi} 
	\frac{1}{2\pi i} \int_{(\sigma_u=1/2+2\epsilon)} \frac{h(u/i) u \Gamma\left(\frac{1}{2}+u\right) \Gamma\left(\frac{1}{2}-u\right)}
	{\Gamma\left(\frac{k}{2}+\bnu+u\right) \Gamma\left(\frac{k}{2}+\bnu-u\right)}
	\frac{1}{2\pi i} \int_{(\sigma_w = 1/2+\epsilon)} 
	\frac{\Gamma\left(-w+u\right)}
	{\Gamma\left(1+w+u\right)}
	\\
	\times
	\Gamma\left(w+\frac{k}{2}+\bnu\right) \Gamma\left(w+\frac{k}{2}-\bnu\right)
	m^w D_{f, \phi} (w+1/2; m) \; dw \; du. 
\end{multline*}

For $\Re(w) < -k/2$, by \eqref{e:Dfphi_spec}, we have
\begin{multline*}
	D_{f, \phi}^*(w+1/2; m)
	:=
	\frac{\Gamma\left(w+\frac{k}{2}+\bnu\right) \Gamma\left(w+\frac{k}{2}-\bnu\right)}
	{\Gamma\left(w+\frac{1}{2}\right)}
	D_{f, \phi} (w+1/2; m) 
	\\
	=
	(4\pi)^{\frac{k-1}{2}+w} 
	\left\{\cI_{0, \cusp}(w+1/2; m) + \cI_{0, \cont, \Int}(w+1/2; m)
	+\Omega_{f, \phi}(w+1/2; m)\right\}
	\\
	-
	\frac{
	\Gamma\left(w+\frac{k}{2}+\bnu\right) \Gamma\left(w+\frac{k}{2}-\bnu\right) 
	\Gamma\left(\frac{1}{2}-w\right)}
	{\Gamma\left(\frac{1}{2}+\frac{k}{2}+\bnu\right) \Gamma\left(\frac{1}{2}+\frac{k}{2}-\bnu\right)}
	\sum_{n=1}^{m-1} \frac{a(m-n) \overline{c_{\phi}(-n)}}{n^{w+\frac{k}{2}-\frac{1}{2}}} 
	, 
\end{multline*}
where
\begin{multline*}
	\Omega_{f, \phi}(w+1/2; m)
	=
	- (4\pi m)^{-w} \sqrt{\pi} \Gamma\left(\frac{1}{2}-w\right) 
	\sum_{\ell=0}^{\lfloor-\Re(w)\rfloor} 
	\frac{(-1)^\ell}{\ell!} \frac{\Gamma\left(\ell+2w\right)}
	{\Gamma\left(\frac{1}{2}-\ell-w\right) \Gamma\left(\frac{1}{2}+\ell +w\right)}
	\\
	\times
	\frac{1}{2} \sum_\cuspa \left\{ \tau_\cuspa (1/2-\ell -w, m) \left<E_\cuspa (*, \ell+w+1/2), U_{f, \phi}\right>
	+
	\tau_\cuspa (\ell+w+1/2, m) \left<E_\cuspa (*, 1/2-\ell-w), U_{f, \phi}\right>\right\}
	.
\end{multline*}

To get the left hand side of \eqref{e:first}, we move the $w$ line of integration to $\sigma_w = -k/2-\epsilon$. 
As we move the $w$ line of integration, we pass over the poles of $\cI_{0, \cusp}(w+1/2; m)$, $\Omega_{f, \phi}(w+1/2; m)$, $\Gamma(w+1/2)$, and also $\Gamma(w+k/2+\bnu) \Gamma(w+k/2-\bnu)$. 
Denote by $\Res_{f, \phi}$ the sum of the residues associated with the poles of $\cI_{0, \cusp}(w+1/2; m)$ and $\Omega_{f, \phi}(w+1/2; m)$.
Then we get
\begin{multline*}
	(4\pi)^{\frac{k-1}{2}} i^k \frac{4}{\pi} 
	\frac{1}{2\pi i} \int_{(\sigma_u=1/2+2\epsilon)} \frac{h(u/i) u \Gamma\left(\frac{1}{2}+u\right) \Gamma\left(\frac{1}{2}-u\right)}
	{\Gamma\left(\frac{k}{2}+\bnu+u\right) \Gamma\left(\frac{k}{2}+\bnu-u\right)}
	\frac{1}{2\pi i} \int_{(\sigma_w = 1/2+\epsilon)} 
	\frac{\Gamma\left(-w+u\right)}
	{\Gamma\left(1+w+u\right)}
	\\
	\times
	\Gamma\left(w+\frac{k}{2}+\bnu\right) \Gamma\left(w+\frac{k}{2}-\bnu\right)
	m^w D_{f, \phi} (w+1/2; m) \; dw \; du
	\\
	=
	(4\pi)^{\frac{k-1}{2}} i^k \frac{4}{\pi} 
	\frac{1}{2\pi i} \int_{(\sigma_u=1/2+2\epsilon)} \frac{h(u/i) u \Gamma\left(\frac{1}{2}+u\right) \Gamma\left(\frac{1}{2}-u\right)}
	{\Gamma\left(\frac{k}{2}+\bnu+u\right) \Gamma\left(\frac{k}{2}+\bnu-u\right)}
	\frac{1}{2\pi i} \int_{(\sigma_w = -k/2-\epsilon)} 
	\frac{\Gamma\left(-w+u\right)}
	{\Gamma\left(1+w+u\right)}
	\\
	\times
	\Gamma\left(w+\frac{k}{2}+\bnu\right) \Gamma\left(w+\frac{k}{2}-\bnu\right)
	m^w D_{f, \phi} (w+1/2; m) \; dw \; du
	+
	\Res_{f, \phi} 
	\\
	+
	(4\pi)^{\frac{k-1}{2}} i^k \frac{4}{\pi} 
	\frac{1}{2\pi i} \int_{(\sigma_u=1/2+2\epsilon)} \frac{h(u/i) u \Gamma\left(\frac{1}{2}+u\right) \Gamma\left(\frac{1}{2}-u\right)}
	{\Gamma\left(\frac{k}{2}+\bnu+u\right) \Gamma\left(\frac{k}{2}+\bnu-u\right)}
	\frac{\Gamma\left(\frac{k}{2} -\bnu+u\right)}
	{\Gamma\left(1-\frac{k}{2}+\bnu+u\right)}
	\\
	\times
	\Gamma\left( 2\bnu\right)
	m^{-\frac{k}{2}+\bnu} D_{f, \phi} (1/2-k/2+ \bnu; m) \; du
	\\
	+
	(4\pi)^{\frac{k-1}{2}} i^k \frac{4}{\pi} 
	\frac{1}{2\pi i} \int_{(\sigma_u=1/2+2\epsilon)} \frac{h(u/i) u \Gamma\left(\frac{1}{2}+u\right) \Gamma\left(\frac{1}{2}-u\right)}
	{\Gamma\left(\frac{k}{2}+\bnu+u\right) \Gamma\left(\frac{k}{2}+\bnu-u\right)}
	\frac{\Gamma\left(\frac{k}{2} +\bnu+u\right)}
	{\Gamma\left(1-\frac{k}{2}-\bnu+u\right)}
	\\
	\times
	\Gamma\left(-2\bnu\right)
	m^{-\frac{k}{2}- \bnu} D_{f, \phi} (1/2-k/2-\bnu; m) \; du
	\\
	+
	\sum_{\ell=0}^{\frac{k}{2}-1}
	\frac{(-1)^\ell}{\ell!}
	(4\pi)^{\frac{k-1}{2}} i^k \frac{4}{\pi} 
	D_{f, \phi}^*(-\ell; m)
	\frac{1}{2\pi i} \int_{(\sigma_u=1/2+2\epsilon)} \frac{h(u/i) u \Gamma\left(\frac{1}{2}+u\right) \Gamma\left(\frac{1}{2}-u\right)}
	{\Gamma\left(\frac{k}{2}+\bnu+u\right) \Gamma\left(\frac{k}{2}+\bnu-u\right)}
	\frac{\Gamma\left(\frac{1}{2}+\ell+u\right)}
	{\Gamma\left(\frac{1}{2}-\ell+u\right)}
	\; du
	.
\end{multline*}

For each $\ell\geq 0$, we have
\begin{multline}\label{e:vanishing}
	\frac{(-1)^{\ell}}{\ell!}
	\frac{1}{2\pi i} \int_{(\sigma_u=1/2+2\epsilon)} \frac{h(u/i) u \Gamma\left(\frac{1}{2}+u\right) \Gamma\left(\frac{1}{2}-u\right)}
	{\Gamma\left(\frac{k}{2}+\bnu+u\right) \Gamma\left(\frac{k}{2}+\bnu-u\right)}
	\frac{\Gamma\left(\frac{1}{2}+\ell+u\right)}
	{\Gamma\left(\frac{1}{2}-\ell+u\right)}
	\; du
	\\
	=
	\frac{1}{\ell!}
	\frac{1}{2\pi i} \int_{(\sigma_u=1/2+2\epsilon)} 
	h(u/i) u
	\frac{
	\Gamma\left(\frac{1}{2}+\ell+u\right) 
	\Gamma\left(\frac{1}{2}+\ell-u\right)}
	{\Gamma\left(\frac{k}{2}+\bnu+u\right) \Gamma\left(\frac{k}{2}+\bnu-u\right)}
	\; du
	=
	0, 
\end{multline}
since $h(u/i) = h(-u/i)$ and $h(\pm i/2)=0$. 

Let
\begin{multline}\label{e:Ifphi}
	I_{f, \phi}(m)
	\\
	:=
	(4\pi)^{\frac{k-1}{2}} i^k \frac{4}{\pi} 
	\frac{1}{2\pi i} \int_{(\sigma_u=1/2+2\epsilon)} \frac{h(u/i) u \Gamma\left(\frac{1}{2}+u\right) \Gamma\left(\frac{1}{2}-u\right)}
	{\Gamma\left(\frac{k}{2}+\bnu+u\right) \Gamma\left(\frac{k}{2}+\bnu-u\right)}
	\frac{1}{2\pi i} \int_{(\sigma_w = -k/2-\epsilon)} 
	\frac{\Gamma\left(-w+u\right)}
	{\Gamma\left(1+w+u\right)}
	\\
	\times
	\Gamma\left(w+\frac{1}{2}\right)
	m^w 
	(4\pi)^{\frac{k-1}{2}+w} 
	\left\{
	\cI_{0, \cusp}(w+1/2; m) 
	+ \cI_{0, \cont, \Int}(w+1/2; m)
	+ \Omega_{f, \phi}(w+1/2; m)
	\right\}
	\\
	\times
	\; dw \; du
	+
	\Res_{f, \phi}
	.
\end{multline}
Here
\begin{multline*}
	\Gamma\left(w+\frac{1}{2}\right)
	m^w 
	(4\pi)^{\frac{k-1}{2}+w} 
	\cI_{0, \cusp}(w+1/2; m)
	\\
	=
	\Gamma\left(w+\frac{1}{2}\right)
	\Gamma\left(\frac{1}{2}-w\right)
	(4\pi)^{\frac{k-1}{2}} 
	\sqrt{\pi} 
	\sum_j 
	\overline{\rho_j(-m)} 
	\frac{\Gamma\left(w-it_j\right) \Gamma\left(w+it_j\right)}
	{\Gamma\left(\frac{1}{2}-it_j\right) \Gamma\left(\frac{1}{2}+it_j\right)}
	\left<u_j, U_{f, \phi}\right>
\end{multline*}
and
\begin{multline*}
	\Gamma\left(w+\frac{1}{2}\right)
	m^w 
	(4\pi)^{\frac{k-1}{2}+w} 
	\cI_{0, \cont, \Int}(w+1/2; m)
	\\
	=
	\Gamma\left(w+\frac{1}{2}\right)
	\Gamma\left(\frac{1}{2}-w\right)
	(4\pi)^{\frac{k-1}{2}} 
	\sqrt{\pi} 
	\sum_{\cuspa} 
	\frac{1}{4\pi} \int_{-\infty}^\infty \tau_\cuspa \left(1/2-it, m\right)
	\frac{\Gamma\left(w-it\right) \Gamma\left(w+it\right)}
	{\Gamma\left(\frac{1}{2}-it\right) \Gamma\left(\frac{1}{2}+it\right)}
	\\
	\times
	\left<E_{\cuspa}(*, 1/2+it), U_{f, \phi}\right>
	\; dt
	.
\end{multline*}

By \eqref{e:Dfphi_special}, 
$$
	D_{f, \phi}( -k/2+1/2\pm \bnu ;m)
	=
 	a(m)\overline{c_\phi^\mp }
	\frac{
	(4\pi)^{-\frac{1}{2}\pm \bnu} 
	\Gamma\left(-\frac{k}{2}+\frac{1}{2}\pm \bnu\right) }
	{\Gamma\left(\pm 2\bnu \right)}
	-
	D_{f, \phi}^{\fin}(-k/2+1/2\pm \bnu; m)
	, 
$$
so we get
$$
	I_{f, \phi}(m)
	=
	M_{f, \phi}(m)
	+
	E_{f, \phi}^{(1)}(m)
	+
	E_{f, \phi}^{(2)}(m)
	.
$$

Define
\begin{multline}\label{e:h**}
	h^{**}(t)
	:=
	\frac{1}{\Gamma\left(\frac{1}{2}-it\right) \Gamma\left(\frac{1}{2}+it\right)}
	\frac{1}{2\pi i} \int_{(\sigma_u=1/2+2\epsilon)} \frac{h(u/i) u \Gamma\left(\frac{1}{2}+u\right) \Gamma\left(\frac{1}{2}-u\right)}
	{\Gamma\left(\frac{k}{2}+\bnu+u\right) \Gamma\left(\frac{k}{2}+\bnu-u\right)}
	\\
	\times
	\frac{1}{2\pi i} \int_{(\sigma_w =\epsilon)} 
	\frac{
	\Gamma\left(-w+u\right)
	\Gamma\left(-w+\frac{1}{2}\right)
	\Gamma\left(w+\frac{1}{2}\right)
	\Gamma\left(w-it\right) \Gamma\left(w+it\right)}
	{\Gamma\left(1+w+u\right)}
	\; dw \; du
	.
\end{multline}
Since the integrals and series in \eqref{e:Ifphi} converge absolutely, 
we can change the ordering of integrals and series, and move the $w$ line of integration to $\sigma_w = \epsilon$.
Then $\Res_{f, \phi}$ is canceled, and applying \eqref{e:vanishing}, we get
\begin{multline*}
	I_{f, \phi}(m)
	=
	(4\pi)^{k-1} \sqrt{\pi} i^k \frac{4}{\pi} 
	\left\{
	\sum_j h^{**}(t_j) \overline{\rho_j(-m)} \left<u_j, U_{f, \phi}\right>
	\right.
	\\
	\left.
	+
	\sum_{\cuspa}
	\frac{1}{4\pi} \int_{-\infty}^\infty h^{**}(t) \tau_{\cuspa}(1/2-it, m) \left<E_{\cuspa}(*, 1/2+it), U_{f, \phi}\right> \; dt
	\right\}
	.
\end{multline*}

We will now compute $h^{**}(t)$. 
Applying Theorem 2.4.3 in \cite{AAR}, we get
$$
	\frac{1}{2\pi i} \int_{(\sigma_w =\epsilon)} 
	\frac{
	\Gamma\left(-w+u\right)
	\Gamma\left(-w+\frac{1}{2}\right)
	\Gamma\left(w+\frac{1}{2}\right)
	\Gamma\left(w-it\right) \Gamma\left(w+it\right)}
	{\Gamma\left(1+w+u\right)}
	\; dw 
	=
	\frac{\Gamma\left(\frac{1}{2}-it\right) \Gamma\left(\frac{1}{2}+it\right)}
	{(u+it) (u-it)}
	.
$$
Thus
\begin{multline*}
	h^{**}(t)
	=
	\frac{1}{\Gamma\left(\frac{1}{2}-it\right) \Gamma\left(\frac{1}{2}+it\right)}
	\frac{1}{2\pi i} \int_{(\sigma_u=1/2+2\epsilon)} \frac{h(u/i) u \Gamma\left(\frac{1}{2}+u\right) \Gamma\left(\frac{1}{2}-u\right)}
	{\Gamma\left(\frac{k}{2}+\bnu+u\right) \Gamma\left(\frac{k}{2}+\bnu-u\right)}
	\frac{\Gamma\left(\frac{1}{2}-it\right) \Gamma\left(\frac{1}{2}+it\right)}
	{(u+it) (u-it)}
	\; du
	\\
	=
	\frac{1}{2} h(t) 
	\frac{\Gamma\left(\frac{1}{2}-it\right) \Gamma\left(\frac{1}{2}+it\right)}
	{\Gamma\left(\frac{k}{2}+\bnu+it\right) \Gamma\left(\frac{k}{2}+\bnu-it\right)}
	.
\end{multline*}
Then we get \eqref{thm:first}. 


\subsection{A proof of Theorem \ref{thm:first_Eisenstein}}

Assume that $N\geq 1$ is square-free. 
Let $\Gamma:= \Gamma_0(N)$. 
The Eisenstein series at the cusp $\cuspa=1/N$ is defined in \eqref{e:Ea}:
$$
	E_N(z, 1/2+ir)
	=
	\sum_{\gamma\in \Gamma_N \bsl \Gamma} \Im(\sigma_N^{-1} \gamma z)^{\frac{1}{2}+ir}, 
$$
where
$$
	\sigma_N = \bpm 1 & \\ N & 1\ebpm \in \Gamma.
$$
Since $\sigma_N^{-1} \Gamma_N \sigma_N= \Gamma_\infty$, 
we get
$$
	E_N(z, 1/2+ir)
	=
	\sum_{\gamma\in \Gamma_\infty \bsl \Gamma} \Im(\gamma z)^{\frac{1}{2}+ir}.
$$

Define
\be\label{e:div_N}
	\sigma^N_{-2ir}(n)
	:=
	\left(\sum_{d\mid n, \atop (d, N)=1} d^{-2ir} \right)
	\prod_{p\mid N, p^{\alpha}\|n, \atop \alpha\geq 0} 
	\frac{p^{-1-2ir}}{1-p^{-2ir}} \left(p-p^{\alpha(-2ir)+1 } -1 + p^{(\alpha+1)(-2ir)}\right)
	.
\ee

We have the following Fourier expansion which we quote from \cite{DI83}, and following \cite{HM06}:
\begin{multline*}
	E_N \left(z, 1/2+ir\right) 
	=
	y^{\frac{1}{2}+ir}
	+
	\frac{\zeta^*(2ir)}{\zeta^*(1+2ir)\prod_{p\mid N} (1-p^{-1-2ir})} 
	\frac{\varphi(N)}{N^{1+2ir}} 
	y^{\frac{1}{2}-ir}\\
	+
	\frac{1}{\zeta^*(1+2ir)\prod_{p\mid N} (1-p^{-1-2ir})}
	\sum_{n\neq 0}
	\frac{\sigma^N_{-2ir}(|n|)|n|^{ir}}{\sqrt{|n|}}
	W_{0, ir}(4\pi|n|y) e^{2\pi inx}, 
\end{multline*}
since
$$
	\sqrt{\pi} W_{0, ir}(y) = \sqrt{y} K_{ir} \left(\frac{y}{2}\right)
	. 
$$

For $\ell\in \Z$, the Maass raising operator $R_\ell$ is defined to be the differential operator:
$$
	R_\ell := iy \frac{\partial}{\partial x} + y \frac{\partial}{\partial y} + \frac{\ell}{2}.
$$
For $n \geq 1$, we have
$$
	\left( R_{k-2}\circ R_{k-4} \circ \cdots \circ R_0 W_{0, ir} (4\pi n y) e^{2\pi inx} \right)
	=
	(-1)^{\frac{k}{2}} W_{\frac{k}{2}, ir} (4\pi |n|y ) e^{2\pi inx}
	.
$$
For $n \leq -1$, we have
$$
	\left( R_{k-2} \circ R_{k-4} \circ \cdots \circ R_0 W_{0,ir} (4\pi |n|y )e^{2\pi inx} \right)
	=
	\frac{(-1)^{\frac{k}{2}} \Gamma\left(ir+\frac{1}{2}+\frac{k}{2}\right)}
	{\Gamma\left(ir+\frac{1}{2}-\frac{k}{2}\right) } 
	W_{-\frac{k}{2}, ir} (4\pi |n|y)e^{2\pi inx}
	.
$$

Define
\begin{multline}\label{e:EkN}
	E^{(k)}_N (z, 1/2+ir) 
	:= (-1)^{\frac{k}{2}} \zeta^*(1+2ir) \prod_{p\mid N}(1-p^{-1-2ir}) 
	\left(R_{k-2}\circ \cdots \circ R_0 E_N\right)(z, 1/2+ir)
	\\
	=
	(-1)^{\frac{k}{2}} \zeta^*(1+2ir) \prod_{p\mid N}(1-p^{-1-2ir}) 
	\frac{\Gamma\left(ir+\frac{k+1}{2}\right)}{\Gamma\left(ir+\frac{1}{2}\right)}
	\sum_{\gamma = \sm a & b\\ c& d\esm \in \Gamma_\infty\bsl \Gamma}
	\left(\frac{c\bar{z}+d}{cz+d}\right)^{\frac{k}{2}}
	\Im(\gamma z)^{\frac{1}{2}+ir}
	.
\end{multline}
Then we get
\begin{multline}\label{e:EkN_Fourier}
	E^{(k)}_N (z, 1/2+ir) 
	\\
	=
	(-1)^{\frac{k}{2}} \zeta^*(1+2ir) \prod_{p\mid N}(1-p^{-1-2ir}) 
	\frac{\Gamma\left(ir+\frac{1+k}{2}\right)}{\Gamma\left(ir+\frac{1}{2}\right)}
	y^{ir+\frac{1}{2}}
	+
	(-1)^{\frac{k}{2}} \zeta^*(2ir)
	\frac{\varphi(N)}{N^{1+2ir}} \frac{\Gamma\left(-ir+\frac{1+k}{2}\right)}{\Gamma\left(\frac{1}{2}-ir\right)} 
	y^{\frac{1}{2}-ir}
	\\
	+
	\sum_{n\geq 1}
	\frac{\sigma^N_{-2ir}(|n|)|n|^{ir}}{\sqrt{|n|}}
	W_{\frac{k}{2}, ir}(4\pi|n|y) e^{2\pi inx}
	\\
	+
	\frac{\Gamma\left(ir+\frac{k+1}{2}\right)}{\Gamma\left(ir+\frac{1-k}{2}\right)}
	\sum_{n\leq -1}
	\frac{\sigma^N_{-2ir}(|n|)|n|^{ir}}{\sqrt{|n|}}
	W_{-\frac{k}{2}, ir}(4\pi|n|y) e^{2\pi inx}.
\end{multline}


Take $\phi(z) = E^{(k)}_N(z, 1/2+\overline{ir_1})$, so $\nu= \overline{ir_1}$.
Recalling \eqref{e:Ufphi}, 
$$
	U_{f, \phi}(z) =
	U_{f, E_N^{(k)}(*, 1/2+\overline{ir_1})}(z)
	=
	y^{\frac{k}{2}} \overline{f(z)} E_N^{(k)}(z, 1/2+\overline{ir_1}). 
$$

\begin{lem}
For each $j\geq 1$, we get
\begin{multline}\label{e:inner_ujU}
	\left<u_j, U_{f, E_N^{(k)}(*, 1/2+\overline{ir_1})}\right>
	\\
	=
	(-1)^{\frac{k}{2}} 
	(4\pi)^{-\frac{k}{2}} (2\pi)^{-2ir_1}
	\prod_{p\mid N}(1-p^{-1-2ir_1}) 
	\Gamma\left(ir_1+\frac{k}{2}+it_j\right) \Gamma\left(ir_1+\frac{k}{2}-it_j\right)
	\\
	\times
	(-1)^{\alpha_j}
	\cL(1/2+ir_1, f\times u_j)
	, 
\end{multline}
where
$$
	\cL(s, f\times u_j)
	=
	\zeta(2s) \sum_{n=1}^\infty \frac{A(n) \rho_j(n)}{n^{s}}.
$$
Here $\alpha_j = 0$ if $u_j$ is even, and $\alpha_j = 1$ if $u_j$ is odd. 

For each $a\mid N$, we get
\begin{multline}\label{e:inner_E1/aU}
	\left<E_{1/a}(*, 1/2+it), U_{f, E_N^{(k)}(*, 1/2+\overline{ir_1})}\right>
	\\
	=
	(-1)^{\frac{k}{2}} 
	(4\pi)^{-\frac{k}{2}} (2\pi)^{-2ir_1}
	\prod_{p\mid N}(1-p^{-1-2ir_1})
	\Gamma\left(ir_1+\frac{k}{2}+it\right) \Gamma\left(ir_1+\frac{k}{2}-it\right)
	\\
	\times
	\frac{2 \left(\frac{N}{a}\right)^{-\frac{1}{2}-it}}
	{\zeta^*(1+2it)\prod_{p\mid N} (1-p^{-1-2it})}
	L\left(1/2+ir_1, f\times E_{1/a}(*, 1/2+it) \right)
	, 
\end{multline} 
where
$$
	L\left(s, f\times E_{1/a}(*, 1/2+it) \right)
	=
	\zeta(2s)
	\sum_{n\geq 1} 
	\frac{A(n) \sigma^{a}_{-2it}(n) n^{it}}	
	{n^s}
	.
$$
\end{lem}
\begin{proof}
For each $j\geq 1$, by \eqref{e:EkN}, we have
\begin{multline}\label{e:inner_fuj}
	\left<u_j, U_{f, E^{(k)}_N (*, 1/2+\overline{ir_1})}\right>
	=
	(-1)^{\frac{k}{2}} \zeta^*(1+2ir_1) \prod_{p\mid N}(1-p^{-1-2ir_1}) 
	\frac{\Gamma\left(ir_1+\frac{k+1}{2}\right)}{\Gamma\left(ir_1+\frac{1}{2}\right)}
	(2\pi)^{-ir_1-\frac{k}{2}}
	\\
	\times
	(-1)^{\alpha_j}
	\sum_{n\geq 1} \frac{A(n) \rho_j(n)}{n^{ir_1+\frac{1}{2}}} 
	\int_0^\infty K_{it_j}(y)e^{-y} y^{ir_1+\frac{k}{2}} \; \frac{dy}{y}
	\\
	=
	(-1)^{\frac{k}{2}} 
	(4\pi)^{-\frac{k}{2}} (2\pi)^{-2ir_1}
	\prod_{p\mid N}(1-p^{-1-2ir_1}) 
	\Gamma\left(ir_1+\frac{k}{2}+it_j\right) \Gamma\left(ir_1+\frac{k}{2}-it_j\right)
	\\
	\times
	(-1)^{\alpha_j}
	\cL(ir_1+1/2, f\times u_j)
	, 
\end{multline}
since
$$
	\int_0^\infty K_{it}(y) e^{-y} y^{s+\frac{k}{2}}\; \frac{dy}{y}
	=
	2^{-s-\frac{k}{2}} \sqrt{\pi} \frac{\Gamma\left(s+\frac{k}{2}-it\right)\Gamma\left(s+\frac{k}{2}+it\right)}
	{\Gamma\left(s+\frac{k+1}{2}\right)}
	.
$$

For each cusp $\cuspa = 1/a$ for $a\mid N$, 
recalling \eqref{e:EisensteinSeries_1/a_Fourier}, 
and by \eqref{e:EkN}, we have
\begin{multline*}
	\left<E_{1/a}(*, 1/2+it),U_{f, E^{(k)}_N(*, 1/2+\bar{ir_1})}\right>
	\\
	=
	\left(\frac{N}{a}\right)^{-\frac{1}{2}-it}
	\frac{2(-1)^{\frac{k}{2}} 
	(4\pi)^{-\frac{k}{2}} (2\pi)^{-2ir_1}
	\prod_{p\mid N}(1-p^{-1-2ir_1})}
	{\zeta^*(1+2it)\prod_{p\mid N} (1-p^{-1-2it})}
	\Gamma\left(ir_1+\frac{k}{2}+it\right) \Gamma\left(ir_1+\frac{k}{2}-it\right)
	\\
	\times
	L\left(ir_1+\frac{1}{2}, f\times E_{1/a}(*, 1/2+it) \right)
	.
\end{multline*}

\end{proof}

Recalling \eqref{e:first}, for $m\geq 1$, 
by \eqref{e:inner_ujU} and \eqref{e:inner_E1/aU},  
\begin{multline*}
	\sum_{j} 
	\frac{h(t_j) }{\cosh(\pi t_j)}
	\overline{\rho_j(m)} 
	\cL(1/2+ir, f\times u_j)
	\\
	+
	\sum_{a\mid N} 
	\frac{1}{4\pi} \int_{-\infty}^\infty 
	\frac{h(t)}{\cosh(\pi t)}
	\tau_{1/a} (1/2-it, m) 
	\frac{2 \left(\frac{N}{a}\right)^{-\frac{1}{2}-it}
	L\left(1/2+ir, f\times E_{1/a}(*, 1/2+it) \right)}
	{\zeta^*(1+2it)\prod_{p\mid N} (1-p^{-1-2it})}
	\; dt
	\\
	=
	(4\pi)^{-\frac{k-1}{2}} 
	(2\pi)^{2ir}
	\prod_{p\mid N}(1-p^{-1-2ir})^{-1}
	\\
	\times
	\left\{
	M_{f, E_N^{(k)}(*, 1/2+\overline{ir})}(m) 
	+ E^{(1)}_{f, E_N^{(k)}(*, 1/2+\overline{ir})}(m) 
	+ E^{(2)}_{f, E_N^{(k)}(*, 1/2+\overline{ir})}(m)
	\right\}
	.
\end{multline*}
Let 
$$
	M_f(m; r)
	:=
	(4\pi)^{-\frac{k-1}{2}} 
	(2\pi)^{2ir}
	\prod_{p\mid N}(1-p^{-1-2ir})^{-1}
	M_{f, E_N^{(k)}(*, 1/2+\overline{ir_1})}(m), 
$$
$$
	E^{(1)}_f(m; r)
	:=
	(4\pi)^{-\frac{k-1}{2}} 
	(2\pi)^{2ir}
	\prod_{p\mid N}(1-p^{-1-2ir})^{-1}
	E^{(1)}_{f, E_N^{(k)}(*, 1/2+\overline{ir})}(m), 
$$
and
$$
	E^{(2)}_f(m; r)
	:=
	(4\pi)^{-\frac{k-1}{2}} 
	(2\pi)^{2ir}
	\prod_{p\mid N}(1-p^{-1-2ir})^{-1}
	E^{(2)}_{f, E_N^{(k)}(*, 1/2+\overline{ir})}(m).
$$
Then we get Theorem \ref{thm:first_Eisenstein}. 


\subsection{A proof of Theorem \ref{thm:upperbound_first}}
	
Setting $r=0$ and recalling Theorem \ref{thm:first_Eisenstein}, 
\begin{multline*}
	\sum_{j} 
	\frac{h(t_j) }{\cosh(\pi t_j)}
	\overline{\rho_j(m)} 
	\cL(1/2, f\times u_j)
	\\
	+
	\sum_{a\mid N} 
	\frac{1}{4\pi} \int_{-\infty}^\infty 
	\frac{h(t)}{\cosh(\pi t)}
	\tau_{1/a} (1/2-it, m) 
	\frac{2 \left(\frac{N}{a}\right)^{-\frac{1}{2}-it}
	L\left(1/2, f\times E_{1/a}(*, 1/2+it) \right)}
	{\zeta^*(1+2it)\prod_{p\mid N} (1-p^{-1-2it})}
	\; dt
	\\
	=
	M_f(m; 0) + E^{(1)}_f(m; 0) + E^{(2)}_f (m; 0)
	.
\end{multline*}

In this context we specialize $h(t)$ to
$$
h_{T, \alpha} (t) = \left(e^{\left(\frac{t-T}{T^\alpha} \right)^2}+ e^{\left(\frac{t+T}{T^\alpha} \right)^2}\right)
	\frac{t^2+\frac{1}{4}}{t^2+R}, 
$$
for $1\ll R < T^2$. 
We first estimate $E_f^{(1)}(m; 0)$:
\begin{multline*}
	E^{(1)}_f(m; 0)
	=
	-
	4  
	\frac{1}{\pi \prod_{p\mid N}(1-p^{-1})}
	\frac{1}{2\pi i} \int_{(\sigma_u = 1/2+2\epsilon)} 
	\frac{h(u/i) u \tan(\pi u)}
	{\Gamma\left(\frac{k}{2}+u\right) \Gamma\left(\frac{k}{2}-u\right)}
	\\
	\times
	\frac{1}{2\pi i } \int_{(\sigma_w = 1/2-k/2-\epsilon)}
	\Gamma\left(-w+u\right) \Gamma\left(-w-u\right)
	\Gamma\left(w+\frac{k}{2}\right) \Gamma\left(w+\frac{k}{2}\right)
	\\
	\times
	m^{w}
	\sum_{n=1}^{m-1} 
	\frac{a(m-n) \sigma^N_{0}(n) }{n^{w+\frac{k}{2}}} 
	\; dw\; du
	.
\end{multline*}
The exponential contribution from the gamma factors is given by a factor $\exp(C)$, where
\begin{multline*}
	C=-\frac{\pi}{2} \left(-2\left|\gamma_u\right| + 2\max\left(\left|\gamma_w\right|, \left|\gamma_u\right|\right)
	+2\left|\gamma_w\right| \right)
	=
	-\frac{\pi}{2} 
	\left\{\begin{array}{lll}
	2\left|\gamma_w\right|, & \text{ if } \left|\gamma_u\right| \geq \left|\gamma_w\right|\\
	2\left|\gamma_w\right| + 2\left(\left|\gamma_w\right|-\left|\gamma_u\right|\right), 
	& \text{ if } \left|\gamma_w\right| \geq \left|\gamma_u\right| .
	\end{array}\right.
\end{multline*}
Thus there is exponential decay unless $\left|\gamma_w\right| \ll 1$. 
By the definition of $h$ the $u$ variable $u = \sigma_u + i\gamma_u$ is essentially constrained to the interval
$|\gamma_u-T| < T^\alpha$, so, bounding the polynomial piece of the gamma estimates via Stirling and performing the $u$ and $w$ integrations we have
\begin{multline*}
	E_f^{(1)}(m; 0)
	\ll 
	N^{\epsilon} T^{1+\alpha} T^{-2\sigma_w-k}
	\sum_{n \leq m-1} 
	\frac{m^{\frac{1}{2}-\frac{k}{2}-\epsilon} (m-n)^{\frac{k-1}{2}} n^\epsilon}
	{n^{\frac{1}{2}-\epsilon}}
	\\
	=
	N^{\epsilon} T^{\alpha+2\epsilon} 
	\sum_{n \leq m-1} 
	\frac{m^{\frac{1}{2}-\frac{k}{2}-\epsilon} (m-n)^{\frac{k-1}{2}} n^\epsilon}
	{n^{\frac{1}{2}-\epsilon}}
	\ll m^{-\epsilon} \sum_{n=1}^{m-1} \frac{1}{n^{\frac{1}{2}-2\epsilon}}
	\ll m^{\frac{1}{2}+\epsilon}
	.
\end{multline*}
Thus
$$
	E_f^{(1)}(m; 0)
	\ll N^\epsilon T^{\alpha+2\epsilon} m^{\frac{1}{2}+\epsilon}. 
$$

We now estimate $E_f^{(2)}(m; 0)$:  
\begin{multline*}
	E^{(2)}_f(m;0)
	=
	\frac{4 i^k }{\pi\prod_{p\mid N}(1-p^{-1})} 
	\frac{1}{2\pi i} \int_{(\sigma_u = 1/2+2\epsilon)} 
	\frac{h(u/i) u 
	\Gamma\left(\frac{1}{2}+u\right) \Gamma\left(\frac{1}{2}-u\right)}
	{\Gamma\left(\frac{k}{2} +u\right) \Gamma\left(\frac{k}{2}-u\right)}
	\\
	\times
	\frac{1}{2\pi i } \int_{(\sigma_w = 1/2+\epsilon)}
	\frac{\Gamma\left(-w+u\right)
	\Gamma\left(w+\frac{k}{2}\right) \Gamma\left(w+\frac{k}{2}\right)} 
	{\Gamma\left(1+w+u\right)}
	m^{w} \sum_{n=1}^\infty \frac{a(m+n) \sigma^N_{0}(n) }
	{n^{w+\frac{k}{2}}}
	\; dw\; du
	.
\end{multline*}

The exponential piece is dealt with as follows: As $T\approx \gamma_u$, the ratio of gamma factors has an exponential contribution
\begin{multline*}
	-\frac{\pi}{2} 
	\left(\left|\gamma_u-\gamma_w\right| + 2\left|\gamma_w\right| 
	- \left|\gamma_w+\gamma_u\right|\right)
	=
	-\frac{\pi}{2}
	\left\{\begin{array}{llll}
	0, & \text{ if } \gamma_w >0 \text{ and } T\geq |\gamma_w|,\\
	2\left(\left|\gamma_w\right|-T\right), & \text{ if } \gamma_w> 0 \text{ and } T\leq \left|\gamma_w\right| , \\
	4|\gamma_w|, & \text{ if } \gamma_w < 0 \text{ and } T\geq \left|\gamma_w\right|, \\
	2\left|\gamma_w\right| + 2T, &\text{ if } \gamma_w<0 \text{ and } T\leq \left|\gamma_w\right|.
	\end{array}
	\right.
\end{multline*}

Thus there is exponential decay except for the case when $\gamma_w>0$ and 
$T\geq \left|\gamma_w\right|$.   In this case, the polynomial contribution of the ratio of gamma factors
$\Gamma(-w+u)/\Gamma(1 + w +u)$ is bounded above by
$$
\frac{\left(T-\gamma_w\right)^{\sigma_u-\sigma_w -\frac12}}{\left(T+\gamma_w\right)^{\sigma_u+\sigma_w+\frac12}}	
.
$$
If $\left|\gamma_w\right| > T^{1-\alpha+\epsilon}$, then 
$$
	\left(\frac{T-\gamma_w}{T+\gamma_w}\right)^{\sigma_u}
	=
	\left(\frac{1-\frac{\gamma_w}{T}}{1+\frac{\gamma_w}{T}}\right)^{\sigma_u}
	\ll \left(1-\frac{T^{\epsilon}}{T^{\alpha}}\right)^{\sigma_u}
	.
$$
The real part of $u$ can be moved to $\sigma_u=cT^{\alpha-\frac{\epsilon}{2}}$, for fixed $c>0$, and $h(u/i)$ will still decay exponentially outside of the interval $|\gamma_u-T| < T^\alpha$.   Also the Stirling approximation will remain valid.  Thus 

$$
	\left(\frac{T-\gamma_w}{T+\gamma_w}\right)^{\sigma_u}\ll \left(1-\frac{T^{\frac{\epsilon}{2}}}{T^{\alpha-\frac{\epsilon}{2}}} \right)^{cT^{\alpha-\frac{\epsilon}{2}}} 
	\ll e^{-cT^{\frac{\epsilon}{2}}}
	,
$$
which is smaller than any polynomially decaying power of $T$.
We may thus assume that $|\gamma_w| \leq T^{1-\alpha+\epsilon}$. 

We separate the series into two pieces:
$$
	m^{w} \sum_{n=1}^\infty \frac{a(m+n) \sigma^N_{0}(n) }
	{n^{w+\frac{k}{2}}}
	=
	m^{w} \sum_{n=m}^\infty \frac{a(m+n) \sigma^N_{0}(n) }
	{n^{w+\frac{k}{2}}}
	+
	m^{w} \sum_{n=1}^{m-1} \frac{a(m+n) \sigma^N_{0}(n) }
	{n^{w+\frac{k}{2}}}
	.
$$

For the first piece 
we keep $\sigma_w = 1/2+\epsilon$. 
Then 
$$
m^{w} \sum_{n=m}^\infty \frac{a(m+n) \sigma^N_{0}(n) }
	{n^{w+\frac{k}{2}}}	\ll
	N^\epsilon m^{\frac{1}{2}+\epsilon}
	\sum_{n=m}^\infty \frac{n^{\frac{k-1}{2}+\frac{\epsilon}{2}}}{n^{\frac{1}{2}+\frac{k}{2}+\epsilon}}
	=
	N^\epsilon m^{\frac{1}{2}+\epsilon} 
	\sum_{n=m}^\infty \frac{1}{n^{1+\frac{\epsilon}{2}}} 
	\ll N^\epsilon m^{\frac{1}{2}+\epsilon}
	.
$$

For the second piece, move the $w$ line of integration to $\sigma_w = 1/2-k/2 - \epsilon$, 
then we have
$$
	m^{w} \sum_{n=1}^{m-1} \frac{a(m+n) \sigma^N_{0}(n) }
	{n^{w+\frac{k}{2}}}
	\ll N^\epsilon
	m^{\frac{1}{2}-\frac{k}{2}-\epsilon +\frac{k-1}{2}+\frac{\epsilon}{2}}
	\sum_{n=1}^{m-1} \frac{1}{n^{\frac{1}{2}-\epsilon}}
	\ll N^\epsilon
	m^{\frac{1}{2}+\epsilon}
	.
$$

After performing the $u$ and $w$ integrations we have, (with a different $\epsilon$)
$$
	E^{(2)}_f(m; 0)
	\ll
	N^\epsilon
	m^{\frac{1}{2}+\epsilon}
	T^{1+\epsilon}
	.
$$

For the main term, we have
\begin{multline*}
	M_f(m;0)
	=
	\frac{A(m)}{m^{\frac{1}{2}}} 
	\left\{
	\frac{1}{\pi^2} \int_{-\infty}^\infty h(t) t\tanh(\pi t)\left(\psi\left(\frac{k}{2}+it\right) + \psi\left(\frac{k}{2}-it\right)\right) \; dt
	\right.
	\\
	\left.
	+
	\left(-2\log(2\pi) - \sum_{p\mid N} \frac{\log p}{p-1}
	-\log m +2\gamma_0\right)
	\frac{1}{\pi^2} \int_{-\infty}^\infty h(t) t \tanh(\pi t) \; dt 
	\right\}
	\\
	\approx
	\frac{a(m)}{m^{\frac{k}{2}}} T^{1+\alpha}\log T
	.
\end{multline*}

\subsection{A proof of Corollary \ref{cor:determine}}\label{ss:proof_cor_detrmine}

The left hand side of Theorem~\ref{thm:upperbound_first} consists of two pieces, corresponding to the discrete and the continuous spectrum.   Apply the theorem to both $f$ and $g$, and set the difference of the two left hand sides equal to the difference between the two right hand sides.
For sufficiently large $T$, assume that 
$$
	\cL(1/2, f\times u_j) = \cL(1/2, g\times u_j), 
$$
for $|t_j| < 2T$.   Then the discrete contribution will vanish, up to a small error term.   The continuous contribution is
$$
\sum_{a\mid N} 
	\frac{1}{4\pi} \int_{-\infty}^\infty 
	\frac{h_{T, \alpha} (t)}{\cosh(\pi t)}
	\tau_{1/a} (1/2-it, m) 
	\frac{2 \left(\frac{N}{a}\right)^{-\frac{1}{2}-it}
	L\left(1/2, f\times E_{1/a}(*, 1/2+it) \right)}
	{\zeta^*(1+2it)\prod_{p\mid N} (1-p^{-1-2it})}
	\; dt
$$
By \eqref{e:tau1/a_neq0},  $\tau_{1/a} (1/2-it, m) \ll (tmN)^\epsilon$.  
We also have
$$
	L(s, f\times E_{1/a}(*, 1/2+it))
	=
	\sum_{n=1}^\infty \frac{A(n) \sigma^a_{-2it}(n) n^{it}}{n^s}
	.
$$
By \eqref{e:sigmaa}, 
for a prime $p\nmid N$, for $r\geq 0$, we have
$$
	\sigma^a_{-2it}(p^r)p^{rit}
	=
	\frac{p^{(r+1)it}-p^{-(r+1)it}}{p^{it}-p^{-it}}, 
$$
for a prime $p\mid a$, for $r\geq 0$, we have
$$
	\sigma^a_{-2it}(p^r)p^{rit}
	=
	\frac{p^{-2it-1}}{1-p^{-2it}} p^{rit} \left(p-p^{-r(2it) +1} -1 +p^{-(r+1) 2it }\right), 
$$
and for a prime $p\mid N$, $p\nmid a$, for $r\geq 0$, we have
$$
	\sigma^a_{-2it}(p^r) p^{rit}
	=p^{rit}. 
$$
Assume that $f$ is a new form. Then 
\begin{multline*}
	L(s, f\times E_{1/a}(*, 1/2+it))
	=
	\prod_{p\mid a} (A(p)p^{-it-s} - p^{-2it-1})
	\prod_{p\mid N/a} (1-A(p)p^{-it-s})
	L(s+it, f) L(s-it, f)
\end{multline*}
and the correction polynomial at $s=1/2$ is
$$
	\prod_{p\mid a} (A(p)p^{-it-\frac{1}{2}} - p^{-2it-1})
	\prod_{p\mid N/a} (1-A(p)p^{-it-\frac{1}{2}})
	\ll (N(1+|t|))^\epsilon.
$$

 Taking $\alpha=1$ in Theorem \ref{thm:upperbound_first}, 
 it is a simple matter to show that  continuous contribution is 
 $\mathcal{O}(m^\epsilon N^{1/2+\epsilon} T^{\alpha + \epsilon})$. 
 This is because, in addition to the two above estimates for $\tau_{1/a} (1/2-it, m)$ and the correction polynomial, we also have the well known weak Lindel\"of on average result
 $$
 \int_0^T L(1/2+it, f) L(1/2-it, f)dt \ll (N^{1/2}T)^{1+\epsilon}.
 $$
 (This follows after using the approximate functional equation to write $ L(1/2+it, f)$ as a Dirichlet polynomial of length $\sqrt{N} T$, then integrating with respect to $t$, and dividing the sum into diagonal, close to diagonal and far from diagonal pieces.)
 
 Applying this estimate we have, assuming that $\cL(1/2, f\times u_j) = \cL(1/2, g\times u_j)$ for 
$|t_j| < 2T$,

\begin{multline*}
	A(m) 
	\frac{1}{\pi^2} \int_{-\infty}^\infty h_{T, 1} (t) t\tanh(\pi t) \left(\psi\left(\frac{k_f}{2}+it\right) + \psi\left(\frac{k_f}{2}-it\right)\right) \; dt
	\\
	-
	B(m)
	\frac{1}{\pi^2} \int_{-\infty}^\infty h_{T, 1} (t) t\tanh(\pi t) \left(\psi\left(\frac{k_g}{2}+it\right) + \psi\left(\frac{k_g}{2}-it\right)\right) \; dt
	\\
	+
	(A(m)-B(m))
	\left(-2\log(2\pi) -\sum_{p\mid N} \frac{\log p}{p-1} -\log m +2\gamma_0\right)
	\frac{1}{\pi^2} \int_{-\infty}^{\infty} h_{T, 1}(t) t \tanh(\pi t) \; dt
	\\
	=
	\cO\left( N^\epsilon m^{1+\epsilon} T^{1+\epsilon}\right) +\mathcal{O}(m^\epsilon N^{1/2 +\epsilon} T^{1+ \epsilon})
	.
\end{multline*}
Since 
$$
	\psi(n+it) = \frac{1}{it}+\frac{1}{it+1} + \cdots + \frac{1}{n-1+it} + \psi(it), 
$$
for $n=1, 2, \ldots$, we get
\begin{multline*}
	\left|A(m) -B(m)\right|
	\left|\frac{1}{\pi^2} \int_{-\infty}^\infty h_{T, 1} (t) t\tanh(\pi t) \left(\psi\left(it\right) + \psi\left(-it\right)\right) \; dt
	\right.
	\\
	\left.
	+
	\left(-2\log(2\pi) -\sum_{p\mid N} \frac{\log p}{p-1} -\log m +2\gamma_0\right)
	\frac{1}{\pi^2} \int_{-\infty}^{\infty} h_{T, 1}(t) t \tanh(\pi t) \; dt
	\right|
	\\
	\ll |A(m)-B(m)| T^2\log T
	\ll N^\epsilon m^{1+\epsilon} T^{1+\epsilon} +m^\epsilon N^{1/2 +\epsilon} T^{1+ \epsilon}
	.
\end{multline*}
Therefore, we obtain
$$
	\left|A(m)-B(m)\right| \ll N^{1/2+\epsilon} m^{1+\epsilon'}T^{-1+\epsilon''}. 
$$

Although the proof of Theorem 1 in \cite{Sen04} is only given for level 1, it immediately generalizes to  as follows: Taking $Q\gg_\epsilon (N(k+1))^{2+\epsilon}$, and any $\epsilon' >0$.  If 
$$
	\left| A(m)-B(m)\right| \ll Q^{-\epsilon'}
$$
for all $m\ll Q$, then $f=g$. 
It follows that if $T\gg N^{1/2+\epsilon}Q^{1+\epsilon}$ then
we can conclude that $f=g$.

\section{Shifted double Dirichlet series}\label{s:SDDS}

Let $\phi_1$ and $\phi_2$ be automorphic forms of weight $k$ for $\Gamma$, of type $\nu_1$ and $\nu_2$ respectively, with the Fourier expansions as in \eqref{e:Fourier_phi}. Then $1/4-\nu_i^2$ is the Laplace eigenvalue of $\phi_i$. 
We further assume that $c_{\phi_1}(1)\neq 0$ and $c_{\phi_2}(1)\neq 0$, and 
$$
	C_{\phi_1}(n)c_{\phi_1}(1):=c_{\phi_1}(n) 
	\text{ and } 
	C_{\phi_2}(n)c_{\phi_2}(1) := c_{\phi_2}(n), 
$$
for $n\geq 1$. 

Recalling \eqref{e:Dfphi},  for $\Re(w)>1$, we have
$$
	D_{f,\phi_1}(w; m) = \sum_{n\geq 1} \frac{a(n+m) \overline{c_{\phi_1}(n)}}{n^{w+\frac{k}{2}-1}}
$$
and for $\Re(s)>1$, we have 
$$
	D_{f, \phi_2}(s; n) = \sum_{m\geq 1} \frac{a(n+m) \overline{c_{\phi_2}(m)}}{m^{s+\frac{k}{2}-1}} 
	.
$$

Let 
\be\label{e:s'}
	s' := s+w+\frac{k}{2}-1.
\ee
For $\Re(s), \Re(w)>1$, define
\be\label{e:cMfphi1phi2}
	\cM_{f, \phi_1, \phi_2} \left(s, w\right)
	:=
	\zeta\left(2s'\right)
	\sum_{m, n\geq 1} \frac{a(n+m) \overline{c_{\phi_1}(n)} \overline{c_{\phi_2}(m)}}
	{m^{s+\frac{k}{2}-1} n^{w+\frac{k}{2}-1} }
	.
\ee
Then 
$$
	\cM_{f, \phi_1, \phi_2}(s, w)
	=
	\zeta(2s') \sum_{m=1}^\infty \frac{\overline{c_{\phi_2}(m)}}{m^{s+\frac{k}{2}-1}} D_{f, \phi_1}(w; m)
	=
	\zeta(2s') \sum_{n=1}^\infty \frac{\overline{c_{\phi_1}(n)}}{n^{w+\frac{k}{2}-1}}
	D_{f, \phi_2}(s; n)
	.
$$
In this section we will prove
\begin{thm}\label{thm:M12}
The function $\cM_{f, \phi_1, \phi_2}(s, w)$ defined in \eqref{e:cMfphi1phi2} satisfies
$\cM_{f, \phi_1, \phi_2}(s, w) = \cM_{f, \phi_2, \phi_1}(w, s)$ and
has a meromorphic continuation to $\C^2$.   Let
\be\label{e:cM*}
	\cM^*_{f, \phi_1, \phi_2}(s, w)
	:=
	G_{f, \phi_2}(s)
	\cM_{f, \phi_1, \phi_2}(s, w)
	+
	G_{f, \phi_2}(s) \cM^{(3)}_{f, \phi_1, \phi_2}(s, w), 
\ee
where the ratio of gamma functions $G_{f, \phi_2}(s)$ is given in \eqref{e:Gfphi}, and $\cM^{(3)}_{f, \phi_1, \phi_2}(s, w)$
is given by the absolutely convergent spectral expansion
\begin{multline}\label{e:cM3_series} 
	\cM^{(3)}_{f, \phi_1, \phi_2}(s, w)
	:=
	\frac{\Gamma(s)\Gamma(1-s)(4\pi)^{\frac{k}{2}-\frac{1}{2}}\sqrt{\pi} }
	{\Gamma\left(\frac{1}{2}+\frac{k}{2}+\bnu_2\right)
	\Gamma\left(\frac{1}{2}+\frac{k}{2}-\bnu_2\right)
	\Gamma\left(w+\frac{k}{2}+\bnu_1-\frac{1}{2}\right) 
	\Gamma\left(w+\frac{k}{2}-\bnu_1-\frac{1}{2}\right)}
	\\
	\times
	\overline{c_{\phi_2}(-1)}
	\left\{
	\sum_j 
	\left<u_j, U_{f, \phi}\right> 
	\Gamma\left(w-\frac{1}{2}+it_j\right) \Gamma\left(w-\frac{1}{2}-it_j\right)
	\cL\left(s', \overline{u_j} \times \overline{\phi_2}\right)
	\right.
	\\
	+
	\sum_{\cuspa} \frac{1}{4\pi} 
	\int_{-\infty}^\infty \left<E_{\cuspa}(*, 1/2+it), U_{f, \phi_1}\right> 
	\Gamma\left(w-\frac{1}{2}+it\right) \Gamma\left(w-\frac{1}{2}-it\right)
	\\
	\times
	\left.
	\cL\left(s', E_\cuspa(*,1/2-it)\times \overline{\phi_2}\right)
	\; dt
	\right\}
	, 
\end{multline}
where
$$
	\cL\left(s', \overline{u_j} \times \overline{\phi_2}\right)
	:=
	\zeta(2s')
	\sum_{m=1}^\infty \frac{\overline{\rho_j(m)} \overline{C_{\phi_2}(m)}}{m^{s+w+\frac{k}{2}-\frac{3}{2}}}
$$
and
$$
	\cL\left(s', E_\cuspa(*,1/2-it)\times \overline{\phi_2}\right)
	:=
	\zeta(2s')
	\sum_{m=1}^\infty \frac{\tau_{\cuspa}(1/2-it, m) \overline{C_{\phi_2}(m)}}{m^{s+w+\frac{k}{2}-\frac{3}{2}}}
	.
$$

 In the region  $\Re(w) > -\Re(s)/2+7/4- k/4$ and $\Re(s) < 1/2-k/2$, the following spectral expansion of $\cM^*_{f, \phi_1, \phi_2}(s, w)$ is absolutely convergent:
 
 \be\label{e:cM*}
	\cM^*_{f, \phi_1, \phi_2}(s, w)	=
	\cS_{\cusp, f, \phi_1, \phi_2} (s, w) 
	+
	\cS_{\cont, \Int, f, \phi_1, \phi_2}(s, w)
	+
	\Phi_{f, \phi_1, \phi_2}(s, w),
\ee
where
\be\label{e:cScusp}
	\cS_{\cusp, f, \phi_1, \phi_2}(s, w)
	:=
	(2\pi)^{\frac{1}{2}-s}
	\overline{c_{\phi_1}(1)}
	\sum_{j} 
	M(s, t_j) 
	(-1)^{\alpha_j}
	\cL(s', \overline{u_j} \times \overline{\phi_1}) 
	\left<u_j, U_{f, \phi_2}\right>
	, 
\ee
\begin{multline}\label{e:cScontInt}
	\cS_{\cont, \Int, f,\phi_1, \phi_2}(s, w)
	\\
	:=
	(2\pi)^{\frac{1}{2}-s}
	\overline{c_{\phi_1}(1)}
	\sum_{\cuspa} \frac{1}{4\pi} \int_{-\infty}^\infty 
	M(s, t)
	\cL(s', E_{\cuspa}(*, 1/2-it)\times \overline{\phi_1})
	\left<E_{\cuspa}(*, 1/2+it), U_{f, \phi_2}\right>
	\; dt
\end{multline}
and
\begin{multline}\label{e:Phi}
	\Phi_{f, \phi_1, \phi_2}(s, w)
	:=
	\zeta(2s') 
	\sum_{n=1}^\infty \frac{\overline{c_{\phi_1}(n)}}{n^{w+\frac{k}{2}-1}}
	\Omega_{f, \phi_2}(s; n)
	\\
	=
	-(4\pi)^{\frac{1}{2}-s} \sqrt{\pi} \Gamma(1-s) 
	\sum_{\ell=0}^{\lfloor -\Re(s)+\frac{1}{2}\rfloor} 
	\frac{(-1)^{\ell}}{\ell!} \frac{\Gamma\left(\ell+2s-1\right)}{\Gamma\left(1-\ell-s\right) \Gamma\left(\ell+s\right)}
	\\
	\times
	\frac{1}{2} \overline{c_{\phi_1}(1)} \sum_{\cuspa}
	\left\{
	\cL\left(s', E_\cuspa(*, 1-\ell-s)\times \overline{\phi_1}\right)
	\left<E_\cuspa (*, \ell+s), U_{f, \phi_2}\right>
	\right.
	\\
	\left.
	+
	\cL\left(s', E_\cuspa(*, \ell+s) \times \overline{\phi_1}\right)
	\left<E_\cuspa (*, 1-\ell-s), U_{f, \phi_2}\right>
	\right\}.
\end{multline}
Here
$$
	M(s, t)
	=
	\frac{\sqrt{\pi} 2^{\frac{1}{2}-s} 
	\Gamma\left(s-\frac{1}{2}-it\right) \Gamma\left(s-\frac{1}{2}+it\right)
	\Gamma(1-s)}
	{\Gamma\left(\frac{1}{2}-it\right) \Gamma\left(\frac{1}{2}+it\right)}.
$$
The poles of $\cM^*_{f, \phi_1, \phi_2}(s, w)$ come from the gamma factors of $\cS_{\cusp, f, \phi_1, \phi_2}(s, w)$, the gamma factors of $\cS_{\cusp, f, \phi_2, \phi_1}(w,s)$, and from $\Phi_{f, \phi_1, \phi_2}(s, w)$.  These are the polar lines:
\begin{itemize}
	\item $s=1/2 \pm it_j -r, \, r \ge 0$;
	\item $w=1/2 \pm it_j -r, \, r \ge 0$;
	\item $2s = 1-r,\, r \ge 0$;
	\item $2w = 1-r,\, r \ge 0$;
	\item $s= -r,\, r \ge 0$;
	\item $w= -r,\, r \ge 0$;
	\item $s' = 1 \pm \nu_1 \pm \nu_2$.
\end{itemize}

\end{thm}
\begin{proof}

For $\Re(s)< 1/2-k/2$ and for sufficiently large $\Re(w)$, 
we may substitute the spectral expansion for $D_{f, \phi_2}(s;n)$ given in Theorem \ref{thm:Dfphi} and bring the sum over $n\geq 1$ inside. 
The sum over $n$ introduces the $L$-series 
$(-1)^{\alpha_j} \cL(s', \overline{u_j} \times \overline{\phi_1})$ and 
$\cL(s', E_\cuspa(*, 1/2-it)\times \overline{\phi_1})$, 
for $\Re(s')>1$, 
giving us 
\begin{multline}\label{e:cM*2}
	\cM^*_{f, \phi_1, \phi_2}(s, w)
	:=
	G_{f, \phi_2}(s)
	\cM_{f, \phi_1, \phi_2}(s, w)
	+
	G_{f, \phi_2}(s) \cM^{(3)}_{f, \phi_1, \phi_2}(s, w)
	\\
	=
	\cS_{\cusp, f, \phi_1, \phi_2} (s, w) 
	+
	\cS_{\cont, \Int, f, \phi_1, \phi_2}(s, w)
	+
	\Phi_{f, \phi_1, \phi_2}(s, w),
\end{multline}
where $\cS_{\cusp, f, \phi_1, \phi_2} (s, w)$ is given in \eqref{e:cScusp},
$\cS_{\cont, \Int, f,\phi_1, \phi_2}(s, w)$ is given in \eqref{e:cScontInt},
$\Phi_{f, \phi_1, \phi_2}(s, w)$ is given in \eqref{e:Phi},
and 
\be\label{e:cM3_series1}
	\cM^{(3)}_{f, \phi_1, \phi_2}(s, w)
	=
	\frac{\overline{c_{\phi_2}(-1)} 
	\Gamma(s) \Gamma\left(1-s\right)}
	{\Gamma\left(\frac{1}{2}+\frac{k}{2}+\bnu_2\right)
	\Gamma\left(\frac{1}{2}+\frac{k}{2}-\bnu_2\right)}
	\zeta(2s') \sum_{n=1}^\infty \frac{\overline{c_{\phi_1}(n)}}{n^{w+\frac{k}{2}-1}}
	\sum_{m=1}^{n-1} \frac{a(n-m) \overline{C_{\phi_2}(m)}}{m^{s+\frac{k}{2}-1}}
	.
\ee
The expression \eqref{e:cM3_series1} will shortly be transformed into \eqref{e:cM3_series}. 

We will now determine a region in which the expression given in \eqref{e:cM*} converges absolutely. 

First, for $\Re(s')>1$ and $\Re(s) < 1/2-k/2$, the series $\cS_{\cusp, f, \phi_1, \phi_2}(s, w)$ and integral $\cS_{\cont, \Int, f, \phi_1, \phi_2}(s, w)$ converge absolutely, by the same argument given in Theorem \ref{thm:Dfphi}. 

For $\Re(s') < 0$, by the functional equation, we have
$$
	e^{-\frac{\pi|t_j|}{2}} \cL(s', \overline{u_j} \times \overline{\phi_1}) \ll |t_j|^{2-4s'}
	.
$$
As given in Proposition 4.1 in \cite{HHR}, we have 
$$
	\sum_{|t_j|\sim T} e^{\frac{\pi|t_j|}{2}} 
	\left<u_j, U_{f, \phi_1}\right>
	\ll
	\left(\sum_{|t_j|\sim T} 1\right)^{\frac{1}{2}}
	\left(\sum_{|t_j|\sim T} e^{\pi|t_j|} \left|\left<u_j, U_{f, \phi_2}\right>\right|^2\right)^{\frac{1}{2}}
	\ll T^{1+k}
	.
$$
The ratio of the gamma factors satisfies
$$
	M(s, t_j) \ll |t_j|^{2\Re(s)-2}
	. 
$$
Thus, 
$$
	\sum_{|t_j|\sim T} 
	M(s, t_j) \cL (s', \overline{u_j} \times \overline{\phi_1}) \left<u_j, U_{f, \phi_2}\right>
	\ll 
	T^{2\Re(s)+1+k-4\Re(s')}, 
$$
and the series $\cS_{\cusp, f, \phi_1, \phi_2}(s, w)$ converges absolutely for $\Re(s) < 2\Re(s')-1/2$ and $\Re(s') < 0$, i.e., $\Re(w) > -\Re(s)/2 + (5-k)/4$ and $\Re(w) < -\Re(s) + 1- k/2$.  
Applying the convexity principle to estimate 
$e^{-\frac{\pi |t_j|}{2}} \cL(s', \overline{u_j} \times \overline{\phi_1})$ 
in the region $0< \Re(s') <1$, these conditions translate to the condition that $\cS_{\cusp, f, \phi_1, \phi_2}(s, w)$ converges absolutely in the region defined by $\Re(s) < 1/2-k/2$ and $\Re(w) > -\Re(s)/2 + 7/4 - k/4$. 
The integral $\cS_{\cont, \Int, f, \phi_1, \phi_2}(s, w)$ converges in a larger region than this. 
There is a possible pole at $s'=1$ of $\cL(s', E_{\cuspa}(*, 1/2-it)\times \overline{\phi_1})$. 
Thus for $\Re(s') <1$, the deviation of the meromorphic continuation of $\cS_{\cont, \Int, f, \phi_1, \phi_2}(s, w)$ from the strict integral definition of $\cS_{\cont, \Int, f, \phi_1, \phi_2}(s, w)$  is given by 
$\Phi_{f, \phi_1, \phi_2}(s, w)$, which is described in \eqref{e:Phi}.
For $\Re(s)<1/2-k/2$ and $\Re(s')<1$, we have
\begin{multline*}
	\cS_{\cont, \Int, f,\phi_1, \phi_2}(s, w)
	=
	M\left( s, (s-1+\bnu_1)/i\right)
	\\
	\times
	\overline{c_{\phi_1}(1)}
	\sum_\cuspa \left<E_\cuspa (*, 3/2-s-\bnu_1), U_{f, \phi_2}\right>
	\Res_{z=-s+1-\bnu_1} \cL\left(s, E_\cuspa (*,1/2-z) \times \overline{\phi_1}\right)
	\\
	+
	M\left( s, (s-1-\bnu_1)/i\right)
	\\
	\times
	\overline{c_{\phi_1}(1)}
	\sum_\cuspa \left<E_\cuspa (*, 3/2-s+\bnu_1), U_{f, \phi_2}\right>
	\Res_{z=-s+1+\bnu_1} \cL\left(s, E_\cuspa(*, 1/2-z) \times \overline{\phi_1}\right) 
	\\
	+
	(2\pi)^{\frac{1}{2}-s}
	\overline{c_{\phi_1}(1)}
	\sum_{\cuspa} \frac{1}{4\pi} \int_{-\infty}^\infty 
	M(s, t)
	\cL(s', E_{\cuspa}(*, 1/2-it)\times \overline{\phi_1})
	\left<E_{\cuspa}(*, 1/2+it), U_{f, \phi_2}\right>
	\; dt
	.
\end{multline*}

The Dirichlet series expression of $\cM_{f, \phi_1, \phi_2}(s, w)$ converges absolutely for $\Re(s), \Re(w)>1$. 
For $1/2-k/2 < \Re(s) < 1$ and $\Re(w) >3/2$, by the upper bound given \eqref{e:D*fphi_upper}, the summation 
$$
	\zeta(2s') \sum_{n=1}^\infty \frac{c_{\phi_1}(n)}{n^{w+\frac{k}{2}-1}} 
	D_{f, \phi_2}(s; n)
$$
converges absolutely. 
Therefore, after subtracting off the poles, $\cM_{f, \phi_1, \phi_2}(s, w)$ has a meromorphic continuation in the region where 
\begin{itemize}
	\item $\Re(w) > -\Re(s)/2+7/4- k/4$ and $\Re(s) < 1/2-k/2$;
	\item $\Re(s) > 1$ and $\Re(w) >1$;
	\item $1/2-k/2 < \Re(s) < 1$ and $\Re(w) >3/2$.
\end{itemize}
We name this region  $R^{(1)}$. 

Similarly, for $\Re(w) < 1/2-k/2$ and for sufficiently large $\Re(s)$, we have
\begin{multline}\label{e:cM2}
	G_{f, \phi_1}(w) \cM_{f, \phi_1, \phi_2}(s, w)
	+
	G_{f, \phi_1}(w)
	\cM^{(3)}_{f, \phi_2, \phi_1}(w, s)
	\\
	=
	\cS_{\cusp, f, \phi_2, \phi_1} (w, s) 
	+
	\cS_{\cont, \Int, f, \phi_2, \phi_1}(w, s)
	+
	\Phi_{f, \phi_2, \phi_1}(w, s)
	.
\end{multline}
Reversing the roles of $s$ and $w$, we obtain a meromorphic continuation to the corresponding area 
\begin{itemize}
	\item $\Re(s) > -\Re(w)/2+7/4- k/4$ and $\Re(w) < 1/2-k/2$;
	\item $\Re(w) > 1$ and $\Re(s) >1$;
	\item $1/2-k/2 < \Re(w) < 1$ and $\Re(s) >3/2$.
\end{itemize}
We name this region $R^{(2)}$.

The series $\cM^{(3)}_{f, \phi_1, \phi_2}(s, w)$ defined in \eqref{e:cM3_series} is originally defined for $\Re(s) < 1/2-k/2$ and sufficiently large $\Re(w)$, and the series given in \eqref{e:cM3_series} converges absolutely for any $s, w\in \C$, away from the poles from the Gamma functions. 
This series can be transformed as a Dirichlet series given in \eqref{e:cM3_series1}. 
For $\Re(s), \Re(w)>1$, interchanging the order of summation in \eqref{e:cM3_series1}, we have
\begin{multline}\label{e:cM3}
	\cM^{(3)}_{f, \phi_1, \phi_2} (s, w)
	=	
	\frac{\Gamma(s) \Gamma\left(1-s\right)}
	{\Gamma\left(\frac{1}{2}+\frac{k}{2}+\bnu_2\right)
	\Gamma\left(\frac{1}{2}+\frac{k}{2}-\bnu_2\right)}
	\zeta(2s') 
	\sum_{m=1}^\infty
	\sum_{n=m+1}^\infty 
	\frac{a(n-m)\overline{c_{\phi_1}(n)} \overline{c_{\phi_2}(-m)}}{n^{w+\frac{k}{2}-1} m^{s+\frac{k}{2}-1}} 
	.
\end{multline}

The inner sum over $n$ is essentially an inner product of $f$ and $\phi_1$ with a standard Poincar\'e series. 
In particular, for $m \geq 1$, if
$$
	P(z, w; m)
	=
	\sum_{\gamma\in \Gamma_\infty\bsl \Gamma} \left(\Im(\gamma z)\right)^{w} e^{2\pi i m \gamma z}, 
$$
then 
\begin{multline*}
	\left< P(*, w; m), U_{f, \phi_1}\right>
	=
	\frac{1}{(2\pi)^{\frac{k}{2}+w-1}} 
	\sum_{n=m+1}^\infty 
	\frac{ a(n-m) \overline{c_{\phi_1}(n)} }{n^{\frac{k}{2}+w-1}} 
	\int_0^\infty W_{\frac{k}{2}, \bnu_1} (2y) 
	e^{-y} 
	y^{\frac{k}{2}+w-1} \; \frac{dy}{y}
	\\
	=
	\frac{\Gamma\left(w+\frac{k}{2}+\bnu_1-\frac{1}{2}\right) \Gamma\left(w+\frac{k}{2}-\bnu_1-\frac{1}{2}\right)}
	{\Gamma(w) (4\pi)^{\frac{k}{2}+w-1}}
	\sum_{n=m+1}^\infty 
	\frac{ a(n-m) \overline{c_{\phi_1}(n)} }{n^{\frac{k}{2}+w-1}} 
	.
\end{multline*}

On the other hand, we also have
\begin{multline*}
	\left<P(*, w; m), U_{f, \phi}\right>
	=
	\sum_j \left<P(*, w; m), u_j\right> \left<u_j, U_{f, \phi_1}\right>
	\\
	+
	\sum_{\cuspa} \frac{1}{4\pi} \int_{-\infty}^\infty 
	\left<P(*, w; m), E_{\cuspa}(*, 1/2+it)\right>
	\left<E_{\cuspa}(*, 1/2+it), U_{f, \phi_1}\right>\; dt
	.
\end{multline*}
Then
$$
	\left<P(*, w; m), u_j\right> 
	=
	\frac{\overline{\rho_j(m)}}{(2\pi m)^{w-\frac{1}{2}}}
	\frac{\sqrt{\pi} 2^{\frac{1}{2}-w} \Gamma\left(w-\frac{1}{2}+it_j\right) \Gamma\left(w-\frac{1}{2}-it_j\right)}
	{\Gamma(w)}
$$
and
$$
	\left<P(*, w; m), E_{\cuspa}(*, 1/2+it)\right>
	=
	\frac{\tau_\cuspa\left(1/2-it, m\right)}{(2\pi m)^{w-\frac{1}{2}}}
	\frac{\sqrt{\pi} 2^{\frac{1}{2}-w} \Gamma\left(w-\frac{1}{2}+it\right) \Gamma\left(w-\frac{1}{2}-it\right)}
	{\Gamma(w)}
	.
$$

\begin{prop}\label{prop:cM3}
For $\Re(w) > 1/2$, the function $\cM^{(3)}_{f, \phi_1, \phi_2}(s, w)$ has the following expression: 
\begin{multline*}
	\cM^{(3)}_{f, \phi_1, \phi_2}(s, w)
	=
	\frac{\Gamma(s)\Gamma(1-s)}
	{\Gamma\left(\frac{1}{2}+\frac{k}{2}+\bnu_2\right)
	\Gamma\left(\frac{1}{2}+\frac{k}{2}-\bnu_2\right)}
	\zeta(2s')
	\sum_{m=1}^\infty \frac{\overline{c_{\phi_2}(-m)}}{m^{s+\frac{k}{2}-1}}
	\sum_{n=m+1}^\infty \frac{a(n-m) \overline{c_{\phi_1}(n)}}{n^{w+\frac{k}{2}-1}}
	\\
	=
	\frac{\Gamma(s)\Gamma(1-s)(4\pi)^{\frac{k}{2}-\frac{1}{2}}\sqrt{\pi} }
	{\Gamma\left(\frac{1}{2}+\frac{k}{2}+\bnu_2\right)
	\Gamma\left(\frac{1}{2}+\frac{k}{2}-\bnu_2\right)
	\Gamma\left(w+\frac{k}{2}+\bnu_1-\frac{1}{2}\right) 
	\Gamma\left(w+\frac{k}{2}-\bnu_1-\frac{1}{2}\right)}
	\\
	\times
	\overline{c_{\phi_2}(-1)}
	\left\{
	\sum_j 
	\left<u_j, U_{f, \phi}\right> 
	\Gamma\left(w-\frac{1}{2}+it_j\right) \Gamma\left(w-\frac{1}{2}-it_j\right)
	\cL\left(s', \overline{u_j} \times \overline{\phi_2}\right)
	\right.
	\\
	+
	\sum_{\cuspa} \frac{1}{4\pi} 
	\int_{-\infty}^\infty \left<E_{\cuspa}(*, 1/2+it), U_{f, \phi_1}\right> 
	\\
	\left.
	\times
	\Gamma\left(w-\frac{1}{2}+it\right) \Gamma\left(w-\frac{1}{2}-it\right)
	\cL\left(s', E_\cuspa(*,1/2-it)\times \overline{\phi_2}\right)
	\; dt
	\right\}
	.
\end{multline*}
The sum and integral converge absolutely, and consequently the function $\cM^{(3)}_{f, \phi_1, \phi_2}(s, w)$ has a meromorphic continuation to all $s, w\in \C$.
\end{prop}


We have now obtained a meromorphic continuations of $\cM_{f, \phi_1, \phi_2}(s, w)$ to the union of the regions $R^{(1)}$ and $R^{(2)}$. 
By subtracting all polar lines, we can construct a function of $s$ and $w$, which is entire in the union of $R^{(1)}$ and $R^{(2)}$. 
As a convex closure of the union of $R^{(1)}$ and $R^{(2)}$ is all of $\C^2$, we have obtained the desired meromorphic continuation of $\cM_{f, \phi_1, \phi_2}(s, w)$.   This completes the proof of Theorem~\ref{thm:M12}.

\end{proof}

\section{Applications to  second moments}\label{s:second}
Recall that
 $f$ is a holomorphic cusp form of even weight $k$ for $\Gamma$, with the Fourier expansion given in \eqref{e:f_Fourier}.
Also $\phi_1$ and $\phi_2$ are automorphic forms of weight $k$ for $\Gamma$, of type $\nu_1$ and $\nu_2$ respectively (i.e., the Laplacian eigenvalue of $\phi_1$ and $\phi_2$ are $1/4-\nu_1^2$ and $1/4-\nu_2^2$ respectively), with the Fourier expansions given in \eqref{e:Fourier_phi}. 

For applications we will be mostly interested in the specific case 
$\phi_1=E_N^{(k)}(z, 1/2+\overline{ir_1})$, which is defined in \eqref{e:EkN}, and $\phi_2 = g(z)y^{k/2}$, where $g$ is a holomorphic cusp form of even weight $k$ for $\Gamma$. 
However, before specializing to this case, we will derive a second moment formula for general $\phi_1, \phi_2$.

For an automorphic form $\phi$ of weight $k$ for $\Gamma$ with the Fourier expansion given in \eqref{e:Fourier_phi}, 
define
\be\label{e:cL-_jphi2}
	\cL(s, u_j\times \phi)
	:=
	\zeta(2s) \sum_{m=1}^\infty \frac{\rho_j(m) C_{\phi}(m)}{m^{s-\frac{1}{2}}}
\ee
and 
\be\label{e:cL_cuapaphi2}
	\cL(s, E_\cuspa(*, 1/2-it)\times \phi_2)
	:=
	\zeta(2s) \sum_{m=1}^\infty \frac{\tau_{\cuspa}(m) C_{\phi}(m)}{m^{s-\frac{1}{2}}}
	, 
\ee
where $c_{\phi}(1) C_{\phi}(m) = c_{\phi}(m)$, for $m\geq 1$.

\begin{thm}\label{thm:second}
We obtain the following: 
Assume that $\phi_2$ is cuspidal, and $c_{\phi_2}(1)\neq 0$.  
Let $\alpha_j = 1$ if $u_j$ is odd and $\alpha_j=0$ if $u_j$ is even. 
For $\Re(s)=1/2$, we have
\begin{multline}\label{e:Ifphi1phi2_general}
	I_{f, \phi_1, \phi_2}(s)
	\\
	:=
	(4\pi)^{k-\frac{1}{2}} i^k 
	\overline{c_{\phi_2}(1)}
	\left\{
	\sum_{j} 
	\frac{h(t_j) }{\cosh(\pi t_j)}
	(-1)^{\alpha_j} \cL\left(s, \overline{u_j}\times \overline{\phi_2}\right)
	\frac{\left<u_j, U_{f, \phi_1}\right>}{\Gamma\left(\frac{k}{2}+\bnu_1+it_j\right)
	\Gamma\left(\frac{k}{2}+\bnu_1-it_j\right)}
	\right.
	\\
	+
	\left.
	\sum_{\cuspa} 
	\frac{1}{4\pi} \int_{-\infty}^\infty 
	\frac{h(t)}{\cosh(\pi t)}
	\cL\left(s, E_\cuspa(*, 1/2-it)\times \overline{\phi_2}\right)
	\frac{\left<E_{\cuspa}(*, 1/2+it), U_{f, \phi_1}\right>} 
	{\Gamma\left(\frac{k}{2}+\bnu_1+it\right) \Gamma\left(\frac{k}{2}+\bnu_1-it\right)}
	\; dt
	\right\}
	\\
	=
	M_{f, \phi_1, \phi_2}(s) 
	+ 
	\Omega^{(1)}_{f, \phi_1, \phi_2}(s)
	+
	\Omega^{(3)}_{f, \phi_1, \phi_2}(s)
	\\
	+
	E^{(1)}_{f, \phi_1, \phi_2}(s)
	+
	E^{(2)}_{f, \phi_1, \phi_2}(s)
	+
	E^{(3)}_{f, \phi_1, \phi_2}(s)
	.
\end{multline}
Here
$M_{f, \phi_1, \phi_2}(s)$ is given in \eqref{e:second_main},
$\Omega^{(1)}_{f, \phi_1, \phi_2}(s)$ in \eqref{e:Omega1},
$E^{(1)}_{f, \phi_1, \phi_2}(s)$ in \eqref{e:E1},
$\Omega^{(3)}_{f, \phi_1, \phi_2}(s)$ in \eqref{e:Omega3},
$E^{(2)}_{f, \phi_1, \phi_2}(s)$ in \eqref{e:E2},
and 
$E^{(3)}_{f, \phi_1, \phi_2}(s)$ in \eqref{e:E3}.

For a square-free $N$, when
 $\phi_1(z) = E_N^{(k)}(z, 1/2+\overline{ir})$, which is a weight $k$ Eisenstein series at the cusp $1/N$, defined in \eqref{e:EkN}, taking $s=1/2-ir$, 
 this specializes to
(by \eqref{e:inner_ujU} and \eqref{e:inner_E1/aU})
\begin{multline}\label{e:second}
	\sum_{j} 
	\frac{h(t_j) }{\cosh(\pi t_j)}
	\cL\left(1/2-ir, \overline{u_j}\times \overline{\phi_2}\right)
	\cL(1/2+ir, f\times u_j)
	\\
	+
	\sum_{a\mid N} 
	\frac{1}{4\pi} \int_{-\infty}^\infty 
	\frac{h(t)}{\cosh(\pi t)}
	\frac{4\left(\frac{N}{a}\right)^{-1}}
	{\zeta^*(1+2it)\zeta^*(1-2it) \prod_{p\mid N} (1-p^{-1-2it})\left(1-p^{-1+2it}\right)}
	\\
	\times
	L\left(1/2-ir, E_{1/a}(*, 1/2-it)\times \overline{\phi_2}\right)
	L\left(1/2+ir, f\times E_{1/a}(*, 1/2+it) \right)
	\; dt
	\\
	=
	\left((4\pi)^{\frac{k}{2}-\frac{1}{2}}  
	(2\pi)^{-2ir}
	\prod_{p\mid N}(1-p^{-1-2ir}) 
	\overline{c_{\phi_2}(1)}\right)^{-1}
	I_{f, E_N^{(k)}(*, 1/2+\overline{ir}), \phi_2}(1/2-ir)
	.
\end{multline}

\end{thm}

\begin{proof}
%

By Theorem \ref{thm:first}, for $m\geq 1$, we have
\begin{multline*}
	I_{f, \phi_1}(m)
	=
	(4\pi)^{k-\frac{1}{2}} i^k 
	\sum_{j} 
	\frac{h(t_j) }{\cosh(\pi t_j)}
	\overline{\rho_j(-m)} 
	\frac{\left<u_j, U_{f, \phi_1}\right>}{\Gamma\left(\frac{k}{2}+\bnu_1+it_j\right)
	\Gamma\left(\frac{k}{2}+\bnu_1-it_j\right)}
	\\
	+
	(4\pi)^{k-\frac{1}{2}} i^k 	
	\sum_{\cuspa} 
	\frac{1}{4\pi} \int_{-\infty}^\infty 
	\frac{h(t)}{\cosh(\pi t)}
	\tau_\cuspa (1/2-it, m) 
	\frac{\left<E_{\cuspa}(*, 1/2+it), U_{f, \phi_1}\right>} 
	{\Gamma\left(\frac{k}{2}+\bnu_1+it\right) \Gamma\left(\frac{k}{2}+\bnu_1-it\right)}
	\; dt
	\\
	=
	M_{f, \phi_1}(m) + E^{(1)}_{f, \phi_1}(m) + E^{(2)}_{f, \phi_1}(m)
	, 
\end{multline*}
where $M_{f, \phi_1}(m)$, $E^{(1)}_{f, \phi_1}(m)$ and $E^{(2)}_{f, \phi_1}(m)$ are given in \eqref{e:Mfphi}, \eqref{e:E1fphi} and \eqref{e:E2fphi}, respectively. 

Taking the summation over $m\geq 1$ on both sides, we have
\begin{multline}\label{e:Ifphi1phi2}
	I_{f, \phi_1, \phi_2}(s)
	=
	\zeta(2s)
	\sum_{m=1}^\infty \frac{\overline{c_{\phi_2}(m)}}{m^{s-\frac{1}{2}}}I_{f, \phi_1}(m)
	\\
	=
	(4\pi)^{k-\frac{1}{2}} i^k 
	\overline{c_{\phi_2}(1)}
	\left\{
	\sum_{j} 
	\frac{h(t_j) }{\cosh(\pi t_j)}
	(-1)^{\alpha_j}
	\cL(s, \overline{u_j} \times \overline{\phi_2})
	\frac{\left<u_j, U_{f, \phi_1}\right>}{\Gamma\left(\frac{k}{2}+\bnu_1+it_j\right)
	\Gamma\left(\frac{k}{2}+\bnu_1-it_j\right)}
	\right.
	\\
	\left.
	+
	\sum_{\cuspa} 
	\frac{1}{4\pi} \int_{-\infty}^\infty 
	\frac{h(t)}{\cosh(\pi t)}
	\cL(s, E_{\cuspa}(*, 1/2-it)\times\overline{\phi_2})
	\frac{\left<E_{\cuspa}(*, 1/2+it), U_{f, \phi_1}\right>} 
	{\Gamma\left(\frac{k}{2}+\bnu_1+it\right) \Gamma\left(\frac{k}{2}+\bnu_1-it\right)}
	\; dt
	\right\}
	\\
	=
	\tM_{f, \phi_1, \phi_2}(s)
	+ 
	\tE^{(1)}_{f, \phi_1, \phi_2}(s)
	+
	\tE^{(2)}_{f, \phi_1, \phi_2}(s) 
	, 
\end{multline}
where
\be\label{e:tMfphi1phi2}
	\tM_{f, \phi_1, \phi_2}(s)
	:=
	\zeta(2s) \sum_{m=1}^\infty \frac{\overline{c_{\phi_2}(m)}}{m^{s-\frac{1}{2}}}
	M_{f, \phi_1}(m)
	, 
\ee
\be\label{e:tE1fphi1phi2}
	\tE^{(1)}_{f, \phi_1, \phi_2}(s)
	:=
	\zeta(2s) \sum_{m=1}^\infty \frac{\overline{c_{\phi_2}(m)}}{m^{s-\frac{1}{2}}}
	E^{(1)}_{f, \phi_1}(m)
\ee
and
\be\label{e:tE2fphi1phi2}
	\tE^{(2)}_{f, \phi_1, \phi_2}(s) 
	:=
	\zeta(2s) \sum_{m=1}^\infty \frac{\overline{c_{\phi_2}(m)}}{m^{s-\frac{1}{2}}}
	E^{(2)}_{f, \phi_1}(m)
	.
\ee


To prove Theorem \ref{thm:second}, we compute each $\tM_{f, \phi_1, \phi_2}(s)$, $\tE^{(1)}_{f, \phi_1, \phi_2}(s)$ and $\tE^{(2)}_{f, \phi_1, \phi_2}(s)$ explicitly in 
Lemma \ref{lem:tM}, Lemma \ref{lem:tE1+tE2_SDDS} and Proposition \ref{prop:tE} respectively. 
For each lemma and proposition, we get a meromorphic continuation for each term for $\Re(s)\geq1/2$. 

For $\tE^{(1)}_{f, \phi_1, \phi_2}(s)$ and $\tE^{(2)}_{f, \phi_1, \phi_2}(s)$, we separate the terms into main term and other terms. 
In Proposition \ref{prop:tE}, 
$$
	\tE^{(1)}_{f, \phi_1, \phi_2}(s)
	=
	\tM^{(1)}_{f, \phi_1, \phi_2}(s)
	+
	\Omega^{(1)}_{f, \phi_1, \phi_2}(s)
	+
	E^{(1)}_{f, \phi_1, \phi_2}(s)
$$
and 
\begin{multline*}
	\tE^{(2)}_{f, \phi_1, \phi_2}(s)
	=
	E^{(2)}_{f, \phi_1, \phi_2}(s)
	+
	E^{(3)}_{f, \phi_1, \phi_2}(s)
	+
	\Omega^{(3)}_{f, \phi_1, \phi_2}(s)
	+
	M^{(3)}_{f, \phi_1, \phi_2}(s)
	\\
	+
	2
	\sum_{\ell=1}^{\frac{k}{2}} 
	\frac{1}{(\ell-1)!}
	(4\pi)^{k-\ell}
	\frac{\Gamma\left(s-\frac{1}{2}+\ell+\bar{\nu_1}\right) 
	\Gamma\left(s-\frac{1}{2}+\ell-\bar{\nu_1}\right)}
	{\Gamma\left(\frac{k}{2}-\ell+\frac{1}{2}+\bnu_2\right)
	\Gamma\left(\frac{k}{2}-\ell +\frac{1}{2}-\bnu_2\right) }
	\cM^*_{f, \phi_1, \phi_2} \left(-\ell+1, s+\ell-\frac{k}{2}\right)
	\\
	\times 
	H_{2, \ell}(s; \bnu_1)
	, 
\end{multline*}
where $H_{2, \ell}(s; \bnu_1)$ is defined in \eqref{e:H2}. 

Collecting the main terms in Lemma \ref{lem:Mfphi1phi2}, 
$$
	M_{f, \phi_1, \phi_2}(s)
	:=
	\tM_{f, \phi_1, \phi_2}(s)
	+
	M^{(1)}_{f, \phi_1, \phi_2}(s)
	+
	M^{(3)}_{f, \phi_1, \phi_2}(s), 
$$
we will show an explicit formula in \eqref{e:second_main}.

\end{proof}

\subsection{$\tM_{f, \phi_1, \phi_2}(s)$}

Define
\be\label{e:H1}
	H_1^\pm (\bnu_1)
	:=
	\frac{1}{\pi^2} \int_{-\infty}^\infty h(t) t \tanh(\pi t) 
	\frac{\Gamma\left(\frac{k}{2}\pm\bnu_1+it\right) 
	\Gamma\left(\frac{k}{2}\pm\bnu_1-it\right)}
	{\Gamma\left(\frac{k}{2}+\bnu_1+it\right) 
	\Gamma\left(\frac{k}{2}+\bnu_1-it\right)}
	\; dt
	.
\ee
Recalling \eqref{e:Mfphi} and \eqref{e:tMfphi1phi2}, we have the following lemma.
\begin{lem}\label{lem:tM}
The function $\tM_{f, \phi_1, \phi_2}(s)$ has a meromorphic continuation to any $s\in \C$ with the following relation:
\begin{multline}\label{e:tM}
	\tM_{f, \phi_1, \phi_2} (s)
	\\
	=
	(4\pi)^{\frac{k}{2}-\bnu_1} 
	\overline{c^+_{\phi_1}}
	\frac{\Gamma\left(-\frac{k}{2}+\frac{1}{2}-\bnu_1\right) }
	{\Gamma\left(\frac{1}{2}+\bnu_1\right) \Gamma\left(\frac{1}{2}-\bnu_1\right)}
	\frac{\zeta(2s) }{\zeta(2s+2\bnu_1+1)}
	\overline{c_{\phi_2}(1)}
	L\left(s+1/2+ \bnu_1 , f\times \overline{\phi_2}\right)
	\frac{1}{2}H_1^+ (\bnu_1)
	\\
	+
	(4\pi)^{\frac{k}{2}+\bnu_1}
	\overline{c^-_{\phi_1}} 
	\frac{\Gamma\left(-\frac{k}{2}+\frac{1}{2}+\bnu_1\right)}
	{\Gamma\left(\frac{1}{2}+\bnu_1\right) \Gamma\left(\frac{1}{2}-\bnu_1\right)}
	\frac{\zeta(2s) }{\zeta(2s-2\bnu_1+1)}
	 \overline{c_{\phi_2}(1)}
	L\left(s+1/2- \bnu_1 , f\times \overline{\phi_2}\right)
	\frac{1}{2}H_1^- (\bnu_1).
\end{multline}

Taking $\phi_1=E_N^{(k)}(*, 1/2+\overline{ir_1})$ and $\phi_2(z)=g(z)y^{k/2}$, 
\begin{multline}\label{e:tMg}
	\tM_{f, E_N^{(k)}(*, 1/2+\overline{ir_1}), gy^{k/2}} (s)
	\\
	=
	\frac{(4\pi)^{-ir_1} }{\pi}
	\zeta(1+2ir_1) \pi^{-\frac{1}{2}-ir_1}
	\prod_{p\mid N} (1-p^{-1-2ir_1} )
	\frac{\zeta(2s) }{\zeta(2s+2ir_1+1)}
	L\left(s+1/2+ ir_1 , f\times \bar{g}\right)
	\frac{1}{2}H_1^+ (ir_1)
	\\
	+
	\frac{(4\pi)^{ir_1}}{\pi}
	\zeta(1-2ir_1) \pi^{-\frac{1}{2}+ir_1}
	\frac{\varphi(N)}{N^{1+2ir_1}}
	\frac{\zeta(2s) }{\zeta(2s-2ir_1+1)}
	L\left(s+1/2-ir_1 , f\times \overline{g}\right)
	\frac{1}{2}H_1^- (ir_1).
\end{multline}
\end{lem}

\subsection{$\tE^{(1)}_{f, \phi_1, \phi_2}(s)$ and $\tE^{(2)}_{f, \phi_1, \phi_2}(s)$}

\begin{lem}\label{lem:tE1+tE2_SDDS}
For $\Re(s) > 2$, we have
\begin{multline}\label{e:tE1_SDDS}
	\tE^{(1)}_{f, \phi_1, \phi_2}(s)
	=
	-
	4
	(4\pi)^{\frac{k-1}{2}} i^k 
	\frac{1}{2\pi i} \int_{(\sigma_u = 1/2+2\epsilon)} 
	\frac{h(u/i) u \tan(\pi u)}{\Gamma\left(\frac{k}{2}+\bnu +u\right) \Gamma\left(\frac{k}{2}+\bnu-u\right)}
	\\
	\times
	\frac{1}{2\pi i } \int_{(\sigma_w = 1/2-k/2-\epsilon)}
	\frac{\Gamma\left(-w+u\right) \Gamma\left(-w-u\right)
	\Gamma\left(w+\frac{k}{2}-\bnu_1\right) \Gamma\left(w+\frac{k}{2}+\bnu_1\right)}
	{\Gamma\left(\frac{1}{2}+w\right) \Gamma\left(\frac{1}{2}-w\right)}
	\\
	\times
	\cM^{(3)}_{f, \phi_2, \phi_1}(w+1/2, s-w-k/2+1/2)
	\; dw\; du.
\end{multline}

Similarly, for $\Re(s)>1+k/2+\epsilon$, let
\begin{multline}\label{e:tE3}
	\tE^{(3)}_{f, \phi_1, \phi_2}(s)
	\\
	:=
	-
	(4\pi)^{\frac{k-1}{2}} i^k \frac{4}{\pi} 
	\frac{1}{2\pi i} \int_{(\sigma_u = 1/2+2\epsilon)} 
	\frac{h(u/i) u 
	\Gamma\left(\frac{1}{2}+u\right) \Gamma\left(\frac{1}{2}-u\right)}
	{\Gamma\left(\frac{k}{2}+\bnu_1 +u\right) \Gamma\left(\frac{k}{2}+\bnu_1-u\right)}
	\frac{1}{2\pi i } \int_{(\sigma_w = 1/2+\epsilon)}
	\frac{\Gamma\left(-w+u\right)}
	{\Gamma\left(1+w+u\right)}
	\\
	\times
	\Gamma\left(w+\frac{k}{2}+\bar{\nu_1}\right) 
	\Gamma\left(w+\frac{k}{2}-\bar{\nu_1}\right)
	\cM^{(3)}_{f, \phi_1, \phi_2}\left(s-w+\frac{1}{2}-\frac{k}{2}, w+\frac{1}{2}\right)
	\; dw\; du
\end{multline}
and
\begin{multline}\label{e:tE4}
	\tE^{(4)}_{f, \phi_1, \phi_2}(s)
	\\
	:=
	(4\pi)^{\frac{k-1}{2}} i^k \frac{4}{\pi} 
	\frac{1}{2\pi i} \int_{(\sigma_u = 1/2+2\epsilon)} 
	\frac{h(u/i) u 
	\Gamma\left(\frac{1}{2}+u\right) \Gamma\left(\frac{1}{2}-u\right)}
	{\Gamma\left(\frac{k}{2}+\bnu_1 +u\right) \Gamma\left(\frac{k}{2}+\bnu_1-u\right)}
	\frac{1}{2\pi i } \int_{(\sigma_w = 1/2+\epsilon)}
	\frac{\Gamma\left(-w+u\right)}
	{\Gamma\left(1+w+u\right)}
	\\
	\times
	\frac{\Gamma\left(w+\frac{k}{2}+\bar{\nu_1}\right) 
	\Gamma\left(w+\frac{k}{2}-\bar{\nu_1}\right)
	\Gamma\left(s-w+\frac{1}{2}-\frac{k}{2}\right) 
	(4\pi)^{s-w-\frac{1}{2}}}
	{\Gamma\left(s-w+\bnu_2\right)
	\Gamma\left(s-w-\bnu_2\right) }
	\\
	\times
	\cM^*_{f, \phi_1, \phi_2} \left(s-w+\frac{1}{2}-\frac{k}{2}, w+\frac{1}{2}\right)
	\; dw \; du
	.
\end{multline}
Then, for $\Re(s) > 1+k/2+\epsilon$, we have
$$
	\tE^{(2)}_{f, \phi_1, \phi_2}(s)
	=
	\tE^{(3)}_{f, \phi_1, \phi_2}(s)
	+
	\tE^{(4)}_{f, \phi_1, \phi_2}(s)
	.
$$
\end{lem}
\begin{proof}

Recalling \eqref{e:E1fphi} and \eqref{e:tE1fphi1phi2}. 
For $\Re(s)>2$, we have
\begin{multline*}
	\tE^{(1)}_{f, \phi_1, \phi_2}(s)
	=
	\zeta(2s)\sum_{m=1}^\infty \frac{\overline{c_{\phi_2}(m)}}{m^{s-\frac{1}{2}}}
	E^{(1)}_{f, \phi}(m)
	\\
	=
	-
	4
	(4\pi)^{\frac{k-1}{2}} i^k 
	\frac{1}{2\pi i} \int_{(\sigma_u = 1/2+2\epsilon)} 
	\frac{h(u/i) u \tan(\pi u)}{\Gamma\left(\frac{k}{2}+\bnu +u\right) \Gamma\left(\frac{k}{2}+\bnu-u\right)}
	\\
	\times
	\frac{1}{2\pi i } \int_{(\sigma_w = 1/2-k/2-\epsilon)}
	\Gamma\left(-w+u\right) \Gamma\left(-w-u\right)
	\Gamma\left(w+\frac{k}{2}-\bnu_1\right) \Gamma\left(w+\frac{k}{2}+\bnu_1\right)
	\\
	\times
	\frac{1}{\Gamma\left(\frac{1}{2}+\frac{k}{2}+\bnu_1\right)
	\Gamma\left(\frac{1}{2}+\frac{k}{2}-\bnu_1\right)}
	\zeta(2s)
	\sum_{m=1}^\infty \frac{\overline{c_{\phi_2}(m)}}{m^{s-w-\frac{1}{2}}}
	\sum_{n=1}^{m-1} 
	\frac{a(m-n) \overline{c_{\phi_1}(-n)}}{n^{w+\frac{k}{2}-\frac{1}{2}}}	\; dw\; du
	\\
	=
	-
	4
	(4\pi)^{\frac{k-1}{2}} i^k 
	\frac{1}{2\pi i} \int_{(\sigma_u = 1/2+2\epsilon)} 
	\frac{h(u/i) u \tan(\pi u)}{\Gamma\left(\frac{k}{2}+\bnu +u\right) \Gamma\left(\frac{k}{2}+\bnu-u\right)}
	\\
	\times
	\frac{1}{2\pi i } \int_{(\sigma_w = 1/2-k/2-\epsilon)}
	\frac{\Gamma\left(-w+u\right) \Gamma\left(-w-u\right)
	\Gamma\left(w+\frac{k}{2}-\bnu_1\right) \Gamma\left(w+\frac{k}{2}+\bnu_1\right)}
	{\Gamma\left(\frac{1}{2}+w\right) \Gamma\left(\frac{1}{2}-w\right)}
	\\
	\times
	\cM^{(3)}_{f, \phi_2, \phi_1}(w+1/2, s-w-k/2+1/2)
	\; dw\; du, 
\end{multline*}
where $\cM^{(3)}_{f, \phi_2, \phi_1}$ is given in \eqref{e:cM3_series1}. 
We are allowed to change the ordering of the series and integrals since the series and integrals are convergent absolutely. 


Similarly, recalling \eqref{e:E2fphi} and \eqref{e:tE2fphi1phi2}, taking a summation over $m\geq 1$, 
for $\Re(s) > 1+k/2+\epsilon$, we get
\begin{multline*}
	\tE^{(2)}_{f, \phi_1, \phi_2}(s)
	=
	\zeta(2s) \sum_{m=1}^\infty \frac{\overline{c_{\phi_2}(m)}}{m^{s-\frac{1}{2}}}
	E^{(2)}_{f, \phi_1}(m)
	\\
	=
	(4\pi)^{\frac{k-1}{2}} i^k \frac{4}{\pi} 
	\frac{1}{2\pi i} \int_{(\sigma_u = 1/2+2\epsilon)} 
	\frac{h(u/i) u 
	\Gamma\left(\frac{1}{2}+u\right) \Gamma\left(\frac{1}{2}-u\right)}
	{\Gamma\left(\frac{k}{2}+\bnu +u\right) \Gamma\left(\frac{k}{2}+\bnu-u\right)}
	\frac{1}{2\pi i } \int_{(\sigma_w = 1/2+\epsilon)}
	\frac{\Gamma\left(-w+u\right)}
	{\Gamma\left(1+w+u\right)}
	\\
	\times
	\Gamma\left(w+\frac{k}{2}+\bar{\nu}\right) 
	\Gamma\left(w+\frac{k}{2}-\bar{\nu}\right)
	\zeta(2s)
	\sum_{m=1}^\infty 
	D_{f, \phi_1}(w+1/2; m) \frac{\overline{c_{\phi_2}(m)}}{m^{s-w-\frac{1}{2}}}
	\; dw\; du
	.
\end{multline*}
By \eqref{e:Dfphi} and \eqref{e:cMfphi1phi2}, we get
$$
	\zeta(2s) \sum_{m=1}^\infty 
	D_{f, \phi_1}(w+1/2; m) \frac{\overline{c_{\phi_2}(m)}}{m^{s-w-\frac{1}{2}}}
	=
	\cM_{f, \phi_1, \phi_2}(s-w+1/2-k/2, w+1/2)
	. 
$$
We have
\begin{multline*}
	\tE^{(2)}_{f, \phi_1, \phi_2}(s)
	\\
	=
	(4\pi)^{\frac{k-1}{2}} i^k \frac{4}{\pi} 
	\frac{1}{2\pi i} \int_{(\sigma_u = 1/2+2\epsilon)} 
	\frac{h(u/i) u 
	\Gamma\left(\frac{1}{2}+u\right) \Gamma\left(\frac{1}{2}-u\right)}
	{\Gamma\left(\frac{k}{2}+\bnu_1 +u\right) \Gamma\left(\frac{k}{2}+\bnu_1-u\right)}
	\frac{1}{2\pi i } \int_{(\sigma_w = 1/2+\epsilon)}
	\frac{\Gamma\left(-w+u\right)}
	{\Gamma\left(1+w+u\right)}
	\\
	\times
	\Gamma\left(w+\frac{k}{2}+\bar{\nu_1}\right) 
	\Gamma\left(w+\frac{k}{2}-\bar{\nu_1}\right)
	\cM_{f, \phi_1, \phi_2} \left(s-w+1/2-k/2, w+1/2\right)
	\; dw\; du
	.
\end{multline*}

Recalling \eqref{e:cM*} and \eqref{e:Gfphi}, 
\begin{multline*}
	\cM^*_{f, \phi_1, \phi_2}(s, w)
	\\
	=
	\frac{\Gamma\left(s+\frac{k}{2}+\bnu_2 -\frac{1}{2}\right)
	\Gamma\left(s+\frac{k}{2}-\bnu_2-\frac{1}{2}\right) 
	(4\pi)^{1-s-\frac{k}{2}} }
	{\Gamma(s)}
	\left\{
	\cM_{f, \phi_1, \phi_2}(s, w)
	+
	\cM^{(3)}_{f, \phi_1, \phi_2}(s, w)\right\}
	.
\end{multline*}
Then 
$$
	\tE^{(2)}_{f, \phi_1, \phi_2}(s)
	=
	\tE^{(3)}_{f, \phi_1, \phi_2}(s)
	+
	\tE^{(4)}_{f, \phi_1, \phi_2}(s)
	.
$$

\end{proof}

In Lemma \ref{lem:tE1+tE2_SDDS}, we rewrite $\tE^{(1)}_{f, \phi_1, \phi_2}(s)$ and $\tE^{(2)}_{f, \phi_1, \phi_2}(s)$ for $\Re(s)> 1+k/2+\epsilon$, associated to the double Dirichlet series. 
Our goal is to get a meromorphic continuation for $\tE^{(1)}_{f, \phi_1, \phi_2}(s)$ and $\tE^{(2)}_{f, \phi_1, \phi_2}(s)$ for $\Re(s) \geq 1/2$. 

We first need to study the following three functions on $s$ and $t$: 
for $|\Im(t)|< \epsilon/2$ and $\Re(s) \geq 1/2$, define
\begin{multline}\label{e:tth1}
	\tth^{(1)}(t; s, \nu_1,\nu_2)
	=
	-
	\frac{(4\pi)^{k}i^k }{\sqrt{\pi}}
	\frac{1}{2\pi i} \int_{(\sigma_u = 1/2+2\epsilon)} 
	\frac{h(u/i) u \tan(\pi u)}{\Gamma\left(\frac{k}{2}+\bnu +u\right) \Gamma\left(\frac{k}{2}+\bnu-u\right)}
	\\
	\times
	\frac{1}{2\pi i } \int_{(\sigma_w = 1/2-\epsilon)}
	\frac{\Gamma\left(-w+\frac{k}{2}+u\right) \Gamma\left(-w+\frac{k}{2}-u\right)
	\Gamma\left(w-\bnu_1\right) \Gamma\left(w+\bnu_1\right)}
	{\Gamma\left(\frac{1}{2}+\frac{k}{2}+\bnu_1\right)
	\Gamma\left(\frac{1}{2}+\frac{k}{2}-\bnu_1\right)
	\Gamma\left(s-w+\frac{k}{2}+\bnu_2\right) 
	\Gamma\left(s-w+\frac{k}{2}-\bnu_2\right)}
	\\
	\times
	\Gamma\left(s-w+it\right) \Gamma\left(s-w-it\right)
	\; dw\; du
	.
\end{multline}
For $|\Im(t)|< \epsilon/2$ and $1/2\leq \Re(s)\leq 1/2+\epsilon/2$, 
define
\begin{multline}\label{e:tth2}
	\tth^{(2)}(t; s, \nu_1, \nu_2)
	=
	\sum_{\ell=0}^{\frac{k}{2}}
	(-1)^\ell
	\frac{\Gamma\left(-\frac{1}{2}+s-\ell +it\right) \Gamma\left(-\frac{1}{2}+s-\ell -it\right)}
	{\Gamma\left(\frac{1}{2}+\frac{k}{2}+\bnu_2\right)
	\Gamma\left(\frac{1}{2}+\frac{k}{2}-\bnu_2\right)}
	\\
	\times
	\frac{(4\pi)^{k} }{\pi\sqrt{\pi} }
	\left\{
	-
	\frac{\pi}{\sin(\pi s)}
	\sum_{\ell_1=\ell}^{\frac{k}{2}} 
	\frac{1}{\ell_1!}
	\frac{h\left( (s-\ell_1 -1/2)/i\right) \left(s-\ell_1-\frac{1}{2}\right) }
	{\Gamma\left(\frac{k}{2}+\bnu_1 +s-\ell_1 -\frac{1}{2}\right) 
	\Gamma\left(\frac{k}{2}+\bnu_1 -s+\ell_1 +\frac{1}{2}\right)
	\Gamma\left(2s-\ell-\ell_1\right)}
	\right.
	\\
	\left.
	+
	\frac{1}{2\pi i} \int_{(\sigma_u=1/2+2\epsilon)} 
	h(u/i) u \frac{\cos(\pi(s+u))}{\cos(\pi u)}
	\frac{\Gamma\left(\frac{1}{2}-s+\ell+u\right)
	\Gamma\left(\frac{1}{2}-s+\ell-u\right)}
	{\Gamma\left(\frac{k}{2}+\bnu_1 +u\right) \Gamma\left(\frac{k}{2}+\bnu_1-u\right)}
	\; du
	\right\}
	\\
	-
	\frac{\Gamma\left(\frac{1}{2}+s+it\right)
	\Gamma\left(\frac{1}{2}+s-it\right)}
	{\Gamma\left(\frac{1}{2}+\frac{k}{2}+\bnu_2\right)
	\Gamma\left(\frac{1}{2}+\frac{k}{2}-\bnu_2\right)}
	\frac{(4\pi)^{k} }{\pi\sqrt{\pi} }
	\\
	\times
	\left\{
	\frac{\pi}{\sin(\pi s)}
	\sum_{\ell_1=0}^{\frac{k}{2}}
	\frac{1}{\ell_1!}
	\frac{h\left((s-1/2-\ell_1)/i\right) \left(s-\ell_1-\frac{1}{2}\right)}
	{\Gamma\left(\frac{k}{2}+\bnu_1 +s-\ell_1-\frac{1}{2} \right) 
	\Gamma\left(\frac{k}{2}+\bnu_1-s+\ell_1+\frac{1}{2}\right)}
	\right.
	\\
	\times 
	\frac{1}
	{\Gamma\left(2s-\ell_1\right)
	\left(s-\ell_1-\frac{1}{2}+it\right) \left(s-\ell_1-\frac{1}{2}-it\right)}
	\\
	\left.
	+
	\frac{1}{2\pi i} \int_{(\sigma_u = 1/2+2\epsilon)} 
	h(u/i) \frac{u}{(u+it)(u-it)} 
	\frac{\cos(\pi(s+u))}{\cos(\pi u)}
	\frac{\Gamma\left(\frac{1}{2}-s+u\right)
	\Gamma\left(\frac{1}{2}-s-u\right)}
	{\Gamma\left(\frac{k}{2}+\bnu_1 +u\right) \Gamma\left(\frac{k}{2}+\bnu_1-u\right)}
	\; du
	\right\}
\end{multline}
and
\begin{multline}\label{e:tth3}
	\tth^{(3)}(t; s, \nu_1, \nu_2)
	=
	\frac{(4\pi)^{k} i^k}{\pi\sqrt{\pi}}
	\frac{1}{2\pi i} \int_{(\sigma_u = 1/2+2\epsilon)} 
	\frac{h(u/i) u 
	\Gamma\left(\frac{1}{2}+u\right) \Gamma\left(\frac{1}{2}-u\right)}
	{\Gamma\left(\frac{k}{2}+\bnu_1 +u\right) \Gamma\left(\frac{k}{2}+\bnu_1-u\right)}
	\\
	\times
	\frac{1}{2\pi i } \int_{(\sigma_w = 1/2+\epsilon)}
	\frac{\Gamma\left(-w+u\right)}
	{\Gamma\left(1+w+u\right)}
	\Gamma\left(w+\frac{k}{2}+\bar{\nu_1}\right) 
	\Gamma\left(w+\frac{k}{2}-\bar{\nu_1}\right)
	\\
	\times
	\frac{\Gamma\left(s-w+\frac{1}{2}-\frac{k}{2}\right)
	\Gamma\left(-s+w+\frac{1}{2}+\frac{k}{2}\right)
	} 
	{\Gamma\left(s-w+\bnu_2\right)
	\Gamma\left(s-w-\bnu_2\right) }
	\frac{\Gamma\left(s-w-\frac{k}{2}-it\right) \Gamma\left(s-w-\frac{k}{2}+it\right)}
	{\Gamma\left(\frac{1}{2}-it\right) \Gamma\left(\frac{1}{2}+it\right)}
	\; dw \; du
	\\
	+
	\sum_{\ell=0}^{\frac{k}{2}}
	\frac{1}{\ell!}
	\frac{(4\pi)^{k} i^k}{\pi\sqrt{\pi}}
	\frac{
	\Gamma\left(s+\ell \pm it+\bar{\nu_1}\right) 
	\Gamma\left(s+\ell \pm it-\bar{\nu_1}\right)
	\Gamma\left(-\ell \mp 2it\right)} 
	{\Gamma\left(\frac{k}{2}-\ell \mp it+\bnu_2\right)
	\Gamma\left(\frac{k}{2}-\ell \mp it-\bnu_2\right) }
	\\
	\times
	\frac{1}{2\pi i} \int_{(\sigma_u = 1/2+2\epsilon)} 
	\frac{h(u/i) u 
	\Gamma\left(\frac{1}{2}+u\right) \Gamma\left(\frac{1}{2}-u\right)}
	{\Gamma\left(\frac{k}{2}+\bnu_1 +u\right) \Gamma\left(\frac{k}{2}+\bnu_1-u\right)}
	\frac{\Gamma\left(-s+\frac{k}{2}-\ell \mp it+u\right)}
	{\Gamma\left(1+s-\frac{k}{2}+\ell \pm it+u\right)}
	\; du
	.
\end{multline}

\begin{lem}\label{lem:tth1_mero}
Taking $t=(s-1\pm \bnu_1)/i$, $\tth^{(1)}_{f, \phi_1, \phi_2} ((s-1\pm \bnu_1)/i;s, \bnu_1, \bnu_2)$ has a meromorphic continuation to any $s\in \C$.
For $1/2\leq \Re(s)\leq 1/2+\epsilon/2$, let
\be\label{e:H1pm}
	H_1^\pm (s; \bnu_1)
	=
	\frac{1}{\pi^2}\int_{-\infty}^\infty 
	h(t) t \tanh(\pi t)
	\frac{
	\Gamma\left(1-2s+\frac{k}{2}\pm \bnu_1+it \right)
	\Gamma\left(1-2s+\frac{k}{2}\pm \bnu_1 -it \right)}
	{\Gamma\left(\frac{k}{2}+\bnu_1+it\right) 
	\Gamma\left(\frac{k}{2}+\bnu_1-it\right)}
	\; dt
\ee
and
\begin{multline}\label{e:tth1pm}
	\tth^{(1)\pm} (s; \bnu_1, \bnu_2)
	=
	-
	\frac{(4\pi)^k i^k}{\sqrt{\pi}}
	\frac{1}{2\pi i} \int_{(\sigma_u = 1/2+2\epsilon)} 
	\frac{h(u/i) u \tan(\pi u)}
	{\Gamma\left(\frac{k}{2}+\bnu_1+u\right) 
	\Gamma\left(\frac{k}{2}+\bnu_1-u\right)}
	\\
	\times
	\frac{1}{2\pi i} 
	\int_{(\sigma_w = 1-\epsilon)} 
	\frac{\Gamma\left(u+\frac{k}{2}-w\right) \Gamma\left(-u+\frac{k}{2}-w\right) 
	\Gamma\left(2s-1\pm \bnu_1-w\right) \Gamma\left(1\mp \bnu_1-w\right)}
	{\Gamma\left(s+\frac{k}{2}+\bnu_2-w\right)
	\Gamma\left(s+\frac{k}{2}-\bnu_2-w\right)}
	\\
	\times
	\frac{
	\Gamma\left(-\bnu_1+w\right) \Gamma\left(\bnu_1+w\right)}
	{\Gamma\left(\frac{1}{2}+\frac{k}{2}+\bnu_1\right)
	\Gamma\left(\frac{1}{2}+\frac{k}{2}-\bnu_1\right)}
	\; dw
	\; du
	.
\end{multline}
Then we have
\begin{multline}\label{e:tth1_Omega}
	\tth^{(1)}\left((s-1\pm \bnu_1)/i; s, \nu_1, \nu_2\right)
	\\	
	=
	\frac{(4\pi)^k i^k \sqrt{\pi}}{2}
	\frac{\Gamma\left(2\mp 2\bnu_1-2s\right)
	\Gamma\left(2s-1\right)
	\Gamma\left(2s-1\pm 2\bnu_1\right)}
	{\Gamma\left(\frac{1}{2}+\frac{k}{2}+\bnu_1\right) 
	\Gamma\left(\frac{1}{2}+\frac{k}{2}-\bnu_1\right)
	\Gamma\left(-s+1+\frac{k}{2}+\bnu_2\mp \bnu_1\right)
	\Gamma\left(-s+1+\frac{k}{2}-\bnu_2\mp \bnu_1\right)}
	\\
	\times
	H_1^\mp (s;\bnu_1)
	\\
	+
	\tth^{(1)\pm}(s; \bnu_1, \bnu_2)
	.
\end{multline}

\end{lem}
\begin{proof}

For $1< \Re(s) < 1+\epsilon/2 + |\Re(\bnu_1)|$, i.e., $|\Re(s-1\pm \bnu_1)| < \epsilon/2$ and $\Re(s)>1$, 
taking $t= (s-1\pm \bnu_1)/i$, 
we have
\begin{multline*}
	\tth^{(1)}((s-1\pm \bnu_1)/i; s, \nu_1,\nu_2)
	=
	-
	\frac{(4\pi)^{k}i^k }{\sqrt{\pi}}
	\frac{1}{2\pi i} \int_{(\sigma_u = 1/2+2\epsilon)} 
	\frac{h(u/i) u \tan(\pi u)}{\Gamma\left(\frac{k}{2}+\bnu +u\right) \Gamma\left(\frac{k}{2}+\bnu-u\right)}
	\\
	\times
	\frac{1}{2\pi i } \int_{(\sigma_w = 1/2-\epsilon)}
	\frac{\Gamma\left(-w+\frac{k}{2}+u\right) \Gamma\left(-w+\frac{k}{2}-u\right)
	\Gamma\left(w-\bnu_1\right) \Gamma\left(w+\bnu_1\right)}
	{\Gamma\left(\frac{1}{2}+\frac{k}{2}+\bnu_1\right)
	\Gamma\left(\frac{1}{2}+\frac{k}{2}-\bnu_1\right)
	\Gamma\left(s-w+\frac{k}{2}+\bnu_2\right) 
	\Gamma\left(s-w+\frac{k}{2}-\bnu_2\right)}
	\\
	\times
	\Gamma\left(2s-w-1\pm \bnu_1\right) \Gamma\left(-w+1\mp \bnu_1\right)
	\; dw\; du
	.
\end{multline*}
Let $C$ be a path running from $-\infty$ to $+\infty$ in such a way that poles of the functions 
$$
	\Gamma\left(-\bnu_1+w\right) \Gamma\left(\bnu_1+w\right)
$$
lie to the left, and the poles of the functions 
$$
	\Gamma\left(u+\frac{k}{2}-w\right) \Gamma\left(-u+\frac{k}{2}-w\right)
	\Gamma\left(2s-1\pm \bnu_1-w\right) \Gamma\left(1\mp \bnu_1-w\right)
$$
lie to the right of $C$. 
Then
\begin{multline*}
	\tth^{(1)}\left( (s-1\pm \bnu_1)/i; s, \nu_1, \nu_2\right)
	=
	-\frac{(4\pi)^k i^k}{\sqrt{\pi}}
	\frac{1}{2\pi i} \int_{(\sigma_u = 1/2+2\epsilon)} 
	\frac{h(u/i) u \tan(\pi u)}
	{\Gamma\left(\frac{k}{2}+\bnu_1+u\right) 
	\Gamma\left(\frac{k}{2}+\bnu_1-u\right)}
	\\
	\times
	\frac{1}{2\pi i} 
	\int_C \frac{\Gamma\left(u+\frac{k}{2}-w\right) \Gamma\left(-u+\frac{k}{2}-w\right) 
	\Gamma\left(2s-1\pm \bnu_1-w\right) \Gamma\left(1\mp \bnu_1-w\right)
	\Gamma\left(-\bnu_1+w\right) \Gamma\left(\bnu_1+w\right)}
	{\Gamma\left(\frac{1}{2}+\frac{k}{2}+\bnu_1\right)
	\Gamma\left(\frac{1}{2}+\frac{k}{2}-\bnu_1\right)
	\Gamma\left(s+\frac{k}{2}+\bnu_2-w\right)
	\Gamma\left(s+\frac{k}{2}-\bnu_2-w\right)}
	\\
	\times
	\; dw\; du
\end{multline*}
and it has an analytic continuation to any $s\in \C$. 
For $1/2\leq \Re(s) \leq 1/2+\epsilon/2$, moving the path $C$ to $\Re(w) = 1-\epsilon$, 
gives us
\begin{multline*}
	\tth^{(1)} \left( (s-1\pm \bnu_1)/i; s, \nu_1, \nu_2\right)
	\\
	=
	-\frac{(4\pi)^k i^k}{\sqrt{\pi}}
	\frac{1}{2\pi i} \int_{(\sigma_u = 1/2+2\epsilon)} 
	h(u/i) u \tan(\pi u)
	\frac{
	\Gamma\left(u+\frac{k}{2}-2s+1\mp \bnu_1 \right)
	\Gamma\left(-u+\frac{k}{2}-2s+1\mp \bnu_1\right)}
	{\Gamma\left(\frac{k}{2}+\bnu_1+u\right) 
	\Gamma\left(\frac{k}{2}+\bnu_1-u\right)}
	\; du
	\\
	\times
	\frac{\Gamma\left(2\mp 2\bnu_1-2s\right)
	\Gamma\left(2s-1\right)
	\Gamma\left(2s-1\pm 2\bnu_1\right)}
	{\Gamma\left(\frac{1}{2}+\frac{k}{2}+\bnu_1\right) 
	\Gamma\left(\frac{1}{2}+\frac{k}{2}-\bnu_1\right)
	\Gamma\left(-s+1+\frac{k}{2}+\bnu_2\mp \bnu_1\right)
	\Gamma\left(-s+1+\frac{k}{2}-\bnu_2\mp \bnu_1\right)}
	\\
	-
	\frac{(4\pi)^k i^k}{\sqrt{\pi}}
	\frac{1}{2\pi i} \int_{(\sigma_u = 1/2+2\epsilon)} 
	\frac{h(u/i) u \tan(\pi u)}
	{\Gamma\left(\frac{k}{2}+\bnu_1+u\right) 
	\Gamma\left(\frac{k}{2}+\bnu_1-u\right)}
	\\
	\times
	\frac{1}{2\pi i} 
	\int_{(\sigma_w = 1-\epsilon)} 
	\frac{\Gamma\left(u+\frac{k}{2}-w\right) \Gamma\left(-u+\frac{k}{2}-w\right) 
	\Gamma\left(2s-1\pm \bnu_1-w\right) \Gamma\left(1\mp \bnu_1-w\right)}
	{\Gamma\left(s+\frac{k}{2}+\bnu_2-w\right)
	\Gamma\left(s+\frac{k}{2}-\bnu_2-w\right)}
	\\
	\times
	\frac{
	\Gamma\left(-\bnu_1+w\right) \Gamma\left(\bnu_1+w\right)}
	{\Gamma\left(\frac{1}{2}+\frac{k}{2}+\bnu_1\right)
	\Gamma\left(\frac{1}{2}+\frac{k}{2}-\bnu_1\right)}
	\; dw
	\; du
	\\
	=
	\frac{(4\pi)^k i^k \sqrt{\pi}}{2}
	\frac{\Gamma\left(2\mp 2\bnu_1-2s\right)
	\Gamma\left(2s-1\right)
	\Gamma\left(2s-1\pm 2\bnu_1\right)}
	{\Gamma\left(\frac{1}{2}+\frac{k}{2}+\bnu_1\right) 
	\Gamma\left(\frac{1}{2}+\frac{k}{2}-\bnu_1\right)
	\Gamma\left(-s+1+\frac{k}{2}+\bnu_2\mp \bnu_1\right)
	\Gamma\left(-s+1+\frac{k}{2}-\bnu_2\mp \bnu_1\right)}
	\\
	\times
	H_1^\mp (s;\bnu_1)
	\\
	+
	\tth^{(1)\pm}(s; \bnu_1, \bnu_2)
	.
\end{multline*}

\end{proof}

For $1/2\leq \Re(s)< 1/2+\epsilon/2$, for $0\leq \ell \leq k/2$, let
\be\label{e:H2}
	H_{2, \ell} (s; \bnu_1)
	:=
	-
	\frac{1}{\pi^2} \int_{-\infty}^\infty 
	h(t) it\frac{\cos(\pi(s+it))}{\cos(\pi it)}
	\frac{\Gamma\left(\frac{1}{2}-s+\frac{k}{2}-\ell+it\right)
	\Gamma\left(\frac{1}{2}-s+\frac{k}{2}-\ell-it\right)}
	{\Gamma\left(\frac{k}{2}+\bnu_1 +it\right) 
	\Gamma\left(\frac{k}{2}+\bnu_1-it\right)}
	\; dt
	.
\ee

\begin{lem}
Taking $t=(s-1\pm \bnu_1)/i$, $\tth^{(2)}_{f, \phi_1, \phi_2} ((s-1\pm \bnu_1)/i;s, \bnu_1, \bnu_2)$ has a meromorphic continuation to 
$1/2\leq \Re(s) < 1/2+\epsilon/2$. 
For $1/2\leq \Re(s)< 1/2+\epsilon/2$, 
let
\begin{multline}\label{e:tth3pm}
	\tth^{(3)\pm}(s; \bnu_1, \bnu_2)
	:=
	\frac{(4\pi)^{k} i^k}{\pi\sqrt{\pi}}
	\frac{1}{2\pi i} \int_{(\sigma_u = 1/2+2\epsilon)} 
	\frac{h(u/i) u 
	\Gamma\left(\frac{1}{2}+u\right) \Gamma\left(\frac{1}{2}-u\right)}
	{\Gamma\left(\frac{k}{2}+\bnu_1 +u\right) \Gamma\left(\frac{k}{2}+\bnu_1-u\right)}
	\\
	\times
	\frac{1}{2\pi i } \int_{(\sigma_w = 1/2+\epsilon)}
	\frac{\Gamma\left(-w+u\right)}
	{\Gamma\left(1+w+u\right)}
	\Gamma\left(w+\frac{k}{2}+\bnu_1\right) 
	\Gamma\left(w+\frac{k}{2}-\bnu_1\right)
	\\
	\times
	\frac{\Gamma\left(s-w+\frac{1}{2}-\frac{k}{2}\right)
	\Gamma\left(-s+w+\frac{1}{2}+\frac{k}{2}\right)
	} 
	{\Gamma\left(s-w+\bnu_2\right)
	\Gamma\left(s-w-\bnu_2\right) }
	\frac{\Gamma\left(-w-\frac{k}{2}+1\mp \bnu_1\right) 
	\Gamma\left(2s-1-w-\frac{k}{2}\pm \bnu_1\right)}
	{\Gamma\left(\frac{3}{2}-s\mp \bnu_1\right) 
	\Gamma\left(-\frac{1}{2}+s\pm \bnu_1\right)}
	\; dw \; du
	\\
	+
	\sum_{\ell=1}^{\frac{k}{2}}
	(-1)^{\ell}
	\frac{(4\pi)^{k}}{\pi\sqrt{\pi}}
	\frac{\Gamma\left(\ell \mp 2\bnu_1\right)
	\Gamma\left(2s-1-\ell \pm 2\bnu_1\right)} 
	{\Gamma\left(s+\frac{k}{2}-\ell \pm \bnu_1+\bnu_2\right)
	\Gamma\left(s+\frac{k}{2}-\ell \pm \bnu_1-\bnu_2\right) }
	\\
	\times
	\frac{1}{2\pi i} \int_{(\sigma_u = 1/2+2\epsilon)} 
	h(u/i) u
	\left(\sin(\pi (\pm \bnu_1)) - \cos(\pi (\pm \bnu_1)) \tan(\pi u)\right)
	\\
	\times
	\frac{\Gamma\left(\frac{k}{2}-\ell  \pm \bnu_1+u\right)
	\Gamma\left(\frac{k}{2}-\ell  \pm \bnu_1-u\right)}
	{\Gamma\left(\frac{k}{2}+\bnu_1 +u\right) \Gamma\left(\frac{k}{2}+\bnu_1-u\right)}
	\; du
	\\
	-
	\sum_{\ell=1}^{\frac{k}{2}}
	\frac{1}{\ell!}
	\frac{(4\pi)^{k}}{\sqrt{\pi}}
	\frac{\Gamma\left(2s-1+\ell \right)}
	{\sin\left(\pi\left(2s\pm 2\bnu_1\right)\right)}
	\frac{1} 
	{\Gamma\left(-s+1+\frac{k}{2}-\ell \mp \bnu_1+\bnu_2\right)
	\Gamma\left(-s+1+\frac{k}{2}-\ell \mp \bnu_1-\bnu_2\right) }
	\\
	\times
	\frac{1}{2\pi i} \int_{(\sigma_u = 1/2+2\epsilon)} 
	h(u/i) u 
	\left(\sin\left(\pi(2s\pm \bnu_1)\right) + \cos\left(\pi(2s\pm \bnu_1)\right)\tan(\pi u)\right)
	\\
	\times
	\frac{\Gamma\left(-2s+1+\frac{k}{2}-\ell \mp \bnu_1+u\right)
	\Gamma\left(-2s+1+\frac{k}{2}-\ell \mp \bnu_1-u\right)}
	{\Gamma\left(\frac{k}{2}+\bnu_1 +u\right) \Gamma\left(\frac{k}{2}+\bnu_1-u\right)}
	\; du
	.
\end{multline}
Then 
\begin{multline}\label{e:tth3_Omega}
	\tth^{(3)}\left( (s-1\pm \bnu_1)/i; s, \nu_1, \nu_2\right)
	=
	\tth^{(3)\pm}(s; \bnu_1, \bnu_2)
	\\
	+
	\frac{(4\pi)^{k}\sqrt{\pi}}{2}
	\frac{\Gamma\left(2s-1\right)
	\cos\left(\pi(2s\pm \bnu_1)\right)}
	{\sin\left(\pi\left(2s\pm 2\bnu_1\right)\right)}
	\frac{1} 
	{\Gamma\left(-s+1+\frac{k}{2}\mp \bnu_1+\bnu_2\right)
	\Gamma\left(-s+1+\frac{k}{2}\mp \bnu_1-\bnu_2\right) }
	H_1^\mp (s;\bnu_1)
\end{multline}

\end{lem}
\begin{proof}

Recalling \eqref{e:tth3}, 
let $C$ be a path running from $-\infty$ to $+\infty$ 
in such a way that the poles of 
$$
	\Gamma\left(w+\frac{k}{2}+\bnu_1\right) 
	\Gamma\left(w+\frac{k}{2}-\bnu_1\right)
	\Gamma\left(-s+w+\frac{1}{2}+\frac{k}{2}\right)
$$
lie to the left, 
and the poles of 
$$
	\Gamma\left(-w+u\right)
	\Gamma\left(s-w+\frac{1}{2}-\frac{k}{2}\right)
	\Gamma\left(s-w-\frac{k}{2}-it\right) 
	\Gamma\left(s-w-\frac{k}{2}+it\right)
$$
lie to the right of $C$.
For $1/2\leq \Re(s)< 1/2+\epsilon/2$, we get
\begin{multline*}
	\tth^{(3)}(t; s, \nu_1, \nu_2)
	=
	\frac{(4\pi)^{k} i^k}{\pi\sqrt{\pi}}
	\frac{1}{2\pi i} \int_{(\sigma_u = 1/2+2\epsilon)} 
	\frac{h(u/i) u 
	\Gamma\left(\frac{1}{2}+u\right) \Gamma\left(\frac{1}{2}-u\right)}
	{\Gamma\left(\frac{k}{2}+\bnu_1 +u\right) \Gamma\left(\frac{k}{2}+\bnu_1-u\right)}
	\\
	\times
	\frac{1}{2\pi i } \int_{C}
	\frac{\Gamma\left(-w+u\right)}
	{\Gamma\left(1+w+u\right)}
	\Gamma\left(w+\frac{k}{2}+\bnu_1\right) 
	\Gamma\left(w+\frac{k}{2}-\bnu_1\right)
	\\
	\times
	\frac{\Gamma\left(s-w+\frac{1}{2}-\frac{k}{2}\right)
	\Gamma\left(-s+w+\frac{1}{2}+\frac{k}{2}\right)
	} 
	{\Gamma\left(s-w+\bnu_2\right)
	\Gamma\left(s-w-\bnu_2\right) }
	\frac{\Gamma\left(s-w-\frac{k}{2}-it\right) \Gamma\left(s-w-\frac{k}{2}+it\right)}
	{\Gamma\left(\frac{1}{2}-it\right) \Gamma\left(\frac{1}{2}+it\right)}
	\; dw \; du
	\\
	+
	\sum_{\ell=1}^{\frac{k}{2}}
	\frac{(4\pi)^{k}}{\pi\sqrt{\pi}}
	\frac{1}{2\pi i} \int_{(\sigma_u = 1/2+2\epsilon)} 
	h(u/i) u 
	\left(\cos(\pi s) - \sin(\pi s) \tan(\pi u)\right)
	\\
	\times
	\frac{\Gamma\left(-s+\frac{1}{2}+\frac{k}{2}-\ell+u\right)
	\Gamma\left(-s+\frac{1}{2}+\frac{k}{2}-\ell-u\right)}
	{\Gamma\left(\frac{k}{2}+\bnu_1 +u\right) 
	\Gamma\left(\frac{k}{2}+\bnu_1-u\right)}
	\; du
	\\
	\times
	\frac{\Gamma\left(s-\frac{1}{2}+\ell+\bnu_1\right) 
	\Gamma\left(s-\frac{1}{2}+\ell -\bnu_1\right)}
	{\Gamma\left(\frac{1}{2}+\frac{k}{2}-\ell +\bnu_2\right)
	\Gamma\left(\frac{1}{2}+\frac{k}{2}-\ell -\bnu_2\right) }
	\frac{1}{\prod_{j=1}^{\ell} \left[\left(\frac{1}{2}-it -j \right) 
	\left(\frac{1}{2}+it-j\right)\right]}
	, 
\end{multline*}
and it has a meromorphic continuation to any $s\in \C$. 

For $\ell=k/2$, we have
\begin{multline*}
	\frac{1}{2\pi i} \int_{(\sigma_u = 1/2+2\epsilon)} 
	h(u/i) u 
	\left(\cos(\pi s) - \sin(\pi s) \tan(\pi u)\right)
	\frac{\Gamma\left(-s+\frac{1}{2}+u\right)
	\Gamma\left(-s+\frac{1}{2}-u\right)}
	{\Gamma\left(\frac{k}{2}+\bnu_1 +u\right) 
	\Gamma\left(\frac{k}{2}+\bnu_1-u\right)}
	\; du
	\\
	=
	\frac{1}{2\pi i} \int_{(\sigma_u = 1/2+2\epsilon)} 
	h(u/i) u 
	\frac{\pi}{\cos(\pi u)}
	\frac{\Gamma\left(\frac{1}{2}-s+u\right)}
	{\Gamma\left(\frac{1}{2}+s+u\right)}
	\frac{1}
	{\Gamma\left(\frac{k}{2}+\bnu_1 +u\right) 
	\Gamma\left(\frac{k}{2}+\bnu_1-u\right)}
	\; du, 
\end{multline*}
so we have
\begin{multline*}
	\tth^{(3)}(t; s, \nu_1, \nu_2)
	=
	\frac{(4\pi)^{k} i^k}{\pi\sqrt{\pi}}
	\frac{1}{2\pi i} \int_{(\sigma_u = 1/2+2\epsilon)} 
	\frac{h(u/i) u 
	\Gamma\left(\frac{1}{2}+u\right) \Gamma\left(\frac{1}{2}-u\right)}
	{\Gamma\left(\frac{k}{2}+\bnu_1 +u\right) \Gamma\left(\frac{k}{2}+\bnu_1-u\right)}
	\\
	\times
	\frac{1}{2\pi i } \int_{C}
	\frac{\Gamma\left(-w+u\right)}
	{\Gamma\left(1+w+u\right)}
	\Gamma\left(w+\frac{k}{2}+\bnu_1\right) 
	\Gamma\left(w+\frac{k}{2}-\bnu_1\right)
	\\
	\times
	\frac{\Gamma\left(s-w+\frac{1}{2}-\frac{k}{2}\right)
	\Gamma\left(-s+w+\frac{1}{2}+\frac{k}{2}\right)
	} 
	{\Gamma\left(s-w+\bnu_2\right)
	\Gamma\left(s-w-\bnu_2\right) }
	\frac{\Gamma\left(s-w-\frac{k}{2}-it\right) \Gamma\left(s-w-\frac{k}{2}+it\right)}
	{\Gamma\left(\frac{1}{2}-it\right) \Gamma\left(\frac{1}{2}+it\right)}
	\; dw \; du
	\\
	+
	\sum_{\ell=1}^{\frac{k}{2}-1}
	\frac{(4\pi)^{k}}{\pi\sqrt{\pi}}
	\frac{1}{2\pi i} \int_{(\sigma_u = 1/2+2\epsilon)} 
	h(u/i) u 
	\left(\cos(\pi s) - \sin(\pi s) \tan(\pi u)\right)
	\\
	\times
	\frac{\Gamma\left(-s+\frac{1}{2}+\frac{k}{2}-\ell+u\right)
	\Gamma\left(-s+\frac{1}{2}+\frac{k}{2}-\ell-u\right)}
	{\Gamma\left(\frac{k}{2}+\bnu_1 +u\right) 
	\Gamma\left(\frac{k}{2}+\bnu_1-u\right)}
	\; du
	\\
	\times
	\frac{\Gamma\left(s-\frac{1}{2}+\ell+\bnu_1\right) 
	\Gamma\left(s-\frac{1}{2}+\ell -\bnu_1\right)}
	{\Gamma\left(\frac{1}{2}+\frac{k}{2}-\ell +\bnu_2\right)
	\Gamma\left(\frac{1}{2}+\frac{k}{2}-\ell -\bnu_2\right) }
	\frac{1}{\prod_{j=1}^{\ell} \left[\left(\frac{1}{2}-it -j \right) 
	\left(\frac{1}{2}+it-j\right)\right]}
	\\
	+
	\frac{(4\pi)^{k}}{\pi\sqrt{\pi}}
	\frac{1}{2\pi i} \int_{(\sigma_u = 1/2+2\epsilon)} 
	h(u/i) u 
	\frac{\pi}{\cos(\pi u)}
	\frac{\Gamma\left(\frac{1}{2}-s+u\right)}
	{\Gamma\left(\frac{1}{2}+s+u\right)}
	\frac{1}
	{\Gamma\left(\frac{k}{2}+\bnu_1 +u\right) 
	\Gamma\left(\frac{k}{2}+\bnu_1-u\right)}
	\; du
	\\
	\times
	\frac{\Gamma\left(s-\frac{1}{2}+\frac{k}{2}+\bnu_1\right) 
	\Gamma\left(s-\frac{1}{2}+\frac{k}{2} -\bnu_1\right)}
	{\Gamma\left(\frac{1}{2}+\bnu_2\right)
	\Gamma\left(\frac{1}{2}-\bnu_2\right) }
	\frac{1}{\prod_{j=1}^{\frac{k}{2}} \left[\left(\frac{1}{2}-it -j \right) 
	\left(\frac{1}{2}+it-j\right)\right]}
	, 
\end{multline*}
and this description also works for $1< \Re(s) < 1+\epsilon/2+|\Re(\bnu_1)|$, 
i.e., $|\Re(s-1\pm \bnu_1)| < \epsilon/2$ 
and $\Re(s)>1$. 

On $1< \Re(s) < 1+\epsilon/2+|\Re(\bnu_1)|$, taking $t=(s-1\pm \bnu_1)/i$, 
we have
\begin{multline*}
	\tth^{(3)}\left( (s-1\pm \bnu_1)/i ; s, \nu_1, \nu_2\right)
	=
	\frac{(4\pi)^{k} i^k}{\pi\sqrt{\pi}}
	\frac{1}{2\pi i} \int_{(\sigma_u = 1/2+2\epsilon)} 
	\frac{h(u/i) u 
	\Gamma\left(\frac{1}{2}+u\right) \Gamma\left(\frac{1}{2}-u\right)}
	{\Gamma\left(\frac{k}{2}+\bnu_1 +u\right) \Gamma\left(\frac{k}{2}+\bnu_1-u\right)}
	\\
	\times
	\frac{1}{2\pi i } \int_{C}
	\frac{\Gamma\left(-w+u\right)}
	{\Gamma\left(1+w+u\right)}
	\Gamma\left(w+\frac{k}{2}+\bnu_1\right) 
	\Gamma\left(w+\frac{k}{2}-\bnu_1\right)
	\\
	\times
	\frac{\Gamma\left(s-w+\frac{1}{2}-\frac{k}{2}\right)
	\Gamma\left(-s+w+\frac{1}{2}+\frac{k}{2}\right)
	} 
	{\Gamma\left(s-w+\bnu_2\right)
	\Gamma\left(s-w-\bnu_2\right) }
	\frac{\Gamma\left(-w-\frac{k}{2}+1\mp \bnu_1\right) 
	\Gamma\left(2s-1-w-\frac{k}{2}\pm \bnu_1\right)}
	{\Gamma\left(\frac{3}{2}-s\mp \bnu_1\right) 
	\Gamma\left(-\frac{1}{2}+s\pm \bnu_1\right)}
	\; dw \; du
	\\
	+
	\sum_{\ell=1}^{\frac{k}{2}-1}
	\frac{(4\pi)^{k}}{\pi\sqrt{\pi}}
	\frac{1}{2\pi i} \int_{(\sigma_u = 1/2+2\epsilon)} 
	h(u/i) u 
	\left(\cos(\pi s) - \sin(\pi s) \tan(\pi u)\right)
	\\
	\times
	\frac{\Gamma\left(-s+\frac{1}{2}+\frac{k}{2}-\ell+u\right)
	\Gamma\left(-s+\frac{1}{2}+\frac{k}{2}-\ell-u\right)}
	{\Gamma\left(\frac{k}{2}+\bnu_1 +u\right) 
	\Gamma\left(\frac{k}{2}+\bnu_1-u\right)}
	\; du
	\\
	\times
	\frac{\Gamma\left(s-\frac{1}{2}+\ell+\bnu_1\right) 
	\Gamma\left(s-\frac{1}{2}+\ell -\bnu_1\right)}
	{\Gamma\left(\frac{1}{2}+\frac{k}{2}-\ell +\bnu_2\right)
	\Gamma\left(\frac{1}{2}+\frac{k}{2}-\ell -\bnu_2\right) }
	\frac{1}{\prod_{j=1}^{\ell} \left[\left(\frac{3}{2}-s\mp \bnu_1-j \right) 
	\left(-\frac{1}{2}+s\pm \bnu_1-j\right)\right]}
	\\
	+
	\frac{(4\pi)^{k}}{\pi\sqrt{\pi}}
	\frac{1}{2\pi i} \int_{(\sigma_u = 1/2+2\epsilon)} 
	h(u/i) u 
	\frac{\pi}{\cos(\pi u)}
	\frac{\Gamma\left(\frac{1}{2}-s+u\right)}
	{\Gamma\left(\frac{1}{2}+s+u\right)}
	\frac{1}
	{\Gamma\left(\frac{k}{2}+\bnu_1 +u\right) 
	\Gamma\left(\frac{k}{2}+\bnu_1-u\right)}
	\; du
	\\
	\times
	\frac{\Gamma\left(s-\frac{1}{2}+\frac{k}{2}+\bnu_1\right) 
	\Gamma\left(s-\frac{1}{2}+\frac{k}{2} -\bnu_1\right)}
	{\Gamma\left(\frac{1}{2}+\bnu_2\right)
	\Gamma\left(\frac{1}{2}-\bnu_2\right) }
	\frac{1}{\prod_{j=1}^{\frac{k}{2}} \left[\left(\frac{3}{2}-s\mp \bnu_1-j \right) 
	\left(-\frac{1}{2}+s\pm \bnu_1-j\right)\right]}
	, 
\end{multline*}
and it has a meromorphic continuation to $1/2 \leq \Re(s) < 1/2+\epsilon/2$. 

Assuming $1/2\leq \Re(s) < 1/2+\epsilon/2$, 
moving $C$ back to $\sigma_w = 1/2+\epsilon$, 
gives us
\begin{multline*}
	\tth^{(3)}\left( (s-1\pm \bnu_1)/i ; s, \nu_1, \nu_2\right)
	=
	\frac{(4\pi)^{k} i^k}{\pi\sqrt{\pi}}
	\frac{1}{2\pi i} \int_{(\sigma_u = 1/2+2\epsilon)} 
	\frac{h(u/i) u 
	\Gamma\left(\frac{1}{2}+u\right) \Gamma\left(\frac{1}{2}-u\right)}
	{\Gamma\left(\frac{k}{2}+\bnu_1 +u\right) \Gamma\left(\frac{k}{2}+\bnu_1-u\right)}
	\\
	\times
	\frac{1}{2\pi i } \int_{(\sigma_w = 1/2+\epsilon)}
	\frac{\Gamma\left(-w+u\right)}
	{\Gamma\left(1+w+u\right)}
	\Gamma\left(w+\frac{k}{2}+\bnu_1\right) 
	\Gamma\left(w+\frac{k}{2}-\bnu_1\right)
	\\
	\times
	\frac{\Gamma\left(s-w+\frac{1}{2}-\frac{k}{2}\right)
	\Gamma\left(-s+w+\frac{1}{2}+\frac{k}{2}\right)
	} 
	{\Gamma\left(s-w+\bnu_2\right)
	\Gamma\left(s-w-\bnu_2\right) }
	\frac{\Gamma\left(-w-\frac{k}{2}+1\mp \bnu_1\right) 
	\Gamma\left(2s-1-w-\frac{k}{2}\pm \bnu_1\right)}
	{\Gamma\left(\frac{3}{2}-s\mp \bnu_1\right) 
	\Gamma\left(-\frac{1}{2}+s\pm \bnu_1\right)}
	\; dw \; du
	\\
	+
	\sum_{\ell=0}^{\frac{k}{2}-1}
	(-1)^\ell 
	\frac{(4\pi)^{k} i^k}{\pi\sqrt{\pi}}
	\frac{1}{2\pi i} \int_{(\sigma_u = 1/2+2\epsilon)} 
	\frac{h(u/i) u 
	\Gamma\left(\frac{1}{2}+u\right) \Gamma\left(\frac{1}{2}-u\right)}
	{\Gamma\left(\frac{k}{2}+\bnu_1 +u\right) \Gamma\left(\frac{k}{2}+\bnu_1-u\right)}
	\\
	\times
	\frac{\Gamma\left(-1+\frac{k}{2}-\ell \pm \bnu_1+u\right)}
	{\Gamma\left(2-\frac{k}{2}+\ell\mp \bnu_1+u\right)}
	\Gamma\left(1+\ell \mp 2\bnu_1\right)
	\\
	\times
	\frac{\Gamma\left(s+\frac{1}{2}-1-\ell \pm \bnu_1\right)
	\Gamma\left(-s+\frac{1}{2}+1+\ell \mp \bnu_1\right)
	} 
	{\Gamma\left(s-1+\frac{k}{2}-\ell \pm \bnu_1+\bnu_2\right)
	\Gamma\left(s-1+\frac{k}{2}-\ell \pm \bnu_1-\bnu_2\right) }
	\frac{
	\Gamma\left(2s-2-\ell \pm 2\bnu_1\right)}
	{\Gamma\left(\frac{3}{2}-s\mp \bnu_1\right) 
	\Gamma\left(-\frac{1}{2}+s\pm \bnu_1\right)}
	\; du
	\\
	+
	\sum_{\ell=0}^{\frac{k}{2}}
	\frac{(-1)^\ell}{\ell!}
	\frac{(4\pi)^{k} i^k}{\pi\sqrt{\pi}}
	\frac{1}{2\pi i} \int_{(\sigma_u = 1/2+2\epsilon)} 
	\frac{h(u/i) u 
	\Gamma\left(\frac{1}{2}+u\right) \Gamma\left(\frac{1}{2}-u\right)}
	{\Gamma\left(\frac{k}{2}+\bnu_1 +u\right) \Gamma\left(\frac{k}{2}+\bnu_1-u\right)}
	\\
	\times
	\frac{\Gamma\left(-2s+1+\frac{k}{2}-\ell \mp \bnu_1+u\right)}
	{\Gamma\left(2s-\frac{k}{2}+\ell \pm \bnu_1+u\right)}
	\Gamma\left(2s-1+\ell \pm2\bnu_1\right) 
	\Gamma\left(2s-1+\ell \right)
	\\
	\times
	\frac{\Gamma\left(-s+\frac{3}{2}-\ell \mp \bnu_1\right)
	\Gamma\left(s-\frac{1}{2}+\ell \pm \bnu_1\right)
	} 
	{\Gamma\left(-s+1+\frac{k}{2}-\ell \mp \bnu_1+\bnu_2\right)
	\Gamma\left(-s+1+\frac{k}{2}-\ell \mp \bnu_1-\bnu_2\right) }
	\frac{\Gamma\left(-2s+2-\ell\mp 2\bnu_1\right) }
	{\Gamma\left(\frac{3}{2}-s\mp \bnu_1\right) 
	\Gamma\left(-\frac{1}{2}+s\pm \bnu_1\right)}
	\; du
	. 
\end{multline*}

For $0\leq \ell \leq k/2-1$, we have
\begin{multline*}
	(-1)^\ell 
	\frac{(4\pi)^{k} i^k}{\pi\sqrt{\pi}}
	\frac{1}{2\pi i} \int_{(\sigma_u = 1/2+2\epsilon)} 
	\frac{h(u/i) u 
	\Gamma\left(\frac{1}{2}+u\right) \Gamma\left(\frac{1}{2}-u\right)}
	{\Gamma\left(\frac{k}{2}+\bnu_1 +u\right) \Gamma\left(\frac{k}{2}+\bnu_1-u\right)}
	\frac{\Gamma\left(-1+\frac{k}{2}-\ell \pm \bnu_1+u\right)}
	{\Gamma\left(2-\frac{k}{2}+\ell\mp \bnu_1+u\right)}
	\\
	\times
	\Gamma\left(1+\ell \mp 2\bnu_1\right)
	\frac{\Gamma\left(s+\frac{1}{2}-1-\ell \pm \bnu_1\right)
	\Gamma\left(-s+\frac{1}{2}+1+\ell \mp \bnu_1\right)
	} 
	{\Gamma\left(s-1+\frac{k}{2}-\ell \pm \bnu_1+\bnu_2\right)
	\Gamma\left(s-1+\frac{k}{2}-\ell \pm \bnu_1-\bnu_2\right) }
	\frac{
	\Gamma\left(2s-2-\ell \pm 2\bnu_1\right)}
	{\Gamma\left(\frac{3}{2}-s\mp \bnu_1\right) 
	\Gamma\left(-\frac{1}{2}+s\pm \bnu_1\right)}
	\; du
	\\
	=
	(-1)^{\ell +1}
	\frac{(4\pi)^{k}}{\pi\sqrt{\pi}}
	\frac{\Gamma\left(1+\ell \mp 2\bnu_1\right)
	\Gamma\left(2s-2-\ell \pm 2\bnu_1\right)} 
	{\Gamma\left(s+\frac{k}{2}-\ell -1\pm \bnu_1+\bnu_2\right)
	\Gamma\left(s+\frac{k}{2}-\ell-1 \pm \bnu_1-\bnu_2\right) }
	\\
	\times
	\frac{1}{2\pi i} \int_{(\sigma_u = 1/2+2\epsilon)} 
	h(u/i) u
	\left(\sin(\pi (\pm \bnu_1)) - \cos(\pi (\pm \bnu_1)) \tan(\pi u)\right)
	\\
	\times
	\frac{\Gamma\left(\frac{k}{2}-\ell -1 \pm \bnu_1+u\right)
	\Gamma\left(\frac{k}{2}-\ell -1 \pm \bnu_1-u\right)}
	{\Gamma\left(\frac{k}{2}+\bnu_1 +u\right) \Gamma\left(\frac{k}{2}+\bnu_1-u\right)}
	\; du
	.
\end{multline*}
For $0\leq \ell \leq k/2$, we have
\begin{multline*}
	\frac{(-1)^\ell}{\ell!}
	\frac{(4\pi)^{k} i^k}{\pi\sqrt{\pi}}
	\frac{1}{2\pi i} \int_{(\sigma_u = 1/2+2\epsilon)} 
	\frac{h(u/i) u 
	\Gamma\left(\frac{1}{2}+u\right) \Gamma\left(\frac{1}{2}-u\right)}
	{\Gamma\left(\frac{k}{2}+\bnu_1 +u\right) \Gamma\left(\frac{k}{2}+\bnu_1-u\right)}
	\\
	\times
	\frac{\Gamma\left(-2s+1+\frac{k}{2}-\ell \mp \bnu_1+u\right)}
	{\Gamma\left(2s-\frac{k}{2}+\ell \pm \bnu_1+u\right)}
	\Gamma\left(2s-1+\ell \pm2\bnu_1\right) 
	\Gamma\left(2s-1+\ell \right)
	\\
	\times
	\frac{\Gamma\left(-s+\frac{3}{2}-\ell \mp \bnu_1\right)
	\Gamma\left(s-\frac{1}{2}+\ell \pm \bnu_1\right)
	} 
	{\Gamma\left(-s+1+\frac{k}{2}-\ell \mp \bnu_1+\bnu_2\right)
	\Gamma\left(-s+1+\frac{k}{2}-\ell \mp \bnu_1-\bnu_2\right) }
	\frac{\Gamma\left(-2s+2-\ell\mp 2\bnu_1\right) }
	{\Gamma\left(\frac{3}{2}-s\mp \bnu_1\right) 
	\Gamma\left(-\frac{1}{2}+s\pm \bnu_1\right)}
	\; du
	\\
	=
	-
	\frac{1}{\ell!}
	\frac{(4\pi)^{k}}{\sqrt{\pi}}
	\frac{\Gamma\left(2s-1+\ell \right)}
	{\sin\left(\pi\left(2s\pm 2\bnu_1\right)\right)}
	\frac{1} 
	{\Gamma\left(-s+1+\frac{k}{2}-\ell \mp \bnu_1+\bnu_2\right)
	\Gamma\left(-s+1+\frac{k}{2}-\ell \mp \bnu_1-\bnu_2\right) }
	\\
	\times
	\frac{1}{2\pi i} \int_{(\sigma_u = 1/2+2\epsilon)} 
	h(u/i) u 
	\left(\sin\left(\pi(2s\pm \bnu_1)\right) + \cos\left(\pi(2s\pm \bnu_1)\right)\tan(\pi u)\right)
	\\
	\times
	\frac{\Gamma\left(-2s+1+\frac{k}{2}-\ell \mp \bnu_1+u\right)
	\Gamma\left(-2s+1+\frac{k}{2}-\ell \mp \bnu_1-u\right)}
	{\Gamma\left(\frac{k}{2}+\bnu_1 +u\right) \Gamma\left(\frac{k}{2}+\bnu_1-u\right)}
	\; du
	.
\end{multline*}
For $\ell=0$, we have
\begin{multline*}
	-
	\frac{1}{\ell!}
	\frac{(4\pi)^{k}}{\sqrt{\pi}}
	\frac{\Gamma\left(2s-1+\ell \right)}
	{\sin\left(\pi\left(2s\pm 2\bnu_1\right)\right)}
	\frac{1} 
	{\Gamma\left(-s+1+\frac{k}{2}-\ell \mp \bnu_1+\bnu_2\right)
	\Gamma\left(-s+1+\frac{k}{2}-\ell \mp \bnu_1-\bnu_2\right) }
	\\
	\times
	\frac{1}{2\pi i} \int_{(\sigma_u = 1/2+2\epsilon)} 
	h(u/i) u 
	\left(\sin\left(\pi(2s\pm \bnu_1)\right) + \cos\left(\pi(2s\pm \bnu_1)\right)\tan(\pi u)\right)
	\\
	\times
	\frac{\Gamma\left(-2s+1+\frac{k}{2}-\ell \mp \bnu_1+u\right)
	\Gamma\left(-2s+1+\frac{k}{2}-\ell \mp \bnu_1-u\right)}
	{\Gamma\left(\frac{k}{2}+\bnu_1 +u\right) \Gamma\left(\frac{k}{2}+\bnu_1-u\right)}
	\; du
	\\
	=
	\frac{(4\pi)^{k}\sqrt{\pi}}{2}
	\frac{\Gamma\left(2s-1\right)
	\cos\left(\pi(2s\pm \bnu_1)\right)}
	{\sin\left(\pi\left(2s\pm 2\bnu_1\right)\right)}
	\frac{1} 
	{\Gamma\left(-s+1+\frac{k}{2}\mp \bnu_1+\bnu_2\right)
	\Gamma\left(-s+1+\frac{k}{2}\mp \bnu_1-\bnu_2\right) }
	H_1^\mp (s;\bnu_1)
	.
\end{multline*}

Thus we have
\begin{multline*}
	\tth^{(3)}\left( (s-1\pm \bnu_1)/i ; s, \nu_1, \nu_2\right)
	=
	\tth^{(3)\pm}(s; \bnu_1, \bnu_2)
	\\
	+
	\frac{(4\pi)^{k}\sqrt{\pi}}{2}
	\frac{\Gamma\left(2s-1\right)
	\cos\left(\pi(2s\pm \bnu_1)\right)}
	{\sin\left(\pi\left(2s\pm 2\bnu_1\right)\right)}
	\frac{1} 
	{\Gamma\left(-s+1+\frac{k}{2}\mp \bnu_1+\bnu_2\right)
	\Gamma\left(-s+1+\frac{k}{2}\mp \bnu_1-\bnu_2\right) }
	H_1^\mp (s;\bnu_1)
	.
\end{multline*}
\end{proof}

For $1/2\leq \Re(s)< 1/2+\epsilon/2$, 
define
\begin{multline}\label{e:E1}
	E^{(1)}_{f, \phi_1, \phi_2}(s)
	:=
	\overline{c_{\phi_1}(-1)}
	\left\{
	\sum_j 
	\tth^{(1)}(t_j; s, \nu_1, \nu_2)
	\left<u_j, U_{f, \phi_2}\right> 
	\cL\left(s, \overline{u_j} \times \overline{\phi_1}\right)
	\right.
	\\
	\left.
	+
	\sum_{\cuspa} \frac{1}{4\pi} 
	\int_{-\infty}^\infty 
	\tth^{(1)}(t; s, \nu_1, \nu_2)
	\left<E_{\cuspa}(*, 1/2+it), U_{f, \phi_2}\right> 
	\cL\left(s, E_\cuspa(*,1/2-it)\times \overline{\phi_1}\right)
	\; dt
	\right\}
	, 
\end{multline}
\begin{multline}\label{e:E2}
	E^{(2)}_{f, \phi_1, \phi_2}(s)
	:=
	-
	\overline{c_{\phi_2}(-1)}
	\left\{
	\sum_j 
	\tth^{(2)}(t_j; s, \nu_1, \nu_2)
	\left<u_j, U_{f, \phi}\right> 
	\cL\left(s, \overline{u_j} \times \overline{\phi_2}\right)
	\right.
	\\
	\left.
	+
	\sum_{\cuspa} \frac{1}{4\pi} 
	\tth^{(2)}(t; s, \nu_1, \nu_2)
	\int_{-\infty}^\infty \left<E_{\cuspa}(*, 1/2+it), U_{f, \phi_1}\right> 
	\cL\left(s, E_\cuspa(*,1/2-it)\times \overline{\phi_2}\right)
	\; dt
	\right\}
\end{multline}
and 
\begin{multline}\label{e:E3}
	E^{(3)}_{f, \phi_1, \phi_2}(s)
	:=
	\overline{c_{\phi_1}(1)}
	\left\{
	\sum_j 
	\tth^{(3)}(t_j; s, \nu_1, \nu_2) 
	(-1)^{\alpha_j} 
	\cL\left(s, \overline{u_j} \times \overline{\phi_1} \right) \left<u_j, U_{f, \phi_2}\right>
	\right.
	\\
	+
	\left.
	\sum_\cuspa \frac{1}{4\pi} \int_{-\infty}^\infty 
	\tth^{(3)} (t; s, \nu_1, \nu_2) 
	\cL\left(s, E_\cuspa (*, 1/2-it) \times \overline{\phi_1}\right)
	\left<E_\cuspa (*, 1/2+it), U_{f, \phi_2}\right> \; dt
	\right\}
	.
\end{multline}

We define
\begin{multline}\label{e:Omega1}
	\Omega^{(1)}_{f, \phi_1, \phi_2}(s)
	:=
	\overline{c_{\phi_1}(-1)}
	\\
	\times
	\left\{
	\tth^{(1)+}(s; \bnu_1, \bnu_2)
	\sum_{\cuspa} 	
	\left<E_\cuspa \left(*, 3/2-s-\bnu_1\right), U_{f, \phi_2}\right> 
	\Res_{z=-s+1 -\bnu_1} 
	\cL\left(s, E_{\cuspa} \left(*, 1/2-z\right) \times \overline{\phi_1}\right)
	\right.
	\\
	\left.
	+
	\tth^{(1)-}(s; \bnu_1, \bnu_2)
	\sum_{\cuspa}
	\left<E_\cuspa \left(*, 3/2-s+\bnu_1\right), U_{f, \phi_2}\right> 
	\Res_{z=-s+1 +\bnu_1} 
	\cL\left(s, E_{\cuspa} \left(*, 1/2-z\right) \times \overline{\phi_1}\right)
	\right\}
\end{multline}
and 
\begin{multline}\label{e:Omega3}
	\Omega^{(3)}_{f, \phi_1, \phi_2}(s)
	:=
	\overline{c_{\phi_1}(1)}
	\\
	\times
	\left\{
	\tth^{(3)+}\left( s; \bnu_1, \bnu_2\right)
	\sum_{\cuspa} \left<E_{\cuspa}(*, 3/2-s-\bnu_1), U_{f, \phi_2} \right>
	\Res_{z=-s+1-\bnu_1} \cL\left(s, E_\cuspa(*, 1/2-z)\times \overline{\phi_1}\right) 
	\right.
	\\
	\left.
	+
	\tth^{(3)-} \left( s; \bnu_1, \bnu_2\right) 	
	\sum_{\cuspa} \left<E_{\cuspa}(*, 3/2-s+\bnu_1), U_{f, \phi_2} \right>
	\Res_{z=-s+1+\bnu_1} \cL\left(s, E_\cuspa(*, 1/2-z)\times \overline{\phi_1}\right) 
	\right\}.
\end{multline}
Finally, we define
\begin{multline}\label{e:M1}
	M^{(1)}_{f, \phi_1, \phi_2}(s)
	:=
	\overline{c_{\phi_1}(-1)}
	\\
	\times
	\left\{
	\frac{(4\pi)^k i^k \sqrt{\pi}}{2}
	\frac{\Gamma\left(2- 2\bnu_1-2s\right)
	\Gamma\left(2s-1\right)
	\Gamma\left(2s-1+2\bnu_1\right)}
	{\Gamma\left(\frac{1}{2}+\frac{k}{2}+\bnu_1\right) 
	\Gamma\left(\frac{1}{2}+\frac{k}{2}-\bnu_1\right)
	\Gamma\left(-s+1+\frac{k}{2}+\bnu_2- \bnu_1\right)
	\Gamma\left(-s+1+\frac{k}{2}-\bnu_2- \bnu_1\right)}
	\right.
	\\
	\times
	H_1^-(s;\bnu_1)
	\sum_{\cuspa} 	
	\left<E_\cuspa \left(*, 3/2-s-\bnu_1\right), U_{f, \phi_2}\right> 
	\Res_{z=-s+1 -\bnu_1} 
	\cL\left(s, E_{\cuspa} \left(*, 1/2-z\right) \times \overline{\phi_1}\right)
	\\
	+
	\frac{(4\pi)^k i^k \sqrt{\pi}}{2}
	\frac{\Gamma\left(2+ 2\bnu_1-2s\right)
	\Gamma\left(2s-1\right)
	\Gamma\left(2s-1- 2\bnu_1\right)}
	{\Gamma\left(\frac{1}{2}+\frac{k}{2}+\bnu_1\right) 
	\Gamma\left(\frac{1}{2}+\frac{k}{2}-\bnu_1\right)
	\Gamma\left(-s+1+\frac{k}{2}+\bnu_2+\bnu_1\right)
	\Gamma\left(-s+1+\frac{k}{2}-\bnu_2+\bnu_1\right)}
	\\
	\left.
	\times
	H_1^+ (s;\bnu_1)
	\sum_{\cuspa}
	\left<E_\cuspa \left(*, 3/2-s+\bnu_1\right), U_{f, \phi_2}\right> 
	\Res_{z=-s+1 +\bnu_1} 
	\cL\left(s, E_{\cuspa} \left(*, 1/2-z\right) \times \overline{\phi_1}\right)
	\right\}
	, 
\end{multline}
and
\begin{multline}\label{e:M3}
	M^{(3)}_{f, \phi_1, \phi_2}(s)
	:=
	\overline{c_{\phi_1}(1)}
	\\
	\times
	\left\{
	\frac{(4\pi)^{k}\sqrt{\pi}}{2}
	\frac{\Gamma\left(2s-1\right)
	\cos\left(\pi(2s+ \bnu_1)\right)}
	{\sin\left(\pi\left(2s+ 2\bnu_1\right)\right)}
	\frac{1} 
	{\Gamma\left(-s+1+\frac{k}{2}-\bnu_1+\bnu_2\right)
	\Gamma\left(-s+1+\frac{k}{2}- \bnu_1-\bnu_2\right) }
	H_1^- (s;\bnu_1)
	\right.
	\\
	\times
	\sum_{\cuspa} \left<E_{\cuspa}(*, 3/2-s-\bnu_1), U_{f, \phi_2} \right>
	\Res_{z=-s+1-\bnu_1} \cL\left(s, E_\cuspa(*, 1/2-z)\times \overline{\phi_1}\right) 
	\\
	+
	\frac{(4\pi)^{k}\sqrt{\pi}}{2}
	\frac{\Gamma\left(2s-1\right)
	\cos\left(\pi(2s- \bnu_1)\right)}
	{\sin\left(\pi\left(2s- 2\bnu_1\right)\right)}
	\frac{1} 
	{\Gamma\left(-s+1+\frac{k}{2}+ \bnu_1+\bnu_2\right)
	\Gamma\left(-s+1+\frac{k}{2}+ \bnu_1-\bnu_2\right) }
	H_1^+ (s;\bnu_1)
	\\
	\times
	\left.
	\sum_{\cuspa} \left<E_{\cuspa}(*, 3/2-s+\bnu_1), U_{f, \phi_2} \right>
	\Res_{z=-s+1+\bnu_1} \cL\left(s, E_\cuspa(*, 1/2-z)\times \overline{\phi_1}\right) 
	\right\}
	.
\end{multline}
Here $\tth^{(1)}$, $\tth^{(2)}$ and $\tth^{(3)}$ are defined in \eqref{e:tth1}, \eqref{e:tth2} and \eqref{e:tth3} respectively; 
$\tth^{(1)\pm}$ and $\tth^{(3)\pm}$ are defined in \eqref{e:tth1pm} and \eqref{e:tth3pm} respectively; 
$H_1^\pm(s; \bnu_1)$ and $H_{2, \ell}(s; \bnu_1)$ are defined in \eqref{e:H1pm} and \eqref{e:H2} respectively. 

Then we have the following proposition.

\begin{prop}\label{prop:tE}
$\tE^{(1)}_{f, \phi_1, \phi_2}(s)$, $\tE^{(2)}_{f, \phi_1, \phi_2}(s)$ have meromorphic continuation to $\Re(s)\geq 1/2$. 
For $1/2\leq \Re(s)<1/2+\epsilon/2$, we have
$$
	\tE^{(1)}_{f, \phi_1, \phi_2}(s)
	=
	E^{(1)}_{f, \phi_1, \phi_2}(s)
	+\Omega^{(1)}_{f, \phi_1, \phi_2}(s)
	+M^{(1)}_{f, \phi_1, \phi_2}(s)
	, 
$$
\begin{multline*}
	\tE^{(2)}_{f, \phi_1, \phi_2}(s)
	=
	E^{(2)}_{f, \phi_1, \phi_2}(s)
	+
	E^{(3)}_{f, \phi_1, \phi_2}(s)
	+
	\Omega^{(3)}_{f, \phi_1, \phi_2}(s)
	+
	M^{(3)}_{f, \phi_1, \phi_2}(s)
	\\
	+
	2
	\sum_{\ell=1}^{\frac{k}{2}} 
	\frac{1}{(\ell-1)!}
	(4\pi)^{k-\ell}
	\frac{\Gamma\left(s-\frac{1}{2}+\ell+\bar{\nu_1}\right) 
	\Gamma\left(s-\frac{1}{2}+\ell-\bar{\nu_1}\right)}
	{\Gamma\left(\frac{k}{2}-\ell+\frac{1}{2}+\bnu_2\right)
	\Gamma\left(\frac{k}{2}-\ell +\frac{1}{2}-\bnu_2\right) }
	\cM^*_{f, \phi_1, \phi_2} \left(-\ell+1, s+\ell-\frac{k}{2}\right)
	\\
	\times 
	H_{2, \ell}(s; \bnu_1)
	.
\end{multline*}

\end{prop}
\begin{proof}
By Proposition \ref{prop:cM3}, for $\Re(s-w-k/2) > 0$, we have
\begin{multline*}
	\cM^{(3)}_{f, \phi_2, \phi_1}(w+1/2, s-w+1/2-k/2)
	=
	\frac{\Gamma\left(w+\frac{1}{2}\right)
	\Gamma\left(\frac{1}{2}-w\right)
	(4\pi)^{\frac{k}{2}-\frac{1}{2}}\sqrt{\pi} }
	{\Gamma\left(\frac{1}{2}+\frac{k}{2}+\bnu_1\right)
	\Gamma\left(\frac{1}{2}+\frac{k}{2}-\bnu_1\right)
	\Gamma\left(s-w+\bnu_2\right) 
	\Gamma\left(s-w -\bnu_2\right)}
	\\
	\times
	\overline{c_{\phi_1}(-1)}
	\left\{
	\sum_j 
	\left<u_j, U_{f, \phi_2}\right> 
	\Gamma\left(s-w-\frac{k}{2}+it_j\right) \Gamma\left(s-w-\frac{k}{2}-it_j\right)
	\cL\left(s, \overline{u_j} \times \overline{\phi_1}\right)
	\right.
	\\
	+
	\sum_{\cuspa} \frac{1}{4\pi} 
	\int_{-\infty}^\infty \left<E_{\cuspa}(*, 1/2+it), U_{f, \phi_2}\right> 
	\\
	\left.
	\times
	\Gamma\left(s-w-\frac{k}{2}+it\right) \Gamma\left(s-w-\frac{k}{2}-it\right)
	\cL\left(s, E_\cuspa(*,1/2-it)\times \overline{\phi_1}\right)
	\; dt
	\right\}
	.
\end{multline*}
Recalling \eqref{e:tE1_SDDS}, 
for $\Re(s) > 2$, we have
\begin{multline*}
	\tE^{(1)}_{f, \phi_1, \phi_2}(s)
	=
	-
	4
	(4\pi)^{\frac{k-1}{2}} i^k 
	\frac{1}{2\pi i} \int_{(\sigma_u = 1/2+2\epsilon)} 
	\frac{h(u/i) u \tan(\pi u)}{\Gamma\left(\frac{k}{2}+\bnu_1 +u\right) \Gamma\left(\frac{k}{2}+\bnu_1-u\right)}
	\\
	\times
	\frac{1}{2\pi i } \int_{(\sigma_w = 1/2-k/2-\epsilon)}
	\frac{\Gamma\left(-w+u\right) \Gamma\left(-w-u\right)
	\Gamma\left(w+\frac{k}{2}-\bnu_1\right) \Gamma\left(w+\frac{k}{2}+\bnu_1\right)}{\Gamma\left(w+\frac{1}{2}\right) \Gamma\left(-w+\frac{1}{2}\right)}
	\\
	\times
	\cM^{(3)}_{f, \phi_2, \phi_1}\left(w+1/2, s-w-k/2+1/2\right)
	\; dw\; du
	\\
	=
	-
	\overline{c_{\phi_1}(-1)}
	\frac{(4\pi)^{k}i^k }{\sqrt{\pi}}
	\frac{1}{2\pi i} \int_{(\sigma_u = 1/2+2\epsilon)} 
	\frac{h(u/i) u \tan(\pi u)}{\Gamma\left(\frac{k}{2}+\bnu_1 +u\right) \Gamma\left(\frac{k}{2}+\bnu_1-u\right)}
	\\
	\times
	\frac{1}{2\pi i } \int_{(\sigma_w = 1/2-k/2-\epsilon)}
	\frac{\Gamma\left(-w+u\right) \Gamma\left(-w-u\right)
	\Gamma\left(w+\frac{k}{2}-\bnu_1\right) \Gamma\left(w+\frac{k}{2}+\bnu_1\right)}
	{\Gamma\left(\frac{1}{2}+\frac{k}{2}+\bnu_1\right)
	\Gamma\left(\frac{1}{2}+\frac{k}{2}-\bnu_1\right)
	\Gamma\left(s-w+\bnu_2\right) 
	\Gamma\left(s-w -\bnu_2\right)}
	\\
	\times
	\left\{
	\sum_j 
	\left<u_j, U_{f, \phi_2}\right> 
	\Gamma\left(s-w-\frac{k}{2}+it_j\right) \Gamma\left(s-w-\frac{k}{2}-it_j\right)
	\cL\left(s, \overline{u_j} \times \overline{\phi_1}\right)
	\right.
	\\
	+
	\sum_{\cuspa} \frac{1}{4\pi} 
	\int_{-\infty}^\infty \left<E_{\cuspa}(*, 1/2+it), U_{f, \phi_2}\right> 
	\Gamma\left(s-w-\frac{k}{2}+it\right) \Gamma\left(s-w-\frac{k}{2}-it\right)
	\\
	\times
	\left.
	\cL\left(s, E_\cuspa(*,1/2-it)\times \overline{\phi_1}\right)
	\; dt
	\right\}
	\; dw\; du
	.
\end{multline*}
For any $s\in \C$, the series and integrals converge absolutely. 
For any $t$ with $|\Im(t)|< \epsilon/2$, for $\Re(s)\geq 1/2$, let
\begin{multline*}
	\tth^{(1)}_0 (t; s, \nu_1,\nu_2)
	:=
	-
	\frac{(4\pi)^{k}i^k }{\sqrt{\pi}}
	\frac{1}{2\pi i} \int_{(\sigma_u = 1/2+2\epsilon)} 
	\frac{h(u/i) u \tan(\pi u)}{\Gamma\left(\frac{k}{2}+\bnu_1 +u\right) \Gamma\left(\frac{k}{2}+\bnu_1-u\right)}
	\\
	\times
	\frac{1}{2\pi i } \int_{(\sigma_w = 1/2-k/2-\epsilon)}
	\frac{\Gamma\left(-w+u\right) \Gamma\left(-w-u\right)
	\Gamma\left(w+\frac{k}{2}-\bnu_1\right) \Gamma\left(w+\frac{k}{2}+\bnu_1\right)}
	{\Gamma\left(\frac{1}{2}+\frac{k}{2}+\bnu_1\right)
	\Gamma\left(\frac{1}{2}+\frac{k}{2}-\bnu_1\right)
	\Gamma\left(s-w+\bnu_2\right) 
	\Gamma\left(s-w -\bnu_2\right)}
	\\
	\times
	\Gamma\left(s-w-\frac{k}{2}+it\right) \Gamma\left(s-w-\frac{k}{2}-it\right)
	\; dw\; du
	.
\end{multline*}
By changing the variable $w$ to $w+k/2$, we get
\begin{multline*}
	\tth^{(1)}_0(t; s, \nu_1,\nu_2)
	=
	-
	\frac{(4\pi)^{k}i^k }{\sqrt{\pi}}
	\frac{1}{2\pi i} \int_{(\sigma_u = 1/2+2\epsilon)} 
	\frac{h(u/i) u \tan(\pi u)}{\Gamma\left(\frac{k}{2}+\bnu_1 +u\right) \Gamma\left(\frac{k}{2}+\bnu_1-u\right)}
	\\
	\times
	\frac{1}{2\pi i } \int_{(\sigma_w = 1/2-\epsilon)}
	\frac{\Gamma\left(-w+\frac{k}{2}+u\right) \Gamma\left(-w+\frac{k}{2}-u\right)
	\Gamma\left(w-\bnu_1\right) \Gamma\left(w+\bnu_1\right)}
	{\Gamma\left(\frac{1}{2}+\frac{k}{2}+\bnu_1\right)
	\Gamma\left(\frac{1}{2}+\frac{k}{2}-\bnu_1\right)
	\Gamma\left(s-w+\frac{k}{2}+\bnu_2\right) 
	\Gamma\left(s-w+\frac{k}{2}-\bnu_2\right)}
	\\
	\times
	\Gamma\left(s-w+it\right) \Gamma\left(s-w-it\right)
	\; dw\; du
	\\
	=
	\tth^{(1)}(t; s, \nu_1,\nu_2)
	.
\end{multline*}
For $\Re(s) = 1+\epsilon/2$, we have
\begin{multline*}
	\tth^{(1)}((s-1\pm \bnu_1)/i; s, \nu_1, \nu_2)
	=
	-
	\frac{(4\pi)^{k} i^k }{\sqrt{\pi}} 
	\frac{1}{2\pi i} \int_{(\sigma_u = 1/2+2\epsilon)} 
	\frac{h(u/i) u \tan(\pi u)}{\Gamma\left(\frac{k}{2}+\bnu_1 +u\right) 
	\Gamma\left(\frac{k}{2}+\bnu_1-u\right)}
	\\
	\times
	\frac{1}{2\pi i } \int_{(\sigma_w = 1/2-\epsilon)}
	\frac{\Gamma\left(-w+\frac{k}{2}+u\right) 
	\Gamma\left(-w+\frac{k}{2}-u\right)
	\Gamma\left(w-\bnu_1\right) \Gamma\left(w+\bnu_1\right)}
	{\Gamma\left(\frac{1}{2}+\frac{k}{2}+\bnu_1\right)
	\Gamma\left(\frac{1}{2}+\frac{k}{2}-\bnu_1\right)
	\Gamma\left(s-w+\frac{k}{2}+\bnu_2\right) 
	\Gamma\left(s-w+\frac{k}{2}-\bnu_2\right)}
	\\
	\times
	\Gamma\left(2s-w-1\pm \bnu_1\right) \Gamma\left(-w+1\mp \bnu_1\right)
	\; dw\; du
	.
\end{multline*}

Then, for $\Re(s) \geq 1/2$, getting the meromorphic continuation of the integral for the continuous spectrum, and applying the functional equation of the Eisenstein series over cusps, we have
\begin{multline*}
	\tE^{(1)}_{f, \phi_1, \phi_2}(s)
	=
	\overline{c_{\phi_1}(-1)}
	\left\{
	\sum_j 
	\tth^{(1)}(t_j; s, \nu_1, \nu_2)
	\left<u_j, U_{f, \phi_2}\right> 
	\cL\left(s, \overline{u_j} \times \overline{\phi_1}\right)
	\right.
	\\
	+
	\sum_{\cuspa} \frac{1}{4\pi} 
	\int_{-\infty}^\infty 
	\tth^{(1)}(t; s, \nu_1, \nu_2)
	\left<E_{\cuspa}(*, 1/2+it), U_{f, \phi_2}\right> 
	\cL\left(s, E_\cuspa(*,1/2-it)\times \overline{\phi_1}\right)
	\; dt
	\\
	+
	\tth^{(1)}\left( (s-1+\bnu_1)/i; s, \nu_1, \nu_2\right) 
	\sum_{\cuspa} 	
	\left<E_\cuspa \left(*, 3/2-s-\bnu_1\right), U_{f, \phi_2}\right> 
	\Res_{z=-s+1 -\bnu_1} 
	\cL\left(s, E_{\cuspa} \left(*, 1/2-z\right) \times \overline{\phi_1}\right)
	\\
	\left.
	+
	\tth^{(1)}\left( (s-1-\bnu_1)/i; s, \nu_1, \nu_2\right)
	\left<E_\cuspa \left(*, 3/2-s+\bnu_1\right), U_{f, \phi_2}\right> 
	\Res_{z=-s+1 +\bnu_1} 
	\cL\left(s, E_{\cuspa} \left(*, 1/2-z\right) \times \overline{\phi_1}\right)
	\right\}
	.
\end{multline*}


Again, by Proposition \ref{prop:cM3}, for $\Re(w) > 0$, we have
\begin{multline*}
	\cM^{(3)}_{f, \phi_1, \phi_2}(s-w+1/2-k/2, w+1/2)
	\\
	=
	\overline{c_{\phi_2}(-1)}
	\frac{\Gamma\left(s-w+\frac{1}{2}-\frac{k}{2}\right)
	\Gamma(-s+w+\frac{1}{2}+\frac{k}{2})
	(4\pi)^{\frac{k}{2}-\frac{1}{2}}\sqrt{\pi} }
	{\Gamma\left(\frac{1}{2}+\frac{k}{2}+\bnu_2\right)
	\Gamma\left(\frac{1}{2}+\frac{k}{2}-\bnu_2\right)
	\Gamma\left(w+\frac{k}{2}+\bnu_1\right) 
	\Gamma\left(w+\frac{k}{2}-\bnu_1\right)}
	\\
	\times
	\left\{
	\sum_j 
	\left<u_j, U_{f, \phi}\right> 
	\Gamma\left(w+it_j\right) \Gamma\left(w-it_j\right)
	\cL\left(s, \overline{u_j} \times \overline{\phi_2}\right)
	\right.
	\\
	\left.
	+
	\sum_{\cuspa} \frac{1}{4\pi} 
	\int_{-\infty}^\infty \left<E_{\cuspa}(*, 1/2+it), U_{f, \phi_1}\right> 
	\Gamma\left(w+it\right) \Gamma\left(w-it\right)
	\cL\left(s, E_\cuspa(*,1/2-it)\times \overline{\phi_2}\right)
	\; dt
	\right\}
	.
\end{multline*}
Recalling \eqref{e:tE3}, for $\Re(s)>1+k/2+\epsilon$, we have
\begin{multline*}
	\tE^{(3)}_{f, \phi_1, \phi_2}(s)
	\\
	=
	-
	(4\pi)^{\frac{k-1}{2}} i^k \frac{4}{\pi} 
	\frac{1}{2\pi i} \int_{(\sigma_u = 1/2+2\epsilon)} 
	\frac{h(u/i) u 
	\Gamma\left(\frac{1}{2}+u\right) \Gamma\left(\frac{1}{2}-u\right)}
	{\Gamma\left(\frac{k}{2}+\bnu_1 +u\right) \Gamma\left(\frac{k}{2}+\bnu_1-u\right)}
	\frac{1}{2\pi i } \int_{(\sigma_w = 1/2+\epsilon)}
	\frac{\Gamma\left(-w+u\right)}
	{\Gamma\left(1+w+u\right)}
	\\
	\times
	\Gamma\left(w+\frac{k}{2}+\bar{\nu_1}\right) 
	\Gamma\left(w+\frac{k}{2}-\bar{\nu_1}\right)
	\cM^{(3)}_{f, \phi_1, \phi_2}\left(s-w+\frac{1}{2}-\frac{k}{2}, w+\frac{1}{2}\right)
	\; dw\; du
	\\
	=
	-
	\overline{c_{\phi_2}(-1)}
	\frac{(4\pi)^{k} i^k }{\pi\sqrt{\pi} }
	\frac{1}{2\pi i} \int_{(\sigma_u = 1/2+2\epsilon)} 
	\frac{h(u/i) u 
	\Gamma\left(\frac{1}{2}+u\right) \Gamma\left(\frac{1}{2}-u\right)}
	{\Gamma\left(\frac{k}{2}+\bnu_1 +u\right) \Gamma\left(\frac{k}{2}+\bnu_1-u\right)}
	\\
	\times
	\frac{1}{2\pi i } \int_{(\sigma_w = 1/2+\epsilon)}
	\frac{\Gamma\left(-w+u\right)}
	{\Gamma\left(1+w+u\right)}
	\frac{\Gamma\left(s-w+\frac{1}{2}-\frac{k}{2}\right)
	\Gamma(-s+w+\frac{1}{2}+\frac{k}{2})}
	{\Gamma\left(\frac{1}{2}+\frac{k}{2}+\bnu_2\right)
	\Gamma\left(\frac{1}{2}+\frac{k}{2}-\bnu_2\right)}
	\\
	\times
	\left\{
	\sum_j 
	\left<u_j, U_{f, \phi}\right> 
	\Gamma\left(w+it_j\right) \Gamma\left(w-it_j\right)
	\cL\left(s, \overline{u_j} \times \overline{\phi_2}\right)
	\right.
	\\
	\left.
	+
	\sum_{\cuspa} \frac{1}{4\pi} 
	\int_{-\infty}^\infty \left<E_{\cuspa}(*, 1/2+it), U_{f, \phi_1}\right> 
	\Gamma\left(w+it\right) \Gamma\left(w-it\right)
	\cL\left(s, E_\cuspa(*,1/2-it)\times \overline{\phi_2}\right)
	\; dt
	\right\}
	\; dw\; du
	.
\end{multline*}

For $\Re(s)>1+k/2+\epsilon$, let
\begin{multline*}
	\tth^{(2)}_0(t; s, \nu_1, \nu_2)
	\\
	:=
	-
	\frac{(4\pi)^{k} i^k }{\pi\sqrt{\pi} }
	\frac{1}{2\pi i} \int_{(\sigma_u = 1/2+2\epsilon)} 
	\frac{h(u/i) u 
	\Gamma\left(\frac{1}{2}+u\right) \Gamma\left(\frac{1}{2}-u\right)}
	{\Gamma\left(\frac{k}{2}+\bnu_1 +u\right) \Gamma\left(\frac{k}{2}+\bnu_1-u\right)}
	\frac{1}{\Gamma\left(\frac{1}{2}+\frac{k}{2}+\bnu_2\right)
	\Gamma\left(\frac{1}{2}+\frac{k}{2}-\bnu_2\right)}
	\\
	\times
	\frac{1}{2\pi i } \int_{(\sigma_w = 1/2+\epsilon)}
	\frac{\Gamma\left(-w+u\right)
	\Gamma\left(s-w+\frac{1}{2}-\frac{k}{2}\right)
	\Gamma(-s+w+\frac{1}{2}+\frac{k}{2})
	\Gamma\left(w+it\right) \Gamma\left(w-it\right)}
	{\Gamma\left(1+w+u\right)}
	\; dw\; du
	.
\end{multline*}
Then 
\begin{multline*}
	\tth^{(2)}_0(t; s, \nu_1, \nu_2)
 	\\
	=
	-
	\frac{(4\pi)^{k} }{\pi\sqrt{\pi} }
	\frac{1}{2\pi i} \int_{(\sigma_u = 1/2+2\epsilon)} 
	\frac{h(u/i) u 
	\Gamma\left(\frac{1}{2}+u\right) \Gamma\left(\frac{1}{2}-u\right)}
	{\Gamma\left(\frac{k}{2}+\bnu_1 +u\right) \Gamma\left(\frac{k}{2}+\bnu_1-u\right)}
	\frac{1}{\Gamma\left(\frac{1}{2}+\frac{k}{2}+\bnu_2\right)
	\Gamma\left(\frac{1}{2}+\frac{k}{2}-\bnu_2\right)}
	\\
	\times
	\frac{1}{2\pi i } \int_{(\sigma_w = 1/2+\epsilon)}
	\frac{\Gamma\left(-w+u\right)
	\Gamma\left(\frac{1}{2}+s-w\right)
	\Gamma\left(\frac{1}{2}-s+w\right)
	\Gamma\left(w+it\right) \Gamma\left(w-it\right)}
	{\Gamma\left(1+w+u\right)}
	\; dw
	\; du
	.
\end{multline*}
By Theorem 2.4.3 in \cite{AAR}, we have
\begin{multline*}
	\tth^{(2)}_0(t; s, \nu_1, \nu_2)
	=
	\sum_{\ell=0}^{\frac{k}{2}}
	(-1)^\ell
	\frac{\Gamma\left(-\frac{1}{2}+s-\ell +it\right) \Gamma\left(-\frac{1}{2}+s-\ell -it\right)}
	{\Gamma\left(\frac{1}{2}+\frac{k}{2}+\bnu_2\right)
	\Gamma\left(\frac{1}{2}+\frac{k}{2}-\bnu_2\right)}
	\\
	\times
	\frac{(4\pi)^{k} }{\pi\sqrt{\pi} }
	\frac{1}{2\pi i} \int_{(\sigma_u = 1/2+2\epsilon)} 
	\frac{h(u/i) u 
	\Gamma\left(\frac{1}{2}+u\right) \Gamma\left(\frac{1}{2}-u\right)}
	{\Gamma\left(\frac{k}{2}+\bnu_1 +u\right) \Gamma\left(\frac{k}{2}+\bnu_1-u\right)}
	\frac{\Gamma\left(\frac{1}{2}-s+\ell+u\right)}
	{\Gamma\left(\frac{1}{2}+s-\ell +u\right)}
	\; du
	\\
	-
	\frac{\Gamma\left(\frac{1}{2}+s+it\right)
	\Gamma\left(\frac{1}{2}+s-it\right)}
	{\Gamma\left(\frac{1}{2}+\frac{k}{2}+\bnu_2\right)
	\Gamma\left(\frac{1}{2}+\frac{k}{2}-\bnu_2\right)}
	\frac{(4\pi)^{k} }{\pi\sqrt{\pi} }
	\\
	\times
	\frac{1}{2\pi i} \int_{(\sigma_u = 1/2+2\epsilon)} 
	\frac{h(u/i) u 
	\Gamma\left(\frac{1}{2}+u\right) \Gamma\left(\frac{1}{2}-u\right)}
	{\Gamma\left(\frac{k}{2}+\bnu_1 +u\right) \Gamma\left(\frac{k}{2}+\bnu_1-u\right)}
	\frac{\Gamma\left(\frac{1}{2}-s+u\right)}
	{\Gamma\left(u+\frac{1}{2}+s\right)
	\left(u+it\right) \left(u-it\right)}
	\; du
	.
\end{multline*}

To get a meromorphic continuation to $\Re(s) \geq 1/2$, 
let $C$ and $C_\ell$ for $0\leq \ell \leq k/2$ be a path running from $-\infty$ to $+\infty$ in such a way that 
the sequence of increasing poles and the sequence of decreasing poles are separated. 
Then we get
\begin{multline*}
	\tth^{(2)}_0 (t; s, \nu_1, \nu_2)
	=
	\sum_{\ell=0}^{\frac{k}{2}}
	(-1)^\ell
	\frac{\Gamma\left(-\frac{1}{2}+s-\ell +it\right) \Gamma\left(-\frac{1}{2}+s-\ell -it\right)}
	{\Gamma\left(\frac{1}{2}+\frac{k}{2}+\bnu_2\right)
	\Gamma\left(\frac{1}{2}+\frac{k}{2}-\bnu_2\right)}
	\\
	\times
	\frac{(4\pi)^{k} }{\pi\sqrt{\pi} }
	\left\{
	-
	\sum_{\ell_1=0}^{\frac{k}{2}-\ell} 
	\frac{(-1)^{\ell_1}}{\ell_1!}
	\frac{h\left( (s-\ell-\ell_1 -1/2)/i\right) \left(s-\ell-\ell_1-\frac{1}{2}\right) 
	\Gamma\left(s -\ell_1-\ell \right) \Gamma\left(1-s+\ell_1+\ell\right) }
	{\Gamma\left(\frac{k}{2}+\bnu_1 +s-\ell_1-\ell -\frac{1}{2}\right) 
	\Gamma\left(\frac{k}{2}+\bnu_1 -s+\ell_1+\ell +\frac{1}{2}\right)
	\Gamma\left(2s-2\ell-\ell_1\right)}
	\right.
	\\
	\left.
	+
	\frac{1}{2\pi i} \int_{C_\ell} 
	\frac{h(u/i) u 
	\Gamma\left(\frac{1}{2}+u\right) \Gamma\left(\frac{1}{2}-u\right)}
	{\Gamma\left(\frac{k}{2}+\bnu_1 +u\right) \Gamma\left(\frac{k}{2}+\bnu_1-u\right)}
	\frac{\Gamma\left(\frac{1}{2}-s+\ell+u\right)}
	{\Gamma\left(\frac{1}{2}+s-\ell +u\right)}
	\; du
	\right\}
	\\
	-
	\frac{\Gamma\left(\frac{1}{2}+s+it\right)
	\Gamma\left(\frac{1}{2}+s-it\right)}
	{\Gamma\left(\frac{1}{2}+\frac{k}{2}+\bnu_2\right)
	\Gamma\left(\frac{1}{2}+\frac{k}{2}-\bnu_2\right)}
	\frac{(4\pi)^{k} }{\pi\sqrt{\pi} }
	\\
	\times
	\left\{
	\sum_{\ell_1=0}^{\frac{k}{2}}
	\frac{(-1)^{\ell_1}}{\ell_1!}
	\frac{h\left((s-1/2-\ell_1)/i\right) \left(s-\ell_1-\frac{1}{2}\right)
	\Gamma\left(s-\ell_1\right) \Gamma\left(1-s+\ell_1\right)}
	{\Gamma\left(\frac{k}{2}+\bnu_1 +s-\ell_1-\frac{1}{2} \right) 
	\Gamma\left(\frac{k}{2}+\bnu_1-s+\ell_1+\frac{1}{2}\right)}
	\right.
	\\
	\times 
	\frac{1}
	{\Gamma\left(2s-\ell_1\right)
	\left(s-\ell_1-\frac{1}{2}+it\right) \left(s-\ell_1-\frac{1}{2}-it\right)}
	\\
	\left.
	+
	\frac{1}{2\pi i} \int_{C} 
	\frac{h(u/i) u 
	\Gamma\left(\frac{1}{2}+u\right) \Gamma\left(\frac{1}{2}-u\right)}
	{\Gamma\left(\frac{k}{2}+\bnu_1 +u\right) \Gamma\left(\frac{k}{2}+\bnu_1-u\right)}
	\frac{\Gamma\left(\frac{1}{2}-s+u\right)}
	{\Gamma\left(u+\frac{1}{2}+s\right)
	\left(u+it\right) \left(u-it\right)}
	\; du
	\right\}
	.
\end{multline*}
Then we have a meromorphic continuation to $\Re(s)\geq 1/2$.
Taking $s$ with $1/2\leq \Re(s) \leq 1/2+\epsilon/2$, we move the $u$ line of integration to $\sigma_u = 1/2+2\epsilon$. 
This gives us
\begin{multline*}
	\tth^{(2)}_0(t; s, \nu_1, \nu_2)
	=
	\sum_{\ell=0}^{\frac{k}{2}}
	(-1)^\ell
	\frac{\Gamma\left(-\frac{1}{2}+s-\ell +it\right) \Gamma\left(-\frac{1}{2}+s-\ell -it\right)}
	{\Gamma\left(\frac{1}{2}+\frac{k}{2}+\bnu_2\right)
	\Gamma\left(\frac{1}{2}+\frac{k}{2}-\bnu_2\right)}
	\\
	\times
	\frac{(4\pi)^{k} }{\pi\sqrt{\pi} }
	\left\{
	-
	\sum_{\ell_1=0}^{\frac{k}{2}-\ell} 
	\frac{(-1)^{\ell_1}}{\ell_1!}
	\frac{h\left( (s-\ell-\ell_1 -1/2)/i\right) \left(s-\ell-\ell_1-\frac{1}{2}\right) 
	\Gamma\left(s -\ell_1-\ell \right) \Gamma\left(1-s+\ell_1+\ell\right) }
	{\Gamma\left(\frac{k}{2}+\bnu_1 +s-\ell_1-\ell -\frac{1}{2}\right) 
	\Gamma\left(\frac{k}{2}+\bnu_1 -s+\ell_1+\ell +\frac{1}{2}\right)
	\Gamma\left(2s-2\ell-\ell_1\right)}
	\right.
	\\
	\left.
	+
	\frac{1}{2\pi i} \int_{(\sigma_u=1/2+2\epsilon)} 
	\frac{h(u/i) u 
	\Gamma\left(\frac{1}{2}+u\right) \Gamma\left(\frac{1}{2}-u\right)}
	{\Gamma\left(\frac{k}{2}+\bnu_1 +u\right) \Gamma\left(\frac{k}{2}+\bnu_1-u\right)}
	\frac{\Gamma\left(\frac{1}{2}-s+\ell+u\right)}
	{\Gamma\left(\frac{1}{2}+s-\ell +u\right)}
	\; du
	\right\}
	\\
	-
	\frac{\Gamma\left(\frac{1}{2}+s+it\right)
	\Gamma\left(\frac{1}{2}+s-it\right)}
	{\Gamma\left(\frac{1}{2}+\frac{k}{2}+\bnu_2\right)
	\Gamma\left(\frac{1}{2}+\frac{k}{2}-\bnu_2\right)}
	\frac{(4\pi)^{k} }{\pi\sqrt{\pi} }
	\\
	\times
	\left\{
	\sum_{\ell_1=0}^{\frac{k}{2}}
	\frac{(-1)^{\ell_1}}{\ell_1!}
	\frac{h\left((s-1/2-\ell_1)/i\right) \left(s-\ell_1-\frac{1}{2}\right)
	\Gamma\left(s-\ell_1\right) \Gamma\left(1-s+\ell_1\right)}
	{\Gamma\left(\frac{k}{2}+\bnu_1 +s-\ell_1-\frac{1}{2} \right) 
	\Gamma\left(\frac{k}{2}+\bnu_1-s+\ell_1+\frac{1}{2}\right)}
	\right.
	\\
	\times 
	\frac{1}
	{\Gamma\left(2s-\ell_1\right)
	\left(s-\ell_1-\frac{1}{2}+it\right) \left(s-\ell_1-\frac{1}{2}-it\right)}
	\\
	\left.
	+
	\frac{1}{2\pi i} \int_{(\sigma_u = 1/2+2\epsilon)} 
	\frac{h(u/i) u 
	\Gamma\left(\frac{1}{2}+u\right) \Gamma\left(\frac{1}{2}-u\right)}
	{\Gamma\left(\frac{k}{2}+\bnu_1 +u\right) \Gamma\left(\frac{k}{2}+\bnu_1-u\right)}
	\frac{\Gamma\left(\frac{1}{2}-s+u\right)}
	{\Gamma\left(u+\frac{1}{2}+s\right)
	\left(u+it\right) \left(u-it\right)}
	\; du
	\right\}
	\\
	=
	\tth^{(2)}(t; s, \nu_1, \nu_2)	
	.
\end{multline*}
Then, for $1/2 \leq \Re(s) < 1/2+\epsilon/2$, we have
\begin{multline*}
	\tE^{(3)}_{f, \phi_1, \phi_2}(s)
	=
	-
	\overline{c_{\phi_2}(-1)}
	\left\{
	\sum_j 
	\tth^{(2)}(t_j; s, \nu_1, \nu_2)
	\left<u_j, U_{f, \phi}\right> 
	\cL\left(s, \overline{u_j} \times \overline{\phi_2}\right)
	\right.
	\\
	\left.
	+
	\sum_{\cuspa} \frac{1}{4\pi} 
	\tth^{(2)}(t; s, \nu_1, \nu_2)
	\int_{-\infty}^\infty \left<E_{\cuspa}(*, 1/2+it), U_{f, \phi_1}\right> 
	\cL\left(s, E_\cuspa(*,1/2-it)\times \overline{\phi_2}\right)
	\; dt
	\right\}
	.
\end{multline*}

By \eqref{e:cM*}, for $\Re(s-w+1/2-k/2) >1 $ and $\Re(w+1/2) > 1$, 
\begin{multline*}
	\cM^*_{f, \phi_1, \phi_2} \left(s-w+\frac{1}{2}-\frac{k}{2}, w+\frac{1}{2}\right)
	\\
	=
	\frac{\Gamma\left(s-w+\bnu_2\right)
	\Gamma\left(s-w-\bnu_2\right)
	(4\pi)^{\frac{1}{2}-s+w}}
	{\Gamma\left(s-w+\frac{1}{2}-\frac{k}{2}\right)}
	\left\{ \cM_{f, \phi_1, \phi_2}\left(s-w+\frac{1}{2}-\frac{k}{2}, w+\frac{1}{2}\right)
	\right.
	\\
	\left.
	+
	\cM_{f, \phi_1, \phi_2}^{(3)}\left(s-w+\frac{1}{2}-\frac{k}{2}, w+\frac{1}{2}\right)\right\}, 
\end{multline*}
where
$$
	\cM_{f, \phi_1, \phi_2}\left(s-w+\frac{1}{2}-\frac{k}{2}, w+\frac{1}{2}\right)
	=
	\zeta(2s)\sum_{m, n\geq 1} \frac{a(n+m) \overline{c_{\phi_1}(n)} \overline{c_{\phi_2}(m)}}
	{m^{s-w-\frac{1}{2}} n^{w+\frac{k}{2}-\frac{1}{2}}}
$$
and
\begin{multline*}
	\cM_{f, \phi_1, \phi_2}^{(3)}\left(s-w+\frac{1}{2}-\frac{k}{2}, w+\frac{1}{2}\right)
	=
	\frac{\overline{c_{\phi_2}(-1)} 
	\Gamma\left(s-w+\frac{1}{2}-\frac{k}{2}\right) 
	\Gamma\left(\frac{1}{2}-s+w+\frac{k}{2}\right)}
	{\Gamma\left(\frac{1}{2}+\frac{k}{2}+\bnu_2\right) 
	\Gamma\left(\frac{1}{2}+\frac{k}{2}-\bnu_2\right)}
	\\
	\times
	\zeta(2s) 
	\sum_{n=1}^\infty 
	\frac{\overline{c_{\phi_1}(n)}}{n^{w+\frac{k}{2}-\frac{1}{2}}}
	\sum_{m=1}^{n-1} \frac{a(n-m) \overline{c_{\phi_2}(m)}}
	{m^{s-w-\frac{1}{2}}}
	.
\end{multline*}
By \eqref{e:cM*}, for $\Re(s)-\Re(w)<0$ and for sufficiently large $\Re(w)$, 
we have
\begin{multline*}
	\cM^*_{f, \phi_1, \phi_2} \left(s-w+\frac{1}{2}-\frac{k}{2}, w+\frac{1}{2}\right)
	=
	\left(4\pi\right)^{-s+w+\frac{k}{2}}
	\sqrt{\pi}
	\Gamma\left(-s+w+\frac{1}{2}+\frac{k}{2}\right)
	\\
	\times
	\left\{
	\sum_j 
	\frac{\Gamma\left(s-w-\frac{k}{2}-it_j\right) \Gamma\left(s-w-\frac{k}{2}+it_j\right) }
	{\Gamma\left(\frac{1}{2}-it_j\right) \Gamma\left(\frac{1}{2}+it_j\right)}
	(-1)^{\alpha_j} \cL\left(s, \overline{u_j} \times \overline{\phi_1}\right) 
	\left<u_j, U_{f, \phi_2}\right>
	\right.
	\\
	+
	\sum_\cuspa \frac{1}{4\pi} \int_{-\infty}^\infty 
	\frac{\Gamma\left(s-w-\frac{k}{2}-it\right) \Gamma\left(s-w-\frac{k}{2}+it\right) }
	{\Gamma\left(\frac{1}{2}-it\right) \Gamma\left(\frac{1}{2}+it\right)}
	\cL\left(s, E_{\cuspa} (*, 1/2-it) \times \overline{\phi_1}\right)
	\left<E_{\cuspa} (*, 1/2+it), U_{f, \phi_2}\right> \; dt
	\\
	-
	\sum_{\ell=0}^{\lfloor\Re\left(-s+w+\frac{k}{2}\right)\rfloor}
	\frac{(-1)^{\ell}}{\ell!} 
	\frac{\Gamma\left(2s-2w-k+\ell\right)} 
	{\Gamma\left(\frac{1}{2}-\ell+\frac{k}{2} -s+w\right)
	\Gamma\left(\frac{1}{2}+\ell-\frac{k}{2}+s-w\right)}
	\\
	\times
	\frac{1}{2}
	\sum_\cuspa 
	\left(
	\cL\left(s, E_\cuspa \left(*, \frac{1}{2}-\ell+\frac{k}{2}-s+w\right) \times \overline{\phi_1}\right)
	\left<E_\cuspa\left(*, \frac{1}{2}+\ell-\frac{k}{2}+s-w\right), U_{f, \phi_2}\right>
	\right.
	\\
	\left.\left.
	+
	\cL\left(s, E_\cuspa\left(\frac{1}{2}+\ell-\frac{k}{2}+s-w\right)\times \overline{\phi_1}\right)
	\left<E_\cuspa\left(*, \frac{1}{2}-\ell +\frac{k}{2}-s+w\right), U_{f, \phi_2}\right>
	\right)
	\right\}
	.
\end{multline*}

Recalling \eqref{e:tE4}, for $\Re(s)>1+k/2+\epsilon$, we have
\begin{multline*}
	\tE^{(4)}_{f, \phi_1, \phi_2}(s)
	\\
	=
	(4\pi)^{\frac{k-1}{2}} i^k \frac{4}{\pi} 
	\frac{1}{2\pi i} \int_{(\sigma_u = 1/2+2\epsilon)} 
	\frac{h(u/i) u 
	\Gamma\left(\frac{1}{2}+u\right) \Gamma\left(\frac{1}{2}-u\right)}
	{\Gamma\left(\frac{k}{2}+\bnu_1 +u\right) \Gamma\left(\frac{k}{2}+\bnu_1-u\right)}
	\frac{1}{2\pi i } \int_{(\sigma_w = 1/2+\epsilon)}
	\frac{\Gamma\left(-w+u\right)}
	{\Gamma\left(1+w+u\right)}
	\\
	\times
	\frac{\Gamma\left(w+\frac{k}{2}+\bar{\nu_1}\right) 
	\Gamma\left(w+\frac{k}{2}-\bar{\nu_1}\right)
	\Gamma\left(s-w+\frac{1}{2}-\frac{k}{2}\right) 
	(4\pi)^{s-w-\frac{1}{2}}}
	{\Gamma\left(s-w+\bnu_2\right)
	\Gamma\left(s-w-\bnu_2\right) }
	\\
	\times
	\cM^*_{f, \phi_1, \phi_2} \left(s-w+\frac{1}{2}-\frac{k}{2}, w+\frac{1}{2}\right)
	\; dw \; du
	.
\end{multline*}
Our aim is to get a meromorphic continuation of $\tE^{(4)}_{f,\phi_1, \phi_2}(s)$ to $\Re(s)\geq 1/2$. 
There are poles passed over as we move $\Re(s)$ from $1+k/2+\epsilon'$ to $1/2$. 
We first curve the $w$ line of integration to pass over those poles. 
With the residues and the original integral, we obtain the meromorphic continuation to $\Re(s)\geq 1/2$.
Now take $s$ with $\Re(s)\geq 1/2$, and move back the $w$ line of integration to $\sigma_w = 1/2+\epsilon$. 
We again pass over the poles, and get the following formula. 

For $1/2\leq \Re(s) \leq 1/2+\epsilon/2$, for $0\leq \ell \leq k/2$, let
\begin{multline*}
	\tth^{(3)R}_\ell (t; s, \nu_1, \nu_2)
	:=
	-
	\frac{1}{\ell!}
	\frac{(4\pi)^{k} i^k}{\pi\sqrt{\pi}}
	\frac{
	\Gamma\left(s+\ell \pm it+\bar{\nu_1}\right) 
	\Gamma\left(s+\ell \pm it-\bar{\nu_1}\right)
	\Gamma\left(-\ell \mp 2it\right)} 
	{\Gamma\left(\frac{k}{2}-\ell \mp it+\bnu_2\right)
	\Gamma\left(\frac{k}{2}-\ell \mp it-\bnu_2\right) }
	\\
	\times
	\frac{1}{2\pi i} \int_{(\sigma_u = 1/2+2\epsilon)} 
	\frac{h(u/i) u 
	\Gamma\left(\frac{1}{2}+u\right) \Gamma\left(\frac{1}{2}-u\right)}
	{\Gamma\left(\frac{k}{2}+\bnu_1 +u\right) \Gamma\left(\frac{k}{2}+\bnu_1-u\right)}
	\frac{\Gamma\left(-s+\frac{k}{2}-\ell \mp it+u\right)}
	{\Gamma\left(1+s-\frac{k}{2}+\ell \pm it+u\right)}
	\; du
	.
\end{multline*}
After changing the ordering of integrals and series in $\tE^{(4)}_{f, \phi_1, \phi_2}(s)$, 
for $1/2\leq \Re(s) \leq 1/2+\epsilon/2$, we get
\begin{multline*}
	\tE^{(4)}_{f, \phi_1, \phi_2}(s)
	\\
	=
	(4\pi)^{\frac{k-1}{2}} i^k \frac{4}{\pi} 
	\frac{1}{2\pi i} \int_{(\sigma_u = 1/2+2\epsilon)} 
	\frac{h(u/i) u 
	\Gamma\left(\frac{1}{2}+u\right) \Gamma\left(\frac{1}{2}-u\right)}
	{\Gamma\left(\frac{k}{2}+\bnu_1 +u\right) \Gamma\left(\frac{k}{2}+\bnu_1-u\right)}
	\frac{1}{2\pi i } \int_{(\sigma_w = 1/2+\epsilon)}
	\frac{\Gamma\left(-w+u\right)}
	{\Gamma\left(1+w+u\right)}
	\\
	\times
	\frac{\Gamma\left(w+\frac{k}{2}+\bar{\nu_1}\right) 
	\Gamma\left(w+\frac{k}{2}-\bar{\nu_1}\right)
	\Gamma\left(s-w+\frac{1}{2}-\frac{k}{2}\right) 
	(4\pi)^{s-w-\frac{1}{2}}}
	{\Gamma\left(s-w+\bnu_2\right)
	\Gamma\left(s-w-\bnu_2\right) }
	\\
	\times
	\cM^*_{f, \phi_1, \phi_2} \left(s-w+\frac{1}{2}-\frac{k}{2}, w+\frac{1}{2}\right)
	\; dw \; du
	\\
	-
	\sum_{\ell=0}^{\frac{k}{2}}
	\overline{c_{\phi_1}(1)}
	\left\{
	\sum_j \tth^{(3)R}_\ell (t_j; s, \nu_1, \nu_2) (-1)^{\alpha_j} 
	\cL\left(s, \overline{u_j} \times \overline{\phi_1} \right) \left<u_j, U_{f, \phi_2}\right>
	\right.
	\\
	+
	\sum_\cuspa \frac{1}{4\pi} \int_{-\infty}^\infty \tth^{(3)R}_\ell (t; s, \nu_1, \nu_2) 
	\cL\left(s, E_\cuspa (*, 1/2-it) \times \overline{\phi_1}\right)
	\left<E_\cuspa (*, 1/2+it), U_{f, \phi_2}\right> \; dt
	\\
	+
	\tth^{(3)R}_\ell \left( (s-1+\bnu_1)/i; s, \nu_1, \nu_2\right) 
	\sum_{\cuspa} \left<E_{\cuspa}(*, 3/2-s-\bnu_1), U_{f, \phi_2} \right>
	\Res_{z=-s+1-\bnu_1} \cL\left(s, E_\cuspa(*, 1/2-z)\times \overline{\phi_1}\right) 
	\\
	\left.
	+
	\tth^{(3)R}_\ell \left( (s-1-\bnu_1)/i; s, \nu_1, \nu_2\right) 
	\sum_{\cuspa} \left<E_{\cuspa}(*, 3/2-s+\bnu_1), U_{f, \phi_2} \right>
	\Res_{z=-s+1+\bnu_1} \cL\left(s, E_\cuspa(*, 1/2-z)\times \overline{\phi_1}\right) 
	\right\}
	\\
	+
	\sum_{\ell=0}^{\frac{k}{2}-1} 
	\frac{(-1)^\ell}{\ell!}
	(4\pi)^{k-\ell-1} i^k \frac{4}{\pi} 
	\frac{\Gamma\left(s+\frac{1}{2}+\ell+\bar{\nu_1}\right) 
	\Gamma\left(s+\frac{1}{2}+\ell-\bar{\nu_1}\right)}
	{\Gamma\left(\frac{k}{2}-\ell -\frac{1}{2}+\bnu_2\right)
	\Gamma\left(\frac{k}{2}-\ell -\frac{1}{2}-\bnu_2\right) }
	\cM^*_{f, \phi_1, \phi_2} \left(-\ell, s+1+\ell-\frac{k}{2}\right)
	\\
	\times
	\frac{1}{2\pi i}  \int_{(\sigma_u = 1/2+2\epsilon)} 
	\frac{h(u/i) u 
	\Gamma\left(\frac{1}{2}+u\right) \Gamma\left(\frac{1}{2}-u\right)}
	{\Gamma\left(\frac{k}{2}+\bnu_1 +u\right) 
	\Gamma\left(\frac{k}{2}+\bnu_1-u\right)}
	\frac{\Gamma\left(\frac{k}{2}-\ell -s-\frac{1}{2}+u\right)}
	{\Gamma\left(\frac{3}{2}-\frac{k}{2}+\ell +s+u\right)}
	\; du
	.
\end{multline*}
After applying the spectral description of $\cM^*_{f, \phi_1, \phi_2}$, 
since 
\begin{multline*}
	\frac{1}{2\pi i}  \int_{(\sigma_u = 1/2+2\epsilon)} 
	\frac{h(u/i) u 
	\Gamma\left(\frac{1}{2}+u\right) \Gamma\left(\frac{1}{2}-u\right)}
	{\Gamma\left(\frac{k}{2}+\bnu_1 +u\right) 
	\Gamma\left(\frac{k}{2}+\bnu_1-u\right)}
	\frac{\Gamma\left(\frac{k}{2}-\ell -s-\frac{1}{2}+u\right)}
	{\Gamma\left(\frac{3}{2}-\frac{k}{2}+\ell +s+u\right)}
	\; du
	\\
	=
	(-1)^{\frac{k}{2}-\ell -1} 
	\frac{1}{2\pi i}  \int_{(\sigma_u = 1/2+2\epsilon)} 
	h(u/i) u
	\frac{\cos(\pi(s+u))}{\cos(\pi u)}
	\frac{\Gamma\left(\frac{k}{2}-\ell -s-\frac{1}{2}+u\right)
	\Gamma\left(\frac{k}{2}-\ell -\frac{1}{2} -s -u\right)}
	{\Gamma\left(\frac{k}{2}+\bnu_1 +u\right) 
	\Gamma\left(\frac{k}{2}+\bnu_1-u\right)}
	\; du
	\\
	=
	(-1)^{\frac{k}{2}-\ell -1} 
	\frac{\pi}{2}
	\frac{1}{\pi^2}  \int_{-\infty}^\infty 
	h(t) it
	\frac{\cos(\pi(s+it))}{\cos(\pi it)}
	\frac{\Gamma\left(\frac{k}{2}-\ell -s-\frac{1}{2}+it\right)
	\Gamma\left(\frac{k}{2}-\ell -\frac{1}{2} -s -it\right)}
	{\Gamma\left(\frac{k}{2}+\bnu_1 +it\right) 
	\Gamma\left(\frac{k}{2}+\bnu_1-it\right)}
	\; dt
	, 
\end{multline*}
we get
\begin{multline*}
	\tE^{(4)}_{f, \phi_1, \phi_2}(s)
	=
	\overline{c_{\phi_1}(1)}
	\left\{
	\sum_j 
	\tth^{(3)}(t_j; s, \nu_1, \nu_2) 
	(-1)^{\alpha_j} 
	\cL\left(s, \overline{u_j} \times \overline{\phi_1} \right) \left<u_j, U_{f, \phi_2}\right>
	\right.
	\\
	+
	\sum_\cuspa \frac{1}{4\pi} \int_{-\infty}^\infty 
	\tth^{(3)} (t; s, \nu_1, \nu_2) 
	\cL\left(s, E_\cuspa (*, 1/2-it) \times \overline{\phi_1}\right)
	\left<E_\cuspa (*, 1/2+it), U_{f, \phi_2}\right> \; dt
	\\
	+
	\tth^{(3)}\left( (s-1+\bnu_2)/i; s, \nu_1, \nu_2\right)
	\\
	\times
	\sum_{\cuspa} \left<E_{\cuspa}(*, 3/2-s-\bnu_1), U_{f, \phi_2} \right>
	\Res_{z=-s+1-\bnu_1} \cL\left(s, E_\cuspa(*, 1/2-z)\times \overline{\phi_1}\right) 
	\\
	+
	\tth^{(3)} \left( (s-1-\bnu_1)/i; s, \nu_1, \nu_2\right) 	
	\\
	\times
	\left.
	\sum_{\cuspa} \left<E_{\cuspa}(*, 3/2-s+\bnu_1), U_{f, \phi_2} \right>
	\Res_{z=-s+1+\bnu_1} \cL\left(s, E_\cuspa(*, 1/2-z)\times \overline{\phi_1}\right) 
	\right\}
	\\
	+
	2
	\sum_{\ell=1}^{\frac{k}{2}} 
	\frac{1}{(\ell-1)!}
	(4\pi)^{k-\ell}
	\frac{\Gamma\left(s-\frac{1}{2}+\ell+\bar{\nu_1}\right) 
	\Gamma\left(s-\frac{1}{2}+\ell-\bar{\nu_1}\right)}
	{\Gamma\left(\frac{k}{2}-\ell+\frac{1}{2}+\bnu_2\right)
	\Gamma\left(\frac{k}{2}-\ell +\frac{1}{2}-\bnu_2\right) }
	\cM^*_{f, \phi_1, \phi_2} \left(-\ell+1, s+\ell-\frac{k}{2}\right)
	\\
	\times 
	H_{2, \ell}(s; \bnu_1)
	.
\end{multline*}

By \eqref{e:tth1_Omega} and \eqref{e:tth3_Omega}, 
\begin{multline*}
	\tth^{(1)}\left((s-1\pm \bnu_1)/i; s, \nu_1, \nu_2\right)
	\\	
	=
	\frac{(4\pi)^k i^k \sqrt{\pi}}{2}
	\frac{\Gamma\left(2\mp 2\bnu_1-2s\right)
	\Gamma\left(2s-1\right)
	\Gamma\left(2s-1\pm 2\bnu_1\right)}
	{\Gamma\left(\frac{1}{2}+\frac{k}{2}+\bnu_1\right) 
	\Gamma\left(\frac{1}{2}+\frac{k}{2}-\bnu_1\right)
	\Gamma\left(-s+1+\frac{k}{2}+\bnu_2\mp \bnu_1\right)
	\Gamma\left(-s+1+\frac{k}{2}-\bnu_2\mp \bnu_1\right)}
	\\
	\times
	H_1^\mp (s;\bnu_1)
	\\
	+
	\tth^{(1)\pm}(s; \bnu_1, \bnu_2)
\end{multline*}
and
\begin{multline*}
	\tth^{(3)}\left( (s-1\pm \bnu_1)/i; s, \nu_1, \nu_2\right)
	=
	\tth^{(3)\pm}(s; \bnu_1, \bnu_2)
	\\
	+
	\frac{(4\pi)^{k}\sqrt{\pi}}{2}
	\frac{\Gamma\left(2s-1\right)
	\cos\left(\pi(2s\pm \bnu_1)\right)}
	{\sin\left(\pi\left(2s\pm 2\bnu_1\right)\right)}
	\frac{1} 
	{\Gamma\left(-s+1+\frac{k}{2}\mp \bnu_1+\bnu_2\right)
	\Gamma\left(-s+1+\frac{k}{2}\mp \bnu_1-\bnu_2\right) }
	H_1^\mp (s;\bnu_1)
	.
\end{multline*}

We prove the proposition.

\end{proof}


\subsection{Main term $M_{f, \phi_1, \phi_2}(s)$}

Recalling \eqref{e:M1} and \eqref{e:M3}, we get
\begin{multline*}
	M^{(1)}_{f, \phi_1, \phi_2}(s)
	+
	M^{(3)}_{f, \phi_1, \phi_2}(s)
	\\
	=
	\left(
	- \frac{\pi \overline{c_{\phi_1}(-1)} i^k}
	{\Gamma\left(\frac{1}{2}+\frac{k}{2}+\bnu_1\right) 
	\Gamma\left(\frac{1}{2}+\frac{k}{2}-\bnu_1\right)}
	+
	\overline{c_{\phi_1}(1)}
	\cos\left(\pi(2s+ \bnu_1)\right)
	\right)
	\\
	\times
	\frac{(4\pi)^k\sqrt{\pi}}{2}
	\frac{\Gamma\left(2s-1\right)}
	{\sin\left(\pi\left(2s+2\bnu_1\right)\right)}
	\frac{1}
	{\Gamma\left(-s+1+\frac{k}{2}+\bnu_2- \bnu_1\right)
	\Gamma\left(-s+1+\frac{k}{2}-\bnu_2- \bnu_1\right)}
	\\
	\times
	H_1^-(s;\bnu_1)
	\sum_{\cuspa} 	
	\left<E_\cuspa \left(*, 3/2-s-\bnu_1\right), U_{f, \phi_2}\right> 
	\Res_{z=-s+1 -\bnu_1} 
	\cL\left(s, E_{\cuspa} \left(*, 1/2-z\right) \times \overline{\phi_1}\right)
	\\
	+
	\left(
	- \frac{\pi \overline{c_{\phi_1}(-1)} i^k}
	{\Gamma\left(\frac{1}{2}+\frac{k}{2}+\bnu_1\right) 
	\Gamma\left(\frac{1}{2}+\frac{k}{2}-\bnu_1\right)}
	+
	\overline{c_{\phi_1}(1)}
	\cos\left(\pi(2s- \bnu_1)\right)
	\right)
	\\
	\times
	\frac{(4\pi)^k \sqrt{\pi}}{2}
	\frac{\Gamma\left(2s-1\right)}
	{\sin\left(\pi\left(2s-2\bnu_1\right)\right)}
	\frac{1}
	{\Gamma\left(-s+1+\frac{k}{2}+\bnu_2+\bnu_1\right)
	\Gamma\left(-s+1+\frac{k}{2}-\bnu_2+\bnu_1\right)}
	\\
	\times
	H_1^+ (s;\bnu_1)
	\sum_{\cuspa} \left<E_\cuspa \left(*, 3/2-s+\bnu_1\right), U_{f, \phi_2}\right> 
	\Res_{z=-s+1 +\bnu_1} 
	\cL\left(s, E_{\cuspa} \left(*, 1/2-z\right) \times \overline{\phi_1}\right)
	.
\end{multline*}
Here $H_1^\pm(s; \bnu_1)$ and $H_{2, \ell}(s; \bnu_1)$ are defined in \eqref{e:H1pm} and \eqref{e:H2} respectively. 

Recalling \eqref{e:tM}, we have the following lemma.
\begin{lem}\label{lem:Mfphi1phi2}
\begin{multline}\label{e:second_main}
	M_{f, \phi_1, \phi_2}(s)
	:=
	\tM_{f, \phi_1, \phi_2} (s)
	+
	M^{(1)}_{f, \phi_1, \phi_2}(s)
	+
	M^{(3)}_{f, \phi_1, \phi_2}(s)
	\\
	=	
	(4\pi)^{\frac{k}{2}-\bnu_1} 
	\overline{c^+_{\phi_1}}
	\frac{\Gamma\left(-\frac{k}{2}+\frac{1}{2}-\bnu_1\right) }
	{\Gamma\left(\frac{1}{2}+\bnu_1\right) \Gamma\left(\frac{1}{2}-\bnu_1\right)}
	\frac{\zeta(2s) }{\zeta(2s+2\bnu_1+1)}
	\overline{c_{\phi_2}(1)}
	L\left(s+1/2+ \bnu_1 , f\times \overline{\phi_2}\right)
	\frac{1}{2}H_1^+ (\bnu_1)
	\\
	+
	\left(
	- \frac{\pi \overline{c_{\phi_1}(-1)} i^k}
	{\Gamma\left(\frac{1}{2}+\frac{k}{2}+\bnu_1\right) 
	\Gamma\left(\frac{1}{2}+\frac{k}{2}-\bnu_1\right)}
	+
	\overline{c_{\phi_1}(1)}
	\cos\left(\pi(2s- \bnu_1)\right)
	\right)
	\\
	\times
	\frac{(4\pi)^k \sqrt{\pi}}{2}
	\frac{\Gamma\left(2s-1\right)}
	{\sin\left(\pi\left(2s-2\bnu_1\right)\right)}
	\frac{1}
	{\Gamma\left(-s+1+\frac{k}{2}+\bnu_2+\bnu_1\right)
	\Gamma\left(-s+1+\frac{k}{2}-\bnu_2+\bnu_1\right)}
	\\
	\times
	H_1^+ (s;\bnu_1)
	\sum_{\cuspa} \left<E_\cuspa \left(*, 3/2-s+\bnu_1\right), U_{f, \phi_2}\right> 
	\Res_{z=-s+1 +\bnu_1} 
	\cL\left(s, E_{\cuspa} \left(*, 1/2-z\right) \times \overline{\phi_1}\right)
	\\
	+
	(4\pi)^{\frac{k}{2}+\bnu_1}
	\overline{c^-_{\phi_1}} 
	\frac{\Gamma\left(-\frac{k}{2}+\frac{1}{2}+\bnu_1\right)}
	{\Gamma\left(\frac{1}{2}+\bnu_1\right) \Gamma\left(\frac{1}{2}-\bnu_1\right)}
	\frac{\zeta(2s) }{\zeta(2s-2\bnu_1+1)}
	 \overline{c_{\phi_2}(1)}
	L\left(s+1/2- \bnu_1 , f\times \overline{\phi_2}\right)
	\frac{1}{2}H_1^- (\bnu_1)
	\\
	+
	\left(
	- \frac{\pi \overline{c_{\phi_1}(-1)} i^k}
	{\Gamma\left(\frac{1}{2}+\frac{k}{2}+\bnu_1\right) 
	\Gamma\left(\frac{1}{2}+\frac{k}{2}-\bnu_1\right)}
	+
	\overline{c_{\phi_1}(1)}
	\cos\left(\pi(2s+ \bnu_1)\right)
	\right)
	\\
	\times
	\frac{(4\pi)^k\sqrt{\pi}}{2}
	\frac{\Gamma\left(2s-1\right)}
	{\sin\left(\pi\left(2s+2\bnu_1\right)\right)}
	\frac{1}
	{\Gamma\left(-s+1+\frac{k}{2}+\bnu_2- \bnu_1\right)
	\Gamma\left(-s+1+\frac{k}{2}-\bnu_2- \bnu_1\right)}
	\\
	\times
	H_1^-(s;\bnu_1)
	\sum_{\cuspa} 	
	\left<E_\cuspa \left(*, 3/2-s-\bnu_1\right), U_{f, \phi_2}\right> 
	\Res_{z=-s+1 -\bnu_1} 
	\cL\left(s, E_{\cuspa} \left(*, 1/2-z\right) \times \overline{\phi_1}\right)
	, 
\end{multline}
where $H_1^\pm(\bnu_1)$ is given in \eqref{e:H1}.

\end{lem}

\subsection{A proof of Theorem \ref{thm:second_intro}}
We will now consider the case where $N$ is square free and  $\phi_1(z) = E_N^{(k)}(z, 1/2+\overline{ir})$, which is a weight $k$ Eisenstein series at the cusp $1/N$, defined in \eqref{e:EkN}, 
We also take $\phi_2(z) = g(z)y^{k/2}$, where $g$ is a holomorphic newform of weight $k$ and level $N$.  Recalling \eqref{e:second}, taking $s=1/2-ir$, 
by \eqref{e:inner_ujU} and \eqref{e:inner_E1/aU}, we have
\begin{multline*}
	\sum_{j} 
	\frac{h(t_j) }{\cosh(\pi t_j)}
	\cL\left(1/2-ir, \overline{u_j}\times \bar{g}\right)
	\cL(1/2+ir, f\times u_j)
	\\
	+
	\sum_{a\mid N} 
	\frac{1}{4\pi} \int_{-\infty}^\infty 
	\frac{h(t)}{\cosh(\pi t)}
	\frac{4\left(\frac{N}{a}\right)^{-1}}
	{\zeta^*(1+2it)\zeta^*(1-2it) \prod_{p\mid N} (1-p^{-1-2it})\left(1-p^{-1+2it}\right)}
	\\
	\times
	L\left(1/2-ir, E_{1/a}(*, 1/2-it)\times \bar{g}\right)
	L\left(1/2+ir, f\times E_{1/a}(*, 1/2+it) \right)
	\; dt
	\\
	=
	\left((4\pi)^{-\frac{1}{2}}  
	(2\pi)^{-2ir}
	\prod_{p\mid N}(1-p^{-1-2ir})
	\right)^{-1}
	\left\{
	M_{f, E_N^{(k)}(*, 1/2+\overline{ir}) , gy^{k/2}}(1/2-ir) 
	\right.
	\\
	\left.
	+ 
	\Omega^{(1)}_{f, E_N^{(k)}(*, 1/2+\overline{ir}), gy^{k/2}}(1/2-ir)
	+
	\Omega^{(3)}_{f, E_N^{(k)}(*, 1/2+\overline{ir}), gy^{k/2}}(1/2-ir)
	\right.
	\\
	\left.
	+
	E^{(1)}_{f, E_N^{(k)}(*, 1/2+\overline{ir}), gy^{k/2}}(1/2-ir)
	+
	E^{(3)}_{f, E_N^{(k)}(*, 1/2+\overline{ir}), gy^{k/2}}(1/2-ir)
	\right\}
	.
\end{multline*}
Here
$M_{f, E_N^{(k)}(*, 1/2+\overline{ir}), gy^{k/2}}(s)$ is given in \eqref{e:second_main},
$\Omega^{(1)}_{f, E_N^{(k)}(*, 1/2+\overline{ir}), gy^{k/2}}(s)$ in \eqref{e:Omega1},
$E^{(1)}_{f, E_N^{(k)}(*, 1/2+\overline{ir}), gy^{k/2}}(s)$ in \eqref{e:E1},
$\Omega^{(3)}_{f, E_N^{(k)}(*, 1/2+\overline{ir}), gy^{k/2}}(s)$ in \eqref{e:Omega3},
and 
$E^{(3)}_{f, E_N^{(k)}(*, 1/2+\overline{ir}), gy^{k/2}}(s)$ in \eqref{e:E3}.

We first compute the inner product appearing in the main term $M_{f, E_N^{(k)} (*, 1/2+\overline{ir}), gy^{k/2}}(s)$. 
\begin{lem}\label{lem:inner-L}
Assume that $f$ and $g$ are new forms of weight $k$, for $\Gamma_0(N)$. 
For $a\mid N$, we have
\be\label{e:inner-L}
	\left<E_{1/a}(*, s), U_{f, gy^{k/2}}(z)\right> 
	=
	\frac{\Gamma(s+k-1)}{(4\pi)^{s+k-1} \zeta(2s)}
	\frac{N}{a} A(N/a) \overline{B(N/a)} L(s, f\times \bar{g}) 
	.
\ee
\end{lem}
\begin{proof}

For $\Gamma=\Gamma_0(N)$, let $\fund = \Gamma\bsl \bH$ be the fundamental domain for $\Gamma$. 
For $\cuspa = 1/a$ with $a\mid N$, compute the inner product
\begin{multline*}
	\left<E_{1/a}(*, s), U_{f, gy^{k/2}}(z)\right>
	=
	\iint_{\fund} \overline{U_{f, gy^{k/2}}(z)} 
	\sum_{\gamma\in \Gamma_{1/a}\bsl \Gamma} \left(\Im(\sigma_{1/a}^{-1}\gamma z)\right)^s\; \frac{dx\; dy}{y^2}
	\\
	=
	\sum_{\gamma\in \Gamma_{1/a} \bsl \Gamma} \int_{\sigma_{1/a}^{-1} \gamma \fund} \overline{U_{f, gy^{k/2}}(\sigma_{1/a} z)} y^s \; \frac{dx\; dy}{y^2}
	.
\end{multline*}
As in \cite{Iwa02}, as $\gamma$ runs over $\Gamma_{1/a}\bsl \Gamma$, the sets $\sigma_{1/a}^{-1}\gamma \fund$ cover the strip $\left\{z\in \bH\;|\; 0 < x< 1\right\}$ once (for an appropriate choice of representatives), giving
\be\label{e:after_covering}
	\int_0^\infty \int_0^1 \overline{U_{f, gy^{k/2}}(\sigma_{1/a}z)} y^s \; \frac{dx\; dy}{y^2} 
	.
\ee
Here
$$
	U_{f, gy^{k/2}}(\sigma_{1/a} z)
	= \left(\Im(\sigma_{1/a} z)\right)^k \overline{f(\sigma_{1/a} z)} g(\sigma_{1/a} z).
$$

Now, we will follow Asai's method \cite{Asa76}. 
For a matrix $\gamma = \sm a & b\\ c& d\esm\in GL_2(\R)$ of positive determinant, write
$$
	\left(f\mid_k \gamma\right)(z) = \det(\gamma)^{\frac{k}{2}} (cz+d)^{-k} f(\gamma z)
	.
$$
Then, for each $a\mid N$, we have
\be\label{e:V_a}
	\overline{U_{f, gy^{k/2}}(\sigma_{1/a} z)}
	=
	\left(f\mid_k \sigma_{1/a}\right)(z) 
	\overline{\left(g\mid_k \sigma_{1/a}\right)(z)}
	y^k
	.
\ee

Let 
$$
	M_a := \frac{N}{a}.
$$
Since $N$ is square-free, $(M_a, a) = 1$, so we can take $C_a, D_a\in \Z$ such that 
$$
	M_a D_a- a C_a = 1.
$$
Define
$$
	W_a := \bpm 1 & C_a\\ a & D_aM_a\ebpm \bpm M_a & \\ & 1\ebpm 
	=
	\bpm M_a & C_a\\ aM_a & D_aM_a\ebpm
$$
and this matrix $W_a$ satisfies conditions described in (2) in \cite{Asa76}. 
Since $f$ and $g$ are newforms of level $N$, $\left(f\mid_k W_a\right)$ and $\left(g\mid_k W_a\right)$ are again newforms of level $N$. 

By Theorem 2 in \cite{Asa76}, we have
$$
	f\mid_k W_a = \mu(M_a) \sqrt{M_a} A(M_a)\cdot f
$$
and 
$$
	g\mid_k W_a  = \mu(M_a) \sqrt{M_a} B(M_a) \cdot g
	.
$$

Since
$$
	W_a 
	= \bpm 1& C_a\\ a & D_a M_a \ebpm \bpm M_a & \\ & 1\ebpm 
	= \sigma_{1/a} 
	\bpm \frac{M_a}{\sqrt{\m_{1/a}}} & \frac{C_a}{\sqrt{\m_{1/a}}} \\ 0 & \sqrt{\m_{1/a}}\ebpm, 
$$
where $\sigma_{1/a}$ is defined in \eqref{e:sigma1/a} 
and $\m_{1/a} = M_a$, we have
$$
	\sigma_{1/a} = W_a \bpm \frac{1}{\sqrt{\m_{1/a}}} & - \frac{C_a}{\m_{1/a}\sqrt{\m_{1/a}}}\\ 0 & \frac{1}{\sqrt{\m_{1/a}}}\ebpm
	.
$$
Then we get
\begin{multline*}
	\left(f\mid_k \sigma_{1/a}\right)(z)
	=
	\mu(\m_{1/a}) \sqrt{\m_{1/a}} A(\m_{1/a}) 
	\left(f\mid_k \bpm \frac{1}{\sqrt{\m_{1/a}}} & - \frac{C_a}{\m_{1/a}\sqrt{\m_{1/a}}} \\ 0 & \frac{1}{\sqrt{\m_{1/a}}}\ebpm\right)(z) 
	\\
	=
	\mu(\m_{1/a}) \sqrt{\m_{1/a}} A(\m_{1/a}) 
	f\left( z- \frac{C_a}{\m_{1/a}}\right)
\end{multline*}
and
$$
	\left(g\mid_k \sigma_{1/a}\right)(z)
	=
	\mu(\m_{1/a}) \sqrt{\m_{1/a}} B(\m_{1/a}) 
	g\left( z- \frac{C_a}{\m_{1/a}}\right)
	.
$$
Substituting into \eqref{e:V_a}, 
$$
	\overline{U_{f, gy^{k/2}}(\sigma_{1/a} z)}
	=
	\m_{1/a} A(\m_{1/a}) \overline{B(\m_{1/a}) }
	f\left( z- \frac{C_a}{\m_{1/a}}\right)
	\overline{g\left( z- \frac{C_a}{\m_{1/a}}\right)}
	y^k
	, 
$$
and then \eqref{e:after_covering}, we get
\begin{multline*}
	\left<E_{1/a}(*, s), U_{f, gy^{k/2}}(z)\right> 
	\\
	=
	\m_{1/a} A(\m_{1/a}) \overline{B(\m_{1/a}) }
	\int_0^\infty 
	\int_0^1
	f\left( z- \frac{C_a}{\m_{1/a}}\right)
	\overline{g\left( z- \frac{C_a}{\m_{1/a}}\right)}
	\; dx\; 
	y^{k+s-1}
	\; \frac{dy}{y} 
	\\
	=
	\m_{1/a} A(\m_{1/a}) \overline{B(\m_{1/a}) }
	\int_0^\infty 
	\sum_{n\geq 1} a(n) \overline{b(n)} 
	e^{-4\pi ny} 
	y^{k+s-1}
	\; \frac{dy}{y} 
	\\
	=
	\frac{\Gamma(s+k-1)}{(4\pi)^{s+k-1} \zeta(2s)}
	\m_{1/a} A(\m_{1/a}) \overline{B(\m_{1/a})} L(s, f\times \bar{g}) 
	.
\end{multline*}
\end{proof}

Define
\be\label{e:H1_r=0}
	H_1
	:=
	\frac{1}{\pi^2} \int_{-\infty}^\infty h(t) t \tanh(\pi t) 
	\; dt
	, 
\ee
\be\label{e:H2_r=0}
	H_2 
	:=
	\frac{1}{\pi^2} \int_{-\infty}^\infty h(t) t \tanh(\pi t) 
	\left(\psi\left(\frac{k}{2}+it\right) + \psi\left(\frac{k}{2}-it\right)\right)
	\; dt
	, 
\ee
\be\label{e:H3}
	H_3:=
	\frac{1}{\pi^2} \int_{-\infty}^\infty h(t) t \tanh(\pi t) 
	\left(\psi\left(\frac{k}{2}+it\right) + \psi\left(\frac{k}{2}-it\right)\right)^2
	\; dt
	, 
\ee
\be\label{e:H4}
	H_4:=
	\frac{1}{\pi^2} \int_{-\infty}^\infty h(t) t \tanh(\pi t) 
	\left(\psi\left(\frac{k}{2}+it\right)+\psi\left(\frac{k}{2}-it\right)
	\right)^3
	\; dt, 
\ee
\begin{multline}\label{e:H10}
	H_{1, 0}
	:=
	\frac{1}{\pi^2} \int_{-\infty}^\infty h(t) t \tanh(\pi t) 
	\\
	\times
	\left(\psi\left(\frac{k}{2}+it\right)+\psi\left(\frac{k}{2}-it\right)\right)
	\left(\psi'\left(\frac{k}{2}+it\right)+\psi'\left(\frac{k}{2}-it\right)\right)
	\; dt, 
\end{multline}
\be\label{e:H0}
	H_0
	:=
	\frac{1}{\pi^2} \int_{-\infty}^\infty h(t) t \tanh(\pi t) 
	\left(\psi'\left(\frac{k}{2}+it\right) + \psi'\left(\frac{k}{2}-it\right) \right)
	\; dt
\ee
and 
\be\label{e:H01}
	H_{0, 1}
	:=
	\frac{1}{\pi^2} \int_{-\infty}^\infty h(t) t \tanh(\pi t) 
	\left(\psi''\left(\frac{k}{2}+it\right) + \psi''\left(\frac{k}{2}-it\right)\right)
	\; dt
	.
\ee

For each $a\mid N$, define $P_a(s; it, ir_1)$ as 
\begin{multline}\label{e:Pa}
	\cL\left(s, E_{1/a}(*, 1/2-it) \times E_N^{(k)}(*, 1/2+ir_1)\right)
	\\
	=:
	\frac{P_a(s; it, ir_1)}{\zeta^*(1-2it)}
	\zeta(s+it+ir_1) \zeta(s+it-ir_1) \zeta(s-it+ir_1) \zeta(s-it-ir_1).
\end{multline}
Then 
\begin{multline}\label{e:Res_cL_Eisenstein}
	\Res_{z=-s+1\mp ir_1}
	\cL\left(s, E_{1/a}(*, 1/2-z) \times E_N^{(k)}(*,1/2+ir_1)\right)
	\\
	=
	\frac{\zeta(1\mp 2ir_1)
	\zeta(2s-1)}
	{\Gamma\left(s-\frac{1}{2}\pm ir_1\right)
	\pi^{-s+\frac{1}{2}\mp ir_1}}
	P_a (s; -s+1\mp ir_1, ir_1)
	.
\end{multline}

\begin{lem}\label{lem:Mfg}
Let $\phi_1(z) = E_N^{(k)}(z, 1/2+\overline{ir})$ and $\phi_2(z) = g(z) y^{k/2}$. 

For $f\neq g$, we have
\begin{multline}\label{e:Nrne}
	\left((4\pi)^{-\frac{1}{2}}  
	\prod_{p\mid N}(1-p^{-1})
	\right)^{-1}
	M_{f, E_N^{(k)}(*, 1/2), g y^{k/2} }(1/2)
	=
	\frac{2}{\zeta(2)} L(1, f\times \bar{g})
	H_3
	\\
	+ 
	C_{f, g}^{(2)} H_2
	+ C^{(1)}_{f, g} H_1
	+ C^{(0)}_{f, g} H_0
	, 
\end{multline}
for some constant $C^{(2)}_{f, g}$, $C^{(1)}_{f, g}$ and $C^{(0)}_{f, g}$. 

For $f=g$, for any $r\in \R$, we have
\begin{multline}\label{e:1re}
	\left((4\pi)^{-\frac{1}{2}}  
	(2\pi)^{-2ir}
	\prod_{p\mid N}(1-p^{-1-2ir})
	\right)^{-1}
	M_{f, E_N^{(k)}(*, 1/2+\overline{ir}), f y^{k/2} }(1/2-ir)
	\\
	=	
	2
	\frac{\zeta(1-2ir)\zeta(1+2ir) }{\zeta(2)}
	\left.\frac{d}{ds}\left((s-1) L(s, f\times \bar{f})\right) \right|_{s=1}
	H_1
	\\
	+
	\left\{ 2\frac{\zeta'(1-2ir)}{\zeta(1-2ir)}
	+2\frac{\zeta'(1+2ir)}{\zeta(1+2ir)}
	- 4\frac{\zeta'(2)}{\zeta(2)}
	- 4\log(2\pi) 
	\right.
	\\
	\left.
	- \frac{\sum_{a\mid N} \frac{N}{a} A(N/a) \overline{B(N/a)} 
	\left. \frac{d}{ds} P_a(s; -s+1-ir, ir)\right|_{s=1/2-ir}}
	{\sum_{a\mid N} \frac{N}{a} A(N/a) \overline{B(N/a)} P_a(1/2-ir; 1/2, ir) }
	\right\}
	\\
	\times
	\frac{\zeta(1-2ir)\zeta(1+2ir) }{\zeta(2)}
	H_1
	\Res_{s=1}L(s, f\times \bar{f})
	\\
	+
	\frac{\zeta(1-2ir)\zeta(1+2ir) }{\zeta(2)}
	\left.\frac{d}{ds}H_1^-(s; ir)\right|_{s=1/2-ir}
	\Res_{s=1}L(s, f\times \bar{f})
	\\
	+
	2^{-4ir}\pi^{-4ir}
	\frac{\zeta(1+2ir) \zeta(1+2ir)}
	{\zeta(2)}
	\\
	\times
	\frac{\left(\sum_{a\mid N} \frac{N}{a} A(N/a) \overline{B(N/a)} P_a(1/2-ir; 1/2+2ir, ir)\right)}
	{2\prod_{p\mid N}(1-p^{-1-2ir})}
	L(1+2ir, f\times \bar{f})
	H_1^+ (3ir)
	\\
	+
	2^{4ir} \pi^{4ir}
	\frac{\varphi(N)}{N^{1+2ir} \prod_{p\mid N}(1-p^{-1-2ir})} 
	\frac{\zeta(1-2ir)\zeta(1-2ir)  }{\zeta(2-4ir)}
	L\left(1- 2ir , f\times \bar{f}\right)
	H_1^- (ir)
	.
\end{multline}
Here $H_1^\pm$ is given in \eqref{e:H1}. 
For $f=g$, taking $r=0$, we have
\begin{multline}\label{e:1re0}
	\left((4\pi)^{-\frac{1}{2}}  
	\prod_{p\mid N}(1-p^{-1})
	\right)^{-1}
	M_{f, E_N^{(k)}(*, 1/2), f y^{k/2} }(1/2)
	=
	\frac{2}{\zeta(2)} 
	\Res_{s=1} L(s, f\times\bar{f})
	H_4
	\\
	+
	C_{f, f}^{(3)} H_3 + C_{f, f}^{(2)} H_2 + C_{f, f}^{(1)} H_1 + C_{f, f}^{(1, 1)} H_{1, 1} + C_{f, f}^{(0)} H_0 + C_{f, f}^{(0, 1)} H_{0, 1}
	, 
\end{multline}
for some constants $C_{f, f}^{(3)}$, $C_{f, f}^{(2)}$, $C_{f, f}^{(1)}$, $C_{f, f}^{(1, 1)}$, 
$C_{f, f}^{(0)}$ and $C_{f, f}^{(0, 1)}$.

\end{lem}

\begin{proof}
Recalling \eqref{e:second_main}, 
taking $\phi_1(z) = E_N^{(k)}(z, 1/2+\overline{ir_1})$ and $\phi_2(z) = g(z)y^{k/2}$, 
\begin{multline*}
	M_{f, E_N^{(k)}(*, 1/2+\overline{ir_1}), g y^{k/2} }(s)
	\\
	=	
	\frac{(4\pi)^{-ir_1} \pi^{-\frac{1}{2}-ir_1} }{2}
	\prod_{p\mid N} (1-p^{-1-2ir_1})
	\frac{\zeta(2s)\zeta(1+2ir_1) }{\zeta(2s+2ir_1+1)}
	L\left(s+1/2+ ir_1 , f\times g\right)
	H_1^+ (ir_1)
	\\
	+
	\left( - \cos(\pi ir_1) + \cos\left(\pi(2s- ir_1)\right) \right)
	\\
	\times
	\frac{(4\pi)^k \sqrt{\pi}}{2}
	\frac{\Gamma\left(2s-1\right)}
	{\sin\left(\pi\left(2s-2ir_1\right)\right)}
	\frac{1}
	{\Gamma\left(-s+\frac{1}{2}+k+ir_1\right)
	\Gamma\left(-s+\frac{3}{2}+ir_1\right)}
	\\
	\times
	H_1^+ (s;ir_1)
	\sum_{\cuspa} \left<E_\cuspa \left(*, 3/2-s+ir_1\right), U_{f, gy^{k/2}}\right> 
	\Res_{z=-s+1 +ir_1} 
	\cL\left(s, E_{\cuspa} \left(*, 1/2-z\right) \times E_N^{(k)}(*, 1/2+ir_1)\right)
	\\
	+
	\frac{(4\pi)^{ir_1}\pi^{-\frac{1}{2}+ir_1}}{2}
	\frac{\varphi(N)}{N^{1+2ir_1}} 
	\frac{\zeta(2s)\zeta(1-2ir_1)  }{\zeta(2s-2ir_1+1)}
	L\left(s+1/2- ir_1 , f\times g\right)
	H_1^- (ir_1)
	\\
	+
	\left(
	- \cos(\pi ir_1)
	+
	\cos\left(\pi(2s+ir_1)\right)
	\right)
	\\
	\times
	\frac{(4\pi)^k\sqrt{\pi}}{2}
	\frac{\Gamma\left(2s-1\right)}
	{\sin\left(\pi\left(2s+2ir_1\right)\right)}
	\frac{1}
	{\Gamma\left(-s+\frac{1}{2}+k-ir_1\right)
	\Gamma\left(-s+\frac{3}{2}-ir_1\right)}
	\\
	\times
	H_1^-(s;ir_1)
	\sum_{\cuspa} 	
	\left<E_\cuspa \left(*, 3/2-s-ir_1 \right), U_{f, gy^{k/2}}\right> 
	\Res_{z=-s+1 -ir_1} 
	\cL\left(s, E_{\cuspa} \left(*, 1/2-z\right) \times E_N^{(k)}(*, 1/2+ir_1)\right)
	.
\end{multline*}
Here $H_1^\pm(\bnu_1)$ is given in \eqref{e:H1} and  
$H_1^\pm(s; \bnu_1)$ is defined in \eqref{e:H1pm}. 

For a square-free $N\geq 1$, by Lemma \ref{lem:inner-L} and \eqref{e:Res_cL_Eisenstein}, 
\begin{multline}\label{e:sum_1/a_L_g}
	\left(
	\sum_{a\mid N}
	\left<E_{1/a} \left(*, 3/2-s\mp ir_1\right), U_{f, gy^{k/2}}\right> 
	\Res_{z=-s+1\mp ir_1} 
	\cL\left(s, E_{1/a} \left(*, 1/2-z\right) \times E_N^{(k)}(*, 1/2+ir_1) \right) 
	\right)
	\\
	=
	\frac{\zeta(1\mp 2ir_1)
	\zeta(2s-1)}
	{\Gamma\left(s-\frac{1}{2}\pm ir_1\right)}
	\frac{\Gamma\left(\frac{1}{2}-s\mp ir_1+k\right)
	\left(\sum_{a\mid N}\frac{N}{a} A(N/a) \overline{B(N/a)} 
	P_a (s; -s+1\mp ir_1, ir_1)\right)}
	{(4\pi)^k
	(2\pi)^{1-2s\mp 2ir_1} \zeta(3-2s\mp 2ir_1 )}
	\\
	\times
	L(3/2-s\mp ir_1, f\times \bar{g}) 
	.
\end{multline}
Applying \eqref{e:EkN_Fourier} and \eqref{e:sum_1/a_L_g}, we have
\begin{multline}\label{e:MfENg}
	M_{f, E_N^{(k)}(*, 1/2+\overline{ir_1}), g y^{k/2} }(s)
	\\
	=	
	\frac{2^{-2ir_1} \pi^{-\frac{1}{2}-2ir_1}}{2}
	\prod_{p\mid N} (1-p^{-1-2ir_1})
	\frac{\zeta(2s)\zeta(1+2ir_1) }{\zeta(2s+2ir_1+1)}
	L\left(s+1/2+ ir_1 , f\times \bar{g}\right)
	H_1^+ (ir_1)
	\\
	+
	\frac{2^{2ir_1} \pi^{-\frac{1}{2}+2ir_1} (2\pi)^{-2+4s}}{8}
	\frac{
	\left(\cos(\pi ir_1)- \cos\left(\pi(2s+ir_1)\right)\right)}
	{\sin(\pi s) \sin\left(\pi\left(s+ir_1\right)\right)}
	\frac{\zeta(1-2ir_1)\zeta(2-2s)}{\zeta(3-2s-2ir_1)}
	\\
	\times
	\left(\sum_{a\mid N} \frac{N}{a} A(N/a) \overline{B(N/a)} 
	P_a(s; -s+1-ir_1, ir_1)\right)
	L(3/2-s-ir_1, f\times\bar{g})
	H_1^-(s;ir_1)
	\\
	+
	\frac{2^{2ir_1}\pi^{-\frac{1}{2}+2ir_1}}{2}
	\frac{\varphi(N)}{N^{1+2ir_1}} 
	\frac{\zeta(2s)\zeta(1-2ir_1)  }{\zeta(2s-2ir_1+1)}
	L\left(s+1/2- ir_1 , f\times \bar{g}\right)
	H_1^- (ir_1)
	\\
	+
	\frac{2^{-2ir_1}\pi^{-\frac{1}{2}-2ir_1}
	(2\pi)^{-2+4s}}{8}
	\frac{
	\left(\cos(\pi ir_1) - \cos\left(\pi(2s- ir_1)\right) \right)}
	{\sin(\pi s)\sin\left(\pi\left(s-ir_1\right)\right)}
	\frac{\zeta(1+2ir_1)\zeta(2-2s)}
	{\zeta(3-2s+2ir_1)}
	\\
	\times
	\left(\sum_{a\mid N} \frac{N}{a} A(N/a) \overline{B(N/a)} P_a(s; -s+1+ir_1, ir_1)\right)
	L(3/2-s+ir_1, f\times \overline{g})
	H_1^+ (s;ir_1)
	.
\end{multline}

Taking $s=1/2-ir$ and $r_1=r$, we have
$$
	H_1^+(1/2-ir; ir)
	=
	H_1^+(3ir)
	\;
	\text{ and } 
	\;
	H_1^-(1/2-ir; ir) = H_1^+(ir) = H_1. 
$$

If $f\neq g$, taking $s=1/2-ir$ and $r_1=r$,  we have
%
%
%
%
\begin{multline*}
	M_{f, E_N^{(k)}(*, 1/2+\overline{ir}), g y^{k/2} }(1/2-ir)
	\\
	=	
	2^{-2ir} \pi^{-\frac{1}{2}-2ir}
	\prod_{p\mid N} (1-p^{-1-2ir})
	\frac{\zeta(1-2ir)\zeta(1+2ir) }{\zeta(2)}
	L\left(1, f\times \bar{g}\right)
	H_1
	\\
	+
	\frac{2^{-6ir}\pi^{-\frac{1}{2}-6ir}}{4}
	\frac{\zeta(1+2ir) \zeta(1+2ir)}
	{\zeta(2)}
	\\
	\times
	\left(\sum_{a\mid N} \frac{N}{a} A(N/a) \overline{B(N/a)} P_a(1/2-ir; 1/2+2ir, ir)\right)
	L(1+2ir, f\times \overline{g})
	H_1^+ (3ir)
	\\
	+
	\frac{2^{2ir} \pi^{-\frac{1}{2}+2ir}}{2}
	\frac{\varphi(N)}{N^{1+2ir}} 
	\frac{\zeta(1-2ir)\zeta(1-2ir)  }{\zeta(2-4ir)}
	L\left(1- 2ir , f\times \bar{g}\right)
	H_1^- (ir)
	.
\end{multline*}

We are now going to carefully evaluate at $r=0$.
By \eqref{e:H1}, we have
$$
	H_1^+(0) = H_1^-(0) = H_1, 
$$
$$
	\left.\frac{d}{dr} H_1^-(ir)\right|_{r=0}
	=
	-2i H_2
	=
	-\left.\frac{d}{dr} H_1^+(3ir)\right|_{r=0}
	, 
$$
$$
	\left.\frac{d^2}{dr^2}H_1^- (ir)\right|_{r=0}
	=
	- 4H_3
	\;  \text{ and }\; 
	\left.\frac{d^2}{dr^2} H_1^{+}(3ir)\right|_{r=0}
	=
	-8 H_0- 4 H_3
	.
$$
Then taking $r=0$, for $f\neq g$, we have \eqref{e:Nrne}.

If $f=g$, taking $s=1/2-ir$ and $r_1=r$ in \eqref{e:MfENg}, we have \eqref{e:1re}. 
Taking $r=0$, since
$$
	\left.\frac{d^3}{dr^3}H_1^- (ir)\right|_{r=0}
	=
	2i H_{0, 1} + 8i H_4
$$
and
$$
	\left.\frac{d^3}{dr^3} H_1^+ (3ir)\right|_{r=0}
	=
	-26 i H_{0, 1} -48 i H_{1, 1} - 8i H_4
	, 
$$
we get \eqref{e:1re0}.

It now remains for us to find upper bounds for 
$\Omega^{(1)}_{f, E_N^{(k)}(*, 1/2+\overline{ir}), gy^{k/2}}(s)$, 
$E^{(1)}_{f, E_N^{(k)}(*, 1/2+\overline{ir}), gy^{k/2}}(s)$,
$\Omega^{(3)}_{f, E_N^{(k)}(*, 1/2+\overline{ir}), gy^{k/2}}(s)$
and 
$E^{(3)}_{f, E_N^{(k)}(*, 1/2+\overline{ir}), gy^{k/2}}(s)$
given in section \ref{s:upper}.   This will be done in the next section, which will complete the proof of Theorem \ref{thm:second_intro}.

\end{proof}

\section{Upper bounds}\label{s:upper}
Recall that the error terms in Theorem~\ref{thm:second} were labeled
$\Omega^{(1)}_{f, \phi_1, \phi_2}(s)$ in \eqref{e:Omega1},
$E^{(1)}_{f, \phi_1, \phi_2}(s)$ in \eqref{e:E1},
$\Omega^{(3)}_{f, \phi_1, \phi_2}(s)$ in \eqref{e:Omega3},
$E^{(2)}_{f, \phi_1, \phi_2}(s)$ in \eqref{e:E2},
and
$E^{(3)}_{f, \phi_1, \phi_2}(s)$ in \eqref{e:E3}.
Our goal is to estimate these in the specific case $\phi_1 =  E^{(k)}_N(*, 1/2+\overline{ir_1})$ and $\phi_2 = g(z)y^{k/2}$.   
Fortunately, the term $E^{(2)}_{f, \phi_1, \phi_2}(s)$ does not occur as $\phi_2$ is a cusp form.    We will prove the following
\begin{prop}\label{prop:bounds}
Fix 
$$
h(t) =  \left(e^{-\left(\frac{t-T}{T^{\alpha}}\right)^2}+e^{-\left(\frac{t+T}{T^{\alpha}}\right)^2}\right) 
	\frac{t^2+1/4}{t^2+R},
$$
with $T \gg 1$, $\alpha \ge 1/3$ and $1\ll R < T^2$. 
In the case $s= 1/2 +ir, r_1 = -r\in \R$, with $|r| \le T^{2/3}$, with $\phi_1,\phi_2$ as above
$$
E^{(1)}_{f, \phi_1, \phi_2}(s)+E^{(3)}_{f, \phi_1, \phi_2}(s)+\Omega^{(1)}_{f, \phi_1, \phi_2}(s)+
\Omega^{(3)}_{f, \phi_1, \phi_2}(s) \ll_N T^{3/2-\alpha/2 + \epsilon} + T^\alpha|r|^{3/2 + \epsilon}.
$$
\end{prop}
The upper bounds for $\Omega^{(1)}_{f, \phi_1, \phi_2}(s)$ and  $\Omega^{(3)}_{f, \phi_1, \phi_2}(s)$
are easier to obtain and better than the corresponding bounds for $E^{(1)}_{f, \phi_1, \phi_2}(s)$ and $E^{(3)}_{f, \phi_1, \phi_2}(s)$.   Therefore, in the following sections, we will only work out in detail the estimates for $E^{(1)}_{f, \phi_1, \phi_2}(s)$ and $E^{(3)}_{f, \phi_1, \phi_2}(s)$.

\subsection{The estimation of $E^{(3)}_{f, E^{(k)}_N(*, 1/2+\overline{ir_1}), \phi_2}(s)$}
The goal of this section is to find an upper bound for 

\begin{multline*}
	E^{(3)}_{f, E^{(k)}_N(*, 1/2+\overline{ir_1}), \phi_2}(s)
	\\
	=
	\sum_j \tth^{(3)} \left(t_j; s, \overline{ir_1}, \nu_2\right) 
	\left<u_j, U_{f, \phi_2}\right>
	(-1)^{\alpha_j} \cL\left(s, E^{(k)}_N(*, 1/2+ir_1)\times \overline{u_j}\right)
	\\
	+
	\sum_{\cuspa} \frac{1}{4\pi} \int_{-\infty}^\infty 
	\tth^{(3)}\left(t; s, \overline{ir_1}, \nu_2\right) 
	\left<E_\cuspa\left(*, 1/2+it\right), U_{f, \phi_2}\right>
	\cL\left(s, E^{(k)}_N(*, 1/2+ir_1)\times E_{\cuspa}\left(*, 1/2-it\right)\right) 
	\; dt
	, 
\end{multline*}
where
\begin{multline*}
	\tth^{(3)}(t; s, \overline{ir_1}, \nu_2)
	=
	\frac{(4\pi)^{k} i^k}{\pi\sqrt{\pi}}
	\frac{1}{2\pi i} \int_{(\sigma_u = 1/2+2\epsilon)} 
	\frac{h(u/i) u 
	\Gamma\left(\frac{1}{2}+u\right) \Gamma\left(\frac{1}{2}-u\right)}
	{\Gamma\left(\frac{k}{2}+ir_1 +u\right) \Gamma\left(\frac{k}{2}+ir_1-u\right)}
	\\
	\times
	\frac{1}{2\pi i } \int_{(\sigma_w = 1/2+\epsilon)}
	\frac{\Gamma\left(-w+u\right)}
	{\Gamma\left(1+w+u\right)}
	\Gamma\left(w+\frac{k}{2}+ir_1\right) 
	\Gamma\left(w+\frac{k}{2}-ir_1 \right)
	\\
	\times
	\frac{\Gamma\left(s-w+\frac{1}{2}-\frac{k}{2}\right)
	\Gamma\left(-s+w+\frac{1}{2}+\frac{k}{2}\right)
	} 
	{\Gamma\left(s-w+\bnu_2\right)
	\Gamma\left(s-w-\bnu_2\right) }
	\frac{\Gamma\left(s-w-\frac{k}{2}-it\right) \Gamma\left(s-w-\frac{k}{2}+it\right)}
	{\Gamma\left(\frac{1}{2}-it\right) \Gamma\left(\frac{1}{2}+it\right)}
	\; dw \; du
	\\
	+
	\sum_{\ell=0}^{\frac{k}{2}}
	\frac{1}{\ell!}
	\frac{(4\pi)^{k} i^k}{\pi\sqrt{\pi}}
	\sum_\pm
	\frac{
	\Gamma\left(s+\ell \pm it+ir_1\right) 
	\Gamma\left(s+\ell \pm it-ir_1\right)
	\Gamma\left(-\ell \mp 2it\right)} 
	{\Gamma\left(\frac{k}{2}-\ell \mp it+\bnu_2\right)
	\Gamma\left(\frac{k}{2}-\ell \mp it-\bnu_2\right) }
	\\
	\times
	\frac{1}{2\pi i} \int_{(\sigma_u = 1/2+2\epsilon)} 
	\frac{h(u/i) u 
	\Gamma\left(\frac{1}{2}+u\right) \Gamma\left(\frac{1}{2}-u\right)}
	{\Gamma\left(\frac{k}{2}+ir_1 +u\right) \Gamma\left(\frac{k}{2}+ir_1-u\right)}
	\frac{\Gamma\left(-s+\frac{k}{2}-\ell \mp it+u\right)}
	{\Gamma\left(1+s-\frac{k}{2}+\ell \pm it+u\right)}
	\; du
	.
\end{multline*}
%
%
%
These were defined in \eqref{e:E3} and \eqref{e:tth3} and are repeated here for convenience.
Let 
\begin{multline}\label{e:h3Rell}
	\tth^{(3)R}_\ell (t; s, \overline{ir_1}, \nu_2)
	:=
	\frac{1}{\ell!}
	\frac{(4\pi)^{k} i^k}{\pi\sqrt{\pi}}
	\sum_\pm
	\frac{
	\Gamma\left(s+\ell \pm it+ir_1\right) 
	\Gamma\left(s+\ell \pm it-ir_1\right)
	\Gamma\left(-\ell \mp 2it\right)} 
	{\Gamma\left(\frac{k}{2}-\ell \mp it+\bnu_2\right)
	\Gamma\left(\frac{k}{2}-\ell \mp it-\bnu_2\right) }
	\\
	\times
	\frac{1}{2\pi i} \int_{(\sigma_u = 1/2+2\epsilon)} 
	\frac{h(u/i) u 
	\Gamma\left(\frac{1}{2}+u\right) \Gamma\left(\frac{1}{2}-u\right)}
	{\Gamma\left(\frac{k}{2}+ir_1 +u\right) \Gamma\left(\frac{k}{2}+ir_1-u\right)}
	\frac{\Gamma\left(-s+\frac{k}{2}-\ell \mp it+u\right)}
	{\Gamma\left(1+s-\frac{k}{2}+\ell \pm it+u\right)}
	\; du.
\end{multline}
We consider the integral part of $\tth^{(3)}$, i.e., 
$$
	\tth^{(3)}(t; s, \overline{ir_1}, \nu_2) - \sum_{\ell=1}^{\frac{k}{2}} \tth^{(3)R}_\ell(t; s, \overline{ir_1}, \nu_2)
$$
in this section. 
The remainder is going to be studied in Section \ref{ss:residues}.

We will work through explicitly the case of most interest to us:  $s= 1/2 +ir$ and $\nu_2= (k-1)/2$.    
Note that  a cancellation of gamma factors occurs in the $w$ integral.

We will assume that for fixed $T \gg 1$, $\alpha \ge 1/3$,
$$
h(t) =  \left(e^{-\left(\frac{t-T}{T^{\alpha}}\right)^2}+e^{-\left(\frac{t+T}{T^{\alpha}}\right)^2}\right) \frac{t^2+1/4}{t^2+R}
$$
Write $u = \sigma_u+i\gamma$ and $w = \sigma_w+i\gamma_2$. 
Bringing the sum over $j$ inside, consider the inner integral
\begin{multline}\label{II}
I(u)=\frac{1}{2\pi i } \int_{(\sigma_w = 1+\epsilon)}
	\frac{\Gamma\left(-w+u\right)\Gamma\left(w+\frac{k}{2}+ir_1\right) 
	\Gamma\left(w+\frac{k}{2}-ir_1\right)
	\Gamma\left(-ir+w+\frac{k}{2}\right)}
	{\Gamma\left(1+w+u\right)\Gamma\left(ir-w+\frac{k}{2}\right) }
	\\
	\times
	\sum_j 
	\left<u_j, U_{f, \phi_2}\right> 
	(-1)^{\alpha_j} 
	\cL\left(1/2+ir, \overline{u_j}\times E^{(k)}_N(*, 1/2+ir_1)\right) \\ \times \frac{\Gamma\left(\frac{1}{2}+ir-w-\frac{k}{2}-it_j\right) 
	\Gamma\left(\frac{1}{2}+ir-w-\frac{k}{2}+it_j\right)}
	{\Gamma\left(\frac{1}{2}-it_j\right) \Gamma\left(\frac{1}{2}+it_j\right)}	\; dw.
\end{multline}
The $L$-series above factors as follows:
$$
	(-1)^{\alpha_j} \cL\left(1/2+ir, \overline{u_j}\times E^{(k)}_N(*, 1/2+ir_1)\right)=  \overline{\rho_j(-d)}L_N(1/2 + i(r+r_1), \overline{u_j})L_N(1/2 +i(r-r_1), \overline{u_j}).
$$
Here $d|N$  and equals 1 if $u_j$  is a new form.  The $N$ subscripts indicate that Euler factors differ from the generic form at primes dividing $N$.  In the simplest case of level 1 and $r_1=r =0$, it is exactly true that 
$$
	(-1)^{\alpha_j} \cL \left(\frac{1}{2}, \overline{u_j}\times E_1 (*, 1/2) \right) =  \overline{\rho_j(-1)}L_1(1/2, \overline{u_j})L_1(1/2, \overline{u_j}).
$$

Our particular case of interest, other than $r = r_1 =0$, will be $r = -r_1$.   In this case 
$$
	(-1)^{\alpha_j} \cL\left(1/2+ir, \overline{u_j}\times E^{(k)}_N(*, 1/2-ir)\right)=  \overline{\rho_j(-d)}L_N(1/2, \overline{u_j})L_N(1/2 +2ri, \overline{u_j}).
$$
In \cite{JM05} it is shown in the case $N=1$, after taking $K=T, t=2r$ in Theorem 2, line  (1.30) and translating into our notation, that 
$$
	\sum_{|t_j|\sim T}\frac{(-1)^{\alpha_j} \cL \left(\frac{1}{2}+ir, \overline{u_j}\times E_1 (*, 1/2-ir) \right)^2}{\cosh(\pi t_j)} \ll (T^2+|r|^{4/3})^{1+\epsilon},
$$
for any $T,r\ge 1$.  In the case $r=0$ it is easy to generalize their results to general level $N$ (actually Motohashi's original proof in \cite{Mot92} suffices), and show that the Lindel\"of Hypothesis holds in the $N$ aspect as well:
$$
	\sum_{|t_j|\sim T}\frac{(-1)^{\alpha_j} \cL \left(\frac{1}{2}, \overline{u_j}\times E_N (*, 1/2) \right)^2}{\cosh(\pi t_j)}\ll N T^{2+\epsilon}.
$$
For general $r$ we believe it is likewise easy to see that their argument works for general level.   However, we will leave the dependence on $N$ suppressed, and write 
\be\label{MI}
	\sum_{|t_j|\sim T}\frac{(-1)^{\alpha_j} \cL\left(\frac{1}{2}+ir, \overline{u_j}\times E_N (*, 1/2-ir) \right)^2}{\cosh(\pi t_j)}
\ll_N (T^2+|r|^{4/3})^{1+\epsilon}.
\ee
We also require the bound  of Proposition 4.1 \cite{HHR}: 
\be\label{HHR}
\sum_{|t_j| \sim T}|\left<V,u_j\right>|^2e^{\pi |t_j|} \ll_{N_0} (\ell_1 \ell_2)^{-k}T^{2k}\log (T).
\ee
Here there is an important subtlety that we will record for future reference.   In \eqref{HHR} the level is in two pieces: $N = \ell_1\ell_2N_0$, where the functions $f,g$ have level $N_0$, and $V(z) = \overline{f(\ell_1z)}g(\ell_2z)u_j(z)y^k$.   The dependence on $\ell_1, \ell_2$ is recorded, but the dependence on $N_0$ is polynomial, with an undetermined exponent.   We will now suppress the $\ell_1, \ell_2$ from the notation.

(Alternatively, applying Watson's formula, \cite{Wat02} one can also bound a single inner product as follows:
\be\label{HHR3}
|\left<V,u_j\right>|e^{\frac{\pi}{2} |t_j|} \ll N^{-1}|t_j|^{k-1+\epsilon} \sqrt{ |L(1/2, f\otimes g\otimes u_j)|} \ll N^{-5/8+\epsilon}|t_j|^{k+\epsilon}.
\ee
Here we have applied the convexity bound in the $N$ and $t_j$ aspects, using the fact that the level of the triple product is $N^3$.  The factor $N^{-1}$ on the right hand side occurs because of our normalization of the inner product.)

Applying Cauchy-Schwartz, followed by \eqref{MI} and \eqref{HHR}, it follows that
for any  $T' \gg 1$, $N=N_0\ell_1\ell_2$,  
\be\label{jsum1}
	\sum_{|t_j| \sim T'} (-1)^{\alpha_j} \cL \left(1/2+ir, \overline{u_j}\times E_N (*, 1/2-ir) \right) \left<u_j, V\right> \ll_{N}T^{k+\epsilon}
(T'^2+|r|^{4/3})^{\frac12+\epsilon}.
\ee

Returning to our estimate, 
by Stirling's formula, for $|t_j| \sim T'$, 
\be\label{jsum2}
 \frac{\Gamma\left(\frac{1}{2}+ir-w-\frac{k}{2}-it_j\right) 
	\Gamma\left(\frac{1}{2}+ir-w-\frac{k}{2}+it_j\right)}
	{\Gamma\left(\frac{1}{2}-it_j\right) \Gamma\left(\frac{1}{2}+it_j\right)} \ll_{\gamma_2,\tau} T'^{-k-2\sigma_w}.
\ee
Combining \eqref{jsum1} and \eqref{jsum2}, the total exponent of $T'$ is $ 1+ \epsilon-2\sigma_w$, which is negative for $\sigma_w >1/2$.
Thus the sum over $j$ converges absolutely for $\sigma_w > 1/2$.   The sum decays exponentially in $t_j$ in the region where $|t_j| <|\gamma_2-r|$.
Also, via a dyadic subdivision, it is easily seen that most of the sum is concentrated in the region $|t_j |\sim | \gamma_2-r|$.  A  version of \eqref{jsum2} with the dependence on 
$r$ and $\gamma_2$ explicit is
\be\label{jsum3}
 \frac{\Gamma\left(\frac{1}{2}+ir-w-\frac{k}{2}-it_j\right) 
	\Gamma\left(\frac{1}{2}+ir-w-\frac{k}{2}+it_j\right)}
	{\Gamma\left(\frac{1}{2}-it_j\right) \Gamma\left(\frac{1}{2}+it_j\right)}
	 \ll  \left( 1+|\gamma_2-r|+|t_j|\right)^{-k/2-\sigma_w}.
\ee
Notice that only one of the two ratios of gamma functions has contributed to the estimate.
Combining \eqref{jsum1} and \eqref{jsum3} gives us
\begin{multline}\label{sum4}
	\sum_j \frac{\Gamma\left(\frac{1}{2}+ir-w-\frac{k}{2}-it_j\right) 
	\Gamma\left(\frac{1}{2}+ir-w-\frac{k}{2}+it_j\right)}
	{\Gamma\left(\frac{1}{2}-it_j\right) \Gamma\left(\frac{1}{2}+it_j\right)}
	(-1)^{\alpha_j} \cL\left(\frac{1}{2}+ir, \overline{u_j}\times E_N (*, 1/2+ir_1) \right) \left<u_j, V\right>
	\\
	\ll 
	\sum_{|t_j |\sim \gamma_2-r} \left( 1+|\gamma_2-r|+|t_j|\right)^{-k/2-\sigma_w}
(-1)^{\alpha_j} \cL\left(\frac{1}{2}+ir, \overline{u_j}\times E_N (*, 1/2+ir_1) \right) \left<u_j, V\right>   \\
\ll_N (1+|\gamma_2 -r|)^{1 + k/2-\sigma_w + \epsilon} 
.
\end{multline}
The other piece of $I(u)$ is the ratio of gamma functions
$$
\frac{\Gamma\left(-w+u\right)\Gamma\left(w+\frac{k}{2}+ir_1\right) 
	\Gamma\left(w+\frac{k}{2}-ir_1\right)
	\Gamma\left(-ir+w+\frac{k}{2}\right)}
	{\Gamma\left(1+w+u\right)\Gamma\left(ir-w+\frac{k}{2}\right) }.
	$$
By Stirling's formula, for $|\gamma|,|\gamma_2| \gg 1$ this is bounded above by
\begin{multline}\label{sum5}
\frac{(1+| \gamma_2+r_1|)^{(k-1)/2+\sigma_w}(1+|\gamma_2-r_1|)^{(k-1)/2 +\sigma_w} (1+|\gamma-\gamma_2|)^{\sigma_u-\sigma_w-\frac{1}{2}}(1+|\gamma_2-r|)^{2\sigma_w}}
	{(1+|\gamma+\gamma_2|)^{\frac{1}{2}+\sigma_u+\sigma_w}}\\ \times e^{-\frac \pi 2(|\gamma_2+r_1|+|\gamma_2-r_1|+|\gamma -\gamma_2|-|\gamma +\gamma_2|)}.
\end{multline}
The exponent $|\gamma_2+r_1|+|\gamma_2-r_1|+|\gamma -\gamma_2|-|\gamma +\gamma_2|$ equals
\begin{itemize}
\item $2\max(|\gamma_2|,|r_1|)+2|\gamma_2|$ If $|\gamma| \ge |\gamma_2|$, with opposite signs, 

\item  $2\max(|\gamma_2|,|r_1|)-2|\gamma_2|$ If $|\gamma| \ge |\gamma_2|$, with equal signs, 

\item $2\max(|\gamma_2|,|r_1|)+2|\gamma|$  If $|\gamma| \le |\gamma_2|$, with opposite signs, 

\item  $2\max(|\gamma_2|,|r_1|)-2|\gamma|$ If $|\gamma| \le |\gamma_2|$, with equal signs.

\end{itemize}
There is exponential decay in $|\gamma_2|$ for all pairs $(\gamma,\gamma_2)$, except for the region where $\gamma_2, \gamma$ have equal signs and
$|r_1|\le |\gamma_2| \le | \gamma|$.   The integral $I(u)$ is thus bounded above by a multiple of the same integral restricted to this region.   We assume therefore, that 
$\gamma \ge \gamma_2 \ge |r_1|\ge 0$ or $\gamma \le \gamma_2 \le -|r_1| \le 0$, and restrict ourselves to considering the positive case.

Combining this analysis with \eqref{sum4} and \eqref{sum5}, we have shown that
for $\gamma \ge \gamma_2 \ge 0$  the integrand of $I(u)$ is bounded above by 
\begin{multline}\label{firstpiece}
\frac{(1+| \gamma_2+r_1|)^{(k-1)/2+\sigma_w}(1+|\gamma_2-r_1|)^{(k-1)/2 +\sigma_w} (1+|\gamma-\gamma_2|)^{\sigma_u-\sigma_w-\frac{1}{2}}(1+|\gamma_2-r|)^{2\sigma_w}}
	{(1+|\gamma+\gamma_2|)^{\frac{1}{2}+\sigma_u+\sigma_w}} \\ \times 
(1+|\gamma_2 -r|)^{1 + k/2-\sigma_w + \epsilon} (1+ |r|+|r_1|)^AN_0^B(\ell_1\ell_2)^{-k/2} \\
\\
	\ll  \frac{(1+ \gamma_2+|r_1| +|r|)^{3k/2+3\sigma_w+\epsilon} (1+\gamma-\gamma_2)^{\sigma_u-\sigma_w-\frac{1}{2}}}
	{(1+\gamma+\gamma_2)^{\frac{1}{2}+\sigma_u+\sigma_w}}(1+|r_1| +|r|)^AN_0^B(\ell_1\ell_2)^{-k/2} \\
	= \frac{(1+ \gamma_2+|r_1| +|r|)^{3k/2+\epsilon}}{\left((1+\gamma)^2 -\gamma_2^2\right)^{1/2}}
	\left(\frac{(1+\gamma_2+|r_1| +|r|)^3}{(1+\gamma)^2-\gamma_2^2}\right)^{\sigma_w}
	\left(\frac{ 1+\gamma-\gamma_2}{1+\gamma+\gamma_2} \right)^{\sigma_u}(1+ |r_1| +|r|)^AN_0^B(\ell_1\ell_2)^{-k/2}.
\end{multline}
for some $A,B >0$.

The function $h$  that we are using decays exponentially quickly away from the region  
$|\gamma-T| \ll T^{\alpha}$.
  The Stirling formula asymptotics remain valid for 
$\sigma_u/\gamma$ approaching 0, and in particular, if we move the $u$ line of integration to  $\sigma_u = c(N_0,r_1,r)T^{\alpha - \epsilon}$, for some $c \gg 1$.   Poles of $\Gamma(1/2-u)$ are crossed over at points $u = 1/2+\ell$, for $\ell \ge 0$,  but the residues are weighted by the exponentially decaying $h$ and their total contribution is exponentially decaying in $T$.

Consider the region where $\gamma_2 \ge  T^{\beta}$, for some $\beta >0$.   In this region, extracting from \eqref{firstpiece} the part of the upper bound above that is raised to the power $\sigma_u$, we have
$$
\left(\frac{ 1+\gamma-\gamma_2}{1+\gamma+\gamma_2} \right)^{\sigma_u}= \left( \frac{1-\frac{\gamma_2}{1+\gamma}}{1+\frac{\gamma_2}{1+\gamma}}\right)^{\sigma_u}\ll 
\left(1- \frac{T^\beta}{T}\right)^{c(N_0,r_1,r)T^{\alpha-\epsilon}}\ll e^{-c(N_0,r_1,r)T^{\alpha +\beta-1 -\epsilon}}.
$$
This will decay exponentially quickly in $T$ when $\alpha+ \beta > 1+2\epsilon$.
For fixed $\sigma_w$, the remaining factors of \eqref{firstpiece} are bounded above by $\gamma^{ C}$, for fixed $C$ and thus the entire contribution of   \eqref{firstpiece}, and hence of $I(u,h)$ decays exponentially in $T$.  The remaining integral over $u$ contributes, by Stirling, and the definition of $h$, a factor of $T^{2-k}$.
 It follows that for  $\alpha+ \beta > 1+2\epsilon$ the contribution from $\gamma_2 \ge  T^{\beta}$ decays exponentially in $T$.

Now consider the contribution from $\gamma_2$ such that $\gamma_2 \le  T^\beta$, and take $\beta \le 2/3 - \epsilon$.   Factoring out the part of \eqref{firstpiece} raised to the power $\sigma_w$, we have 
$$
\left(\frac{(1+\gamma_2)^3}{(1+\gamma)^2-\gamma_2^2}\right)^{\sigma_w} \ll 
\frac{T^{3\beta\sigma_w}}{T^{2\sigma_w}} \ll T^{-3 \epsilon \sigma_w}.
$$
Finally, move $\sigma_u$ to  $(c+1)/ \epsilon$, and  $\sigma_w$ to $c/ \epsilon$,  for sufficiently large $c>>c(N_0,k)$. In the shifted integral the product in \eqref{firstpiece} is seen to be less than $T^{-3c}$, for $c \gg 1$.    However, poles of $\Gamma(\frac{1-k}{2}-\tau-w\pm it_j)$ are passed over, contributing on the order of $c/\epsilon$ residues at $\frac{1-k}{2}+ir-w \pm it_j =-\ell$, for $\ell \ge 0$.   The residues at these points have already been labeled as $\tth^{(3)R}_\ell (t; 1/2+ir, \overline{ir_1}, \nu_2)$, and thus the contribution of $E^{(3)}_{f, \phi_2}(r, r_1)$ decays to an arbitrary polynomial degree, with the exception of the sum of residues
\be
	\sum_{0\le \ell <c/\epsilon}
	\sum_j 
	\tth^{(3)R}_{\ell}(t_j; 1/2+ir, \overline{ir_1}, \nu_2)
	\left<u_j, U_{f, \phi_2}\right> 
	(-1)^{\alpha_j}
	\cL\left(1/2+ir, E^{(k)}_N(*, 1/2+ir_1) \times \overline{u_j}\right).
\ee
This contribution will be dealt with in the next section.

\subsection{The estimation of $E^{(3)}_{f, E_N^{(k)}(*, 1/2+\overline{ir_1}), \phi_2)}$ II : Residues}\label{ss:residues}

Let 
$$
	\cR_{f, E^{(k)}_N(*, 1/2+\overline{ir_1}), \phi_2} (1/2 +ir)
	:= 
	\cR_{f, E^{(k)}_N(*, 1/2+\overline{ir_1}), \phi_2,\text{cusp}} (1/2 +ir)+\cR_{f, E^{(k)}_N(*, 1/2+\overline{ir_1}), \phi_2, \text{cont}} (1/2 +ir), 
$$
with 
\begin{multline*}
	\cR_{f, E^{(k)}_N(*, 1/2+\overline{ir_1}), \phi_2,\text{cusp}} (1/2 +ir)
	\\
	:=
	\sum_{\ell=0}^{\frac{k}{2}}
	\sum_j 
	\tth^{(3)R}_{\ell}(t_j; 1/2+ir, \overline{ir_1}, \nu_2)
	\left<u_j, U_{f, \phi_2}\right> 
	(-1)^{\alpha_j}
	\cL\left(1/2+ir, E^{(k)}_N(*, 1/2+ir_1) \times \overline{u_j}\right)
\end{multline*}
and
\begin{multline*}
	\cR_{f, E^{(k)}_N(*, 1/2+\overline{ir_1}), \phi_2,\text{cont}} (1/2 +ir)
	\\
	:=
	\sum_{\ell=0}^{\frac{k}{2}}
	\sum_{\cuspa} \frac{1}{4\pi} \int_{-\infty}^\infty 
	\tth^{(3)R}_{\ell}(t; 1/2+ir, \overline{ir_1}, \nu_2)
	\left<E_{\cuspa}(*, 1/2+it), U_{f, \phi_2}\right> 
	\\
	\times
	\cL\left(1/2+ir, E^{(k)}_N(*, 1/2+ir_1)\times E_{\cuspa}(*, 1/2-it)\right)	
	\; dt
	, 
\end{multline*}
with
\begin{multline*}
	\tth^{(3)R}_\ell (t; 1/2+ir, \overline{ir_1}, \nu_2)
	=
	\frac{1}{\ell!}
	\frac{(4\pi)^{k} i^k}{\pi\sqrt{\pi}}
	\sum_\pm
	\frac{
	\Gamma\left(s+\ell \pm it+ir_1\right) 
	\Gamma\left(s+\ell \pm it-ir_1\right)
	\Gamma\left(-\ell \mp 2it\right)} 
	{\Gamma\left(\frac{k}{2}-\ell \mp it+\bnu_2\right)
	\Gamma\left(\frac{k}{2}-\ell \mp it-\bnu_2\right) }
	\\
	\times
	\frac{1}{2\pi i} \int_{(\sigma_u = 1/2+2\epsilon)} 
	\frac{h(u/i) u 
	\Gamma\left(\frac{1}{2}+u\right) \Gamma\left(\frac{1}{2}-u\right)}
	{\Gamma\left(\frac{k}{2}+ir_1 +u\right) \Gamma\left(\frac{k}{2}+ir_1-u\right)}
	\frac{\Gamma\left(-\frac{1}{2}-ir+\frac{k}{2}-\ell \mp it+u\right)}
	{\Gamma\left(\frac{3}{2}+ir-\frac{k}{2}+\ell \pm it+u\right)}
	\; du.
\end{multline*}
as given in \eqref{e:h3Rell}, at $s=1/2+ir$. 
We will also include the extensions of these sums to $k/2 <\ell <c/\epsilon$, as mentioned at the end of the previous section.   Thus the objective of this section is to bound 
\begin{multline*}
	{\cR '}_{f, E^{(k)}_N(*, 1/2+\overline{ir_1}), \phi_2,\text{cont}} (1/2 +ir)
	\\
	=
	\sum_{\ell <c/\epsilon}
	\frac{(4\pi)^{k} i^k }{\pi^{\frac{3}{2}}} 
	\sum_{\cuspa} \frac{1}{4\pi} \int_{-\infty}^\infty 
	\tth^{(3)R}_{\ell}(t; 1/2+ir, \overline{ir_1}, \nu_2)
	\left<E_{\cuspa}(*, 1/2+it), U_{f, \phi_2}\right> 
	\\
	\times
	\cL\left(1/2+ir, E^{(k)}_N(*, 1/2+ir_1)\times E_{\cuspa}(*, 1/2-it)\right)	
	\; dt
	.
\end{multline*}

By Stirling, and \eqref{jsum1}, the exponential contribution to the cuspidal part of the term 
$\cR_{f, E^{(k)}_N(*, 1/2+\overline{ir_1}), \phi_2,\text{cusp}} (1/2+ir)$ is
\begin{multline*}
\exp{\left(-\frac\pi2\left(|r-t_j+r_1| +|r-t_j-r_1|+2|\gamma| +|\gamma +t_j -r|-|r_1+\gamma| -|r_1-\gamma|-|\gamma-t_j+r|\right)\right)}\\
= \exp{\left(-\frac\pi2\left(2 \max (|r-t_j|, |r_1|) +2|\gamma|  -2 \max(|\gamma|, |r_1|) +|\gamma +t_j -r|-|\gamma-t_j+r|\right)\right)}\
\end{multline*}
The key thing is to first verify that this exponent is always non-positive, and second, to determine the regions where the exponential decay vanishes.    By symmetry we will restrict ourselves to the case $\gamma \ge 0$.  We will also restrict ourselves to the case $|r|,|r_1| \le \gamma$.  An examination of various cases shows that for $\gamma \ge 0$, the exponent will be 0 precisely when 
$\gamma \ge |t_j -r| \ge |r_1|$, with $t_j \le 0$, $r\ge 0$.

Recall  that $h$ concentrates $u = \sigma_u + i \gamma$ near $|\gamma - T| \ll T^\alpha$.   As noted previously, we may move $\sigma_u$ anywhere satisfying $1+\epsilon \le \sigma_u \le T^{\alpha - \epsilon}$ without altering this property of $h$.  It follows that for $\sigma_u$ in this range, the sum over $j$ converges absolutely and
\begin{multline}\label{gammaprod}
{\cR}'_{f, E^{(k)}_N(*, 1/2+\overline{ir_1}), \phi_2,\text{cusp}} (1/2 +ir) \\ \ll \sum_{\ell <c/\epsilon} T^{1+\alpha}\sum_{|t_j |\le T -|r|}
\frac{\Gamma\left( 2it_j -\ell\right) 
	\Gamma\left(\frac{1}{2}+ir- it_j + \ell +ir_1\right) 
	\Gamma\left(\frac{1}{2}+ir- it_j + \ell  -ir_1\right) }
	{\Gamma\left(\frac{k}{2}+ it_j-\ell+\bnu_2\right) 
	\Gamma\left(\frac{k}{2}+ it_j-\ell-\bnu_2\right)}\\ \times
	\frac{
	\Gamma\left(\frac{1}{2}+u\right) \Gamma\left(\frac{1}{2}-u\right)\Gamma\left(-\frac{1}{2}-ir+\frac{k}{2} + it_j -\ell+u\right)}
	{\Gamma\left(\frac{k}{2}+ir_1 +u\right) 
	\Gamma\left(\frac{k}{2}+ir_1-u\right)\Gamma\left(\frac{3}{2}+ir -\frac{k}{2} - it_j +\ell+u\right)}\\ \times
	\left<u_j, U_{f, \phi_2}\right> 
	(-1)^{\alpha_j}
	\cL\left(1/2+ir, E^{(k)}_N(*, 1/2+ir_1) \times \overline{u_j}\right).
\end{multline}
Here $u = \sigma_u +iT$. 

Applying Stirling's formula to two of the gamma factors, we find that, as $r>0, t_j<0$,
$$
\frac{\Gamma\left(-\frac{1}{2}-ir+\frac{k}{2} + it_j -\ell+u\right)}{\Gamma\left(\frac{3}{2}+ir -\frac{k}{2} - it_j +\ell+u\right)} \ll \frac{\left( T-|t_j|-|r|\right)^{k/2-1 -\ell + \sigma_u}}{\left(T + |t_j| + |r|\right)^{1-k/2 +\ell + \sigma_u}}.
$$

This contains the factor
$$
\left(\frac{T -|t_j|-|r|}{T +|t_j|+|r|}   \right)^{\sigma_u}
$$
Suppose first that $|r| + |t_j| > T^{1-\alpha+\epsilon}$.
Then 
$$
\left(\frac{T -|r|-|t_j|}{T +|r|+|t_j|}   \right)^{\sigma_u} \ll \left( 1 - \frac{1}{T^{\alpha- \epsilon}} \right)^{\sigma_u}.
$$
Choosing $\sigma_u> C T^{\alpha-\epsilon} \log{(T N(1+|r|))}$ gives us
$$
 \left( 1 - \frac{1}{T^{\alpha-\epsilon}} \right)^{\sigma_u} \ll(T N(1+|r|))^{-C},
$$
which will dominate the remainder of the integral for $C$ sufficiently large.

We have now reduced ourselves to the case $|t_j |+ |r| \le T^{1-\alpha+ \epsilon}$. The sum over $|t_j| \le T^{1-\alpha+\epsilon}-|r|$ can be broken up into a dyadic  sum and is bounded above by an estimate for the sum over $|t_j|+|r| \sim T^{1-\alpha+\epsilon}$.
Thus, including the estimates above, and bounding the remaining gamma factors of \eqref{gammaprod} by Stirling, and referring to \eqref{jsum1}, with $T' = T^{1-\alpha+\epsilon}$, we find that 
$$
{\cR}'_{f, E^{(k)}_N(*, 1/2+\overline{ir_1}), \phi_2,\text{cusp}} (1/2 +ir)  \ll_N  \sum_{\ell <c/\epsilon}T^{\alpha-2\ell+(1-\alpha + \epsilon)(3/2 + 3\ell +\epsilon)}.
$$
The right hand side makes it clear that we must choose $\alpha \ge 1/3$, as otherwise the exponent will grow with $\ell$.
If $\alpha = 1/3$, this becomes
$$
{\cR}'_{f, E^{(k)}_N(*, 1/2+\overline{ir_1}), \phi_2,\text{cusp}} (1/2 +ir)\ll_{r,r_1,N} T^{4/3 + \epsilon'},
$$
while choosing $\alpha = 1$ gives us
$$
{\cR}'_{f, E^{(k)}_N(*, 1/2+\overline{ir_1}), \phi_2,\text{cusp}} (1/2 +ir) \ll_{r,r_1,N} T^{1+ \epsilon'}.
$$
In general, for $1\ge \alpha \ge 1/3$, 
$$
{\cR}'_{f, E^{(k)}_N(*, 1/2+\overline{ir_1}), \phi_2,\text{cusp}} (1/2 +ir) \ll_{N} 
T^{3/2-\alpha/2 + \epsilon'}.
$$

If $|r| \ge T^{1-\alpha+ \epsilon}$ then the case  $|t_j |+ |r| \le T^{1-\alpha+ \epsilon}$ never occurs, and 
$$
{\cR}'_{f, E^{(k)}_N(*, 1/2+\overline{ir_1}), \phi_2,\text{cusp}} (1/2 +ir) \ll_N T (1+|r|))^{-C}.
$$
If $|r| \le T^{1-\alpha+ \epsilon}$ then for $\alpha = 1/3$ we obtain
$$
{\cR}'_{f, E^{(k)}_N(*, 1/2+\overline{ir_1}), \phi_2,\text{cusp}} (1/2 +ir)\ll_{\N}T^{4/3 + \epsilon'}.
$$
The contribution from $\cR_{f, E^{(k)}_N(*, 1/2+\overline{ir_1}), \phi_2,\text{cont}} (1/2 +ir)$ is bounded in a similar way, and contributes a smaller amount.
\subsection{The estimation of $E^{(1)}_{f, E^{(k)}_N(*, 1/2+\overline{ir_1}), gy^{k/2}}(s)$}
Recall that
\begin{multline}\label{e:E1_EkN_g}
	E^{(1)}_{f, E^{(k)}_N(*, 1/2+\overline{ir_1}), gy^{k/2}}(s)
	\\
	=
	\left\{
	\sum_j 
	\frac{\Gamma\left(\frac{k}{2}+\frac{1}{2}+ir_1\right)}
	{\Gamma\left(-\frac{k}{2}+\frac{1}{2}+ir_1\right)}
	\tth^{(1)}(t; s, \overline{ir_1}, k/2-1/2)
	\left<u_j, U_{f, gy^{k/2}}\right> 
	\cL\left(s, \overline{u_j} \times E^{(k)}_N(*, 1/2+ir_1)\right)
	\right.
	\\
	+
	\sum_{a\mid N} \frac{1}{4\pi} \int_{-\infty}^\infty 
	\frac{\Gamma\left(\frac{k}{2}+\frac{1}{2}+ir_1\right)}
	{\Gamma\left(-\frac{k}{2}+\frac{1}{2}+ir_1\right)}
	\tth^{(1)}(t; s, \overline{ir_1}, k/2-1/2)
	\\
	\left.
	\left<E_{1/a}\left(*, 1/2+it\right), U_{f, gy^{k/2}}\right>
	\cL\left(s, E_{1/a}(*, 1/2-it) \times E^{(k)}_N(*, 1/2+ir_1)\right)\; dt
	\right\}
	, 
\end{multline}
where
\begin{multline}\label{e:tth1_EkN_g}
	\frac{\Gamma\left(\frac{k}{2}+\frac{1}{2}+ir_1\right)}
	{\Gamma\left(-\frac{k}{2}+\frac{1}{2}+ir_1\right)}
	\tth^{(1)}(t; s, \overline{ir_1}, k/2-1/2)
	\\
	=
	-
	\frac{(4\pi)^{k} \cos(\pi ir_1) }{\pi \sqrt{\pi}} 
	\frac{1}{2\pi i} \int_{(\sigma_u = 1/2+2\epsilon)} 
	\frac{h(u/i) u \tan(\pi u)}{\Gamma\left(\frac{k}{2}+ir_1 +u\right) 
	\Gamma\left(\frac{k}{2}+ir_1-u\right)}
	\\
	\times
	\frac{1}{2\pi i } \int_{(\sigma_w = 1/2-\epsilon)}
	\frac{\Gamma\left(-w+\frac{k}{2}+u\right) 
	\Gamma\left(-w+\frac{k}{2}-u\right)
	\Gamma\left(w-ir_1\right) \Gamma\left(w+ir_1\right)}
	{
	\Gamma\left(s-w+k-\frac{1}{2}\right) 
	\Gamma\left(s-w+\frac{1}{2}\right)}
	\\
	\times
	\Gamma\left(s-w+it\right) \Gamma\left(s-w-it\right)
	\; dw\; du
	.
\end{multline}
Here we take $s = 1/2 + ir$, and $r_1 = \pm r$.  The exponential decay arising from the gamma factors and 
$\cos(\pi i r_1)$ is
$$
\exp{\left(-\frac{\pi}{2}\left( -2|r_1| + 2\max(|\gamma_2|,|r_1|)-2|r-\gamma_2| +2\max(|r-\gamma_2|,|t_j|) \right)\right)}.
$$
There is exponential decay except if $|\gamma_2| \le T$,  $|r-\gamma_2|\ge|t_j|$   and $|\gamma_2|\le |r_1|$.   
For each $j$, the product of $\cos(\pi i r_1)$ times the gamma function ratio is bounded above as follows:
\begin{multline*}
	\frac{\cos(\pi i r_1)\Gamma\left(-w+\frac{k}{2}+u\right) 
	\Gamma\left(-w+\frac{k}{2}-u\right)
	\Gamma\left(w-ir_1\right) \Gamma\left(w+ir_1\right)
}
	{\Gamma\left(\frac{k}{2}+ir_1 +u\right) 
	\Gamma\left(\frac{k}{2}+ir_1-u\right)
	\Gamma\left(ir-w+k\right) 
	\Gamma\left(1+ir-w\right)}\\ \times \Gamma\left(\frac12 +ir-w+it_j\right) \Gamma\left(\frac12 +ir-w-it_j\right)\\
	\ll 
	T^{-2\sigma_w}\frac{\left(|\gamma_2|-|r_1|\right)^{\sigma_w -1/2}\left(|\gamma_2|+|r_1|\right)^{\sigma_w -1/2}}{\left(|\gamma_2-r|\right)^{-2\sigma_w +k}}\Gamma\left(\frac12 +ir-w+it_j\right) \Gamma\left(\frac12 +ir-w-it_j\right) \\
\ll T^{-2\sigma_w}\ |r|^{4 \sigma_w-1 -k}\Gamma\left(\frac12 +ir-w+it_j\right) \Gamma\left(\frac12 +ir-w-it_j\right).
		\end{multline*}

The estimate improves if $w$ is shifted to $\sigma_w = 1/2 +\ell'+\epsilon$, with $\ell'<k/2$.    Simple poles are passed over at $w = 1/2 +i(r+t_j) +\ell$, with $\ell \le \ell'$.   The sum of the residues at these poles, after the $u$ integration is performed, give the largest contribution to $E^{(1)}_{f, E^{(k)}_N(*, 1/2+\overline{ir_1}), gy^{k/2}}(1/2 + ir)$.   As
$$
\Res_{w = 1/2 +i(r+t_j) +\ell}\Gamma\left(\frac12 +ir-w+it_j\right) \Gamma\left(\frac12 +ir-w-it_j\right)
\ll \left(1+|t_j|\right)^{-1/2-\ell},
$$
For each $\ell$ this contribution is bounded above by
\begin{multline*}
T^{ \alpha -2\ell}|r|^{1/2 + 3\ell-k}\\ \times \sum_{|t_j| \ll |r|}
\left<E_{1/a}\left(*, 1/2+it\right), U_{f, gy^{k/2}}\right>
	\cL\left(s, E_{1/a}(*, 1/2-it) \times E^{(k)}_N(*, 1/2+ir_1)\right).
\end{multline*}
Referring to \eqref{jsum1}, with $T' = |r|$, this is bounded above by
\be\label{E0}
T^{ \alpha -2\ell}|r|^{1/2 + 3\ell-k+k+1 +\epsilon}= T^\alpha |r|^{3/2 + \epsilon}\left(\frac{|r|^3}{T^2}\right)^\ell.
\ee
For $|r|< T^{2/3}$ this will be bounded above by our main term, which is on the order of $T^{1 + \alpha}$.

In the shifted integral, we have
$$
\Gamma\left(\frac12 +ir-w+it_j\right) \Gamma\left(\frac12 +ir-w-it_j\right)
\ll \left(1+|t_j|\right)^{-2\sigma_w},
$$
Integrating over $u$, and $w$ with  $|\gamma_2| \le |r|$, and app;lying  \eqref{jsum1} again, the shifted integral is bounded above by
$$
T^{1+\alpha}|r|^{1+\epsilon}\left(\frac{|r|}{T}\right)^{2\sigma_w},
$$
Taking $|r|< T^{2/3}$, this is bounded above by \eqref{E0} as long as $\sigma_w >3$.
This concludes the proof of Proposition~\ref{prop:bounds}.

\thispagestyle{empty}
{\footnotesize
\nocite{*}
\bibliographystyle{amsalpha}
\bibliography{reference_SDDS}

\providecommand{\bysame}{\leavevmode\hbox to3em{\hrulefill}\thinspace}
\providecommand{\MR}{\relax\ifhmode\unskip\space\fi MR }
\providecommand{\MRhref}[2]{%
  \href{http://www.ams.org/mathscinet-getitem?mr=#1}{#2}
}
\providecommand{\href}[2]{#2}
\begin{thebibliography}{AAR99}

\bibitem[AAR99]{AAR}
George~E. Andrews, Richard Askey, and Ranjan Roy, \emph{Special functions},
  Encyclopedia of Mathematics and its Applications, vol.~71, Cambridge
  University Press, Cambridge, 1999.

\bibitem[Asa76]{Asa76}
Tetsuya Asai, \emph{On the {F}ourier coefficients of automorphic forms at
  various cusps and some applications to {R}ankin's convolution}, J. Math. Soc.
  Japan \textbf{28} (1976), no.~1, 48--61. \MR{0427235 (55 \#270)}

\bibitem[Blo04]{B}
Valentin Blomer, \emph{Shifted convolution sums and subconvexity bounds for
  automorphic {$L$}-functions}, Int. Math. Res. Not. (2004), no.~73,
  3905--3926.

\bibitem[CD05]{CD05}
Gautam Chinta and Adrian Diaconu, \emph{Determination of a {${\rm GL}_3$}
  cuspform by twists of central {$L$}-values}, Int. Math. Res. Not. (2005),
  no.~48, 2941--2967.

\bibitem[cit]{cite-key}


\bibitem[DI83]{DI83}
J-M. Deshouillers and H.~Iwaniec, \emph{{Kloosterman sums and Fourier
  coefficients of cusp forms}}, Invent. Math. (1982/83), no.~70, 219--288.

\bibitem[GHS09]{GHS09}
Satadal Ganguly, Jeffrey Hoffstein, and Jyoti Sengupta, \emph{Determining
  modular forms on {${\rm SL}_2(\Bbb Z)$} by central values of convolution
  {$L$}-functions}, Math. Ann. \textbf{345} (2009), no.~4, 843--857.

\bibitem[GR00]{GR}
I.~S. Gradshteyn and I.~M. Ryzhik, \emph{Table of integrals, series, and
  products}, sixth ed., Academic Press Inc., San Diego, CA, 2000, Translated
  from the Russian, Translation edited and with a preface by Alan Jeffrey and
  Daniel Zwillinger.

\bibitem[HHR]{HHR}
Jeff Hoffstein, Tom Hulse, and Andre Reznikov, \emph{{Multiple Dirichlet Series
  and Shifted Convolution}}, arXiv:1110.4868v2.

\bibitem[HM06]{HM06}
Gergely Harcos and Philippe Michel, \emph{{The sub convexity problem for
  Rankin-Selberg $L$-functions and equidistribution of Heegner points. II}},
  Invent. Math. \textbf{163} (2006), no.~3, 581--655.

\bibitem[Hul13]{Hul13}
Thomas~A. Hulse, \emph{{Triple Shifted Sums of Automorphic $L$-functions}},
  thesis (2013).

\bibitem[Ivi02]{Ivi02}
Aleksandar Ivi{{\'c}}, \emph{On the moments of {H}ecke series at central
  points}, Funct. Approx. Comment. Math. \textbf{30} (2002), 49--82.
  \MR{2136511 (2006g:11094)}

\bibitem[Iwa02]{Iwa02}
H.~Iwaniec, \emph{{Spectral Methods of Automorphic Forms (2nd ed.)}}, AMS,
  Providence, 2002.

\bibitem[JM05]{JM05}
Matti Jutila and Yoichi Motohashi, \emph{Uniform bound for {H}ecke
  {$L$}-functions}, Acta Math. \textbf{195} (2005), 61--115. \MR{2233686
  (2007d:11051)}

\bibitem[LLY06]{LLY06}
Yuk-Kam Lau, Jianya Liu, and Yangbo Ye, \emph{A new bound {$k^{2/3+\epsilon}$}
  for {R}ankin-{S}elberg {$L$}-functions for {H}ecke congruence subgroups},
  IMRP Int. Math. Res. Pap. (2006), Art. ID 35090, 78. \MR{2235495
  (2007c:11062)}

\bibitem[LR97]{LR97}
Wenzhi Luo and Dinakar Ramakrishnan, \emph{Determination of modular forms by
  twists of critical {$L$}-values}, Invent. Math. \textbf{130} (1997), no.~2,
  371--398.

\bibitem[Luo99]{Luo99}
Wenzhi Luo, \emph{Special {$L$}-values of {R}ankin-{S}elberg convolutions},
  Math. Ann. \textbf{314} (1999), no.~3, 591--600.

\bibitem[Mot92]{Mot92}
Y{\=o}ichi Motohashi, \emph{Spectral mean values of {M}aass waveform
  {$L$}-functions}, J. Number Theory \textbf{42} (1992), no.~3, 258--284.

\bibitem[MS12]{MS12}
Ritabrata Munshi and Jyoti Sengupta, \emph{On effective determination of maass
  forms from central values of rankin-selberg $l$-function}.

\bibitem[Sar01]{Sar01}
Peter Sarnak, \emph{Estimates for {R}ankin-{S}elberg {$L$}-functions and
  quantum unique ergodicity}, J. Funct. Anal. \textbf{184} (2001), no.~2,
  419--453. \MR{1851004 (2003c:11050)}

\bibitem[Sen04]{Sen04}
J.~Sengupta, \emph{Distinguishing {H}ecke eigenvalues of primitive cusp forms},
  Acta Arith. \textbf{114} (2004), no.~1, 23--34.

\bibitem[Wat02]{Wat02}
Thomas~Crawford Watson, \emph{Rankin triple products and quantum chaos},
  ProQuest LLC, Ann Arbor, MI, 2002, Thesis (Ph.D.)--Princeton University.
  \MR{2703041}

\bibitem[Zha11]{Zha11}
Yichao Zhang, \emph{Determining modular forms of general level by central
  values of convolution {$L$}-functions}, Acta Arith. \textbf{150} (2011),
  no.~1, 93--103. \MR{2825575}

\end{thebibliography}
}

\end{document}